\documentclass[]{scrartcl}

\usepackage[utf8]{luainputenc}
\usepackage[USenglish]{babel}
\usepackage{csquotes}

\usepackage[a4paper,top=27mm,bottom=20mm,inner=25mm,outer=20mm]{geometry}

\usepackage[%
  backend=bibtex,bibencoding=ascii,
  style=numeric-comp,
  giveninits=true, uniquename=init, 
  natbib=true,
  url=true,
  doi=true,
  isbn=false,
  backref=false,
  maxnames=99,
  ]{biblatex}
\usepackage{amsmath}
\allowdisplaybreaks
\numberwithin{equation}{section}
\usepackage{amssymb}
\usepackage{commath}
\usepackage{mathtools}
\usepackage{bbm}
\usepackage{nicefrac}
\usepackage{subdepth}

\usepackage{siunitx}
\usepackage{algorithm}
\usepackage{algpseudocode}

\usepackage{amsthm}

\usepackage[plainpages=false,pdfpagelabels,hidelinks,unicode]{hyperref}

\usepackage{cleveref}

\usepackage{thmtools}
\usepackage{etoolbox}
\makeatletter
\patchcmd{\thmt@setheadstyle}
 {\bgroup\thmt@space}
 {\thmt@space}
 {}{}
\patchcmd{\thmt@setheadstyle}
 {\egroup\fi}
 {\fi}
 {}{}
\makeatother
\declaretheoremstyle[
  bodyfont=\normalfont\itshape,
  headformat=\NAME\ \NUMBER\NOTE,
]{myplain}
\declaretheoremstyle[
  headformat=\NAME\ \NUMBER\NOTE,
]{mydefinition}
\newcommand{\envqed}{{\lower-0.3ex\hbox{$\triangleleft$}}}
\declaretheorem[style=myplain,numberwithin=section]{theorem}
\declaretheorem[style=myplain,numberlike=theorem]{lemma}
\declaretheorem[style=myplain,numberlike=theorem]{proposition}

\declaretheorem[style=myplain,numberlike=theorem]{corollary}
\declaretheorem[style=mydefinition,numberlike=theorem,qed=\envqed]{definition}
\declaretheorem[style=mydefinition,numberlike=theorem,qed=\envqed]{remark}

\declaretheorem[style=mydefinition,qed=\envqed]{assumption}

\newenvironment{assumptionstar}
  {\addtocounter{assumption}{-1}%
   \assumption}
  {\endassumption}


\usepackage{color}
\usepackage{graphicx}
\usepackage[small]{caption}
\usepackage{subcaption}
\usepackage{tikz} 
\usepackage[section]{placeins} 

\begingroup\expandafter\expandafter\expandafter\endgroup
\expandafter\ifx\csname pdfsuppresswarningpagegroup\endcsname\relax
\else
\fi

\usepackage{booktabs}
\usepackage{rotating}
\usepackage{multirow}

\usepackage{enumitem}

\usepackage{ifluatex}
\ifluatex
  \usepackage[no-math]{fontspec}
\else
  \usepackage[T1]{fontenc}
\fi
\usepackage{newpxtext,newpxmath}

\let\epsilon\varepsilon
\let\phi\varphi
\let\rho\varrho

\renewcommand{\O}{\mathcal{O}}

\newcommand{\N}{\mathbb{N}}
\providecommand\e{}
\renewcommand{\e}{\mathrm{e}}
\providecommand\Z{}
\renewcommand{\Z}{\mathbb{Z}}
\providecommand\R{}
\renewcommand{\R}{\mathbb{R}}

\newcommand{\I}{\operatorname{I}}

\renewcommand{\vec}[1]{\pmb{#1}}

\newcommand{\solbase}{u}                 
\newcommand{\usol}{\solbase}             
\newcommand{\unum}{\solbase^{h}}         
\newcommand{\vsol}{\vec{\solbase}}       
\newcommand{\vnum}{\vec{\solbase}^{h}}   
\newcommand{\wbase}{w}                    
\newcommand{\wsol}{\wbase}                
\newcommand{\wnum}{\wbase^{h}}            
\newcommand{\diag}{\operatorname{diag}}
\newcommand{\amdisc}[2]{\vec{J a}^{#1,h}_{#2}}     
\newcommand{\amdisci}[3]{J a^{#1,h}_{#2,#3}}       
\newcommand{\Jdisc}{\vec{J}^{h}}                 
\newcommand{\Jdiscn}{J^{h}}                        
\newcommand{\fnum}{f^{\mathrm{num}}}
\newcommand{\mean}[1]{\{\!\{#1\}\!\}}
\newcommand{\logmean}[1]{\{\!\{#1\}\!\}_{\log}}
\newcommand{\prodmean}[2]{(\!(#1 \cdot #2)\!)}
\newcommand{\distper}{d_{\mathrm{per}}}

\newcommand{\orcid}[1]{ORCID:~\href{https://orcid.org/#1}{#1}}
\usepackage{authblk}

\newenvironment{keywords}{\par\textbf{Key words.}}{\par}
\newenvironment{AMS}{\par\textbf{AMS subject classification.}}{\par}

\title{Convergence of entropy-conservative summation-by-parts discretizations to smooth solutions of hyperbolic conservation laws}

\author[1]{Hendrik~Ranocha\thanks{\orcid{0000-0002-3456-2277}}}
\affil[1]{Institute of Mathematics, Johannes Gutenberg University Mainz, Staudingerweg 9, 55128 Mainz, Germany}

\date{November 24, 2026}

\makeatletter
\hypersetup{pdfauthor={Hendrik Ranocha}}
\hypersetup{pdftitle={Convergence of entropy-conservative summation-by-parts discretizations to smooth solutions of hyperbolic conservation laws}}
\makeatother

\begin{document}

\maketitle

\begin{abstract}
\noindent
  Although entropy-based summation-by-parts (SBP) discretizations of hyperbolic conservation laws are widely used for their robustness and stability properties, there are very few results on their convergence.
We extend a recent convergence analysis of Worku, Del Rey Fern\'andez, and Zingg (2026, \url{https://doi.org/10.48550/arXiv.2603.18369}) in two ways.
First, instead of allowing only hyperbolic conservation laws whose fluxes are homogeneous and have globally bounded second derivatives (a restriction essentially to linear or quadratic fluxes), we consider general hyperbolic systems with strictly convex entropy and source terms depending on time and space.
Second, instead of requiring a special class of SBP operators, we consider a general framework of diagonal-norm SBP operators on curved meshes, including finite differences, continuous and discontinuous Galerkin methods.
Since the error analysis is based on a discrete relative entropy, it is restricted to smooth solutions.
To enable a unified treatment of conservation laws, we restrict the analysis to periodic boundary conditions.
Numerical results demonstrate that the predicted convergence rates are sharp in general, but can be improved for special cases such as discontinuous Galerkin methods with even polynomial degree and multi-block finite difference methods.
An optimal analysis is expected to require more sophisticated arguments specialized to the class of methods instead of the general framework of SBP operators used in this work.

\end{abstract}

\begin{keywords}
  summation-by-parts operators,
  entropy-conservative methods,
  relative entropy,
  hyperbolic conservation laws,
  convergence analysis
\end{keywords}

\begin{AMS}
  65M12, 
  65M06, 
  65M60, 
  65M70, 
  65M20, 
  35L65 
\end{AMS}

\section{Introduction}
\label{sec:introduction}

Entropy-conservative and entropy-stable discretizations of hyperbolic conservation laws have become a central tool in the design of robust high-order methods \cite{chen2020review,gassner2021novel}.
Building on Tadmor's entropy-conservative two-point fluxes \cite{tadmor1987numerical,tadmor2003entropy} and their high-order extensions \cite{lefloch2002fully,fisher2013high}, flux differencing schemes based on summation-by-parts (SBP) operators transfer the entropy structure of the continuous problem to finite difference, discontinuous Galerkin, and related discretizations \cite{gassner2013skew,carpenter2014entropy,chen2017entropy,ranocha2018comparison}, including curvilinear meshes \cite{crean2018entropy}.
The resulting methods satisfy a semidiscrete entropy equality or inequality by construction and show strongly improved robustness for nonlinear problems \cite{gassner2016split,winters2018comparative,rojas2021robustness}.
However, there have recently been some discussions about local linear/energy stability issues of these methods \cite{gassner2022stability,ranocha2021preventing}.
While \cite{gassner2022stability} discussed why these issues do not contradict convergence, such results pointed out a lack of convergence theory for these methods.
Recently, Worku, Del Rey Fern{\'a}ndez, and Zingg \cite{worku2026convergence} started to fill this gap by proving convergence of entropy-conservative split-form discretizations of symmetric hyperbolic systems with homogeneous flux functions with globally bounded second derivatives in the multi-dimensional continuous-SBP framework.
This article extends their work to more general systems including the compressible Euler equations, allows source terms depending on space and time, and considers a more general class of SBP operators.
The technique based on the relative entropy of Dafermos \cite{Dafermos1979} and DiPerna \cite{Diperna1979} (see also \cite[Chapter~5]{dafermos2016hyperbolic}) is in the spirit of the celebrated Lax equivalence theorem: (entropy) stability and consistency imply convergence.

Relative-entropy methods are established tools in the numerical analysis of conservation laws.
For example, Arvanitis, Makridakis, and Tzavaras \cite{arvanitis2004stability} used the relative entropy for finite element schemes based on relaxation models, obtaining an $\O(\epsilon)$ rate for the relaxation approximation of smooth solutions.
Jovanovi{\'c} and Rohde \cite{jovanovic2006error} proved the a~priori error estimate $\O(\sqrt{h})$ for first-order finite volume schemes approximating classical solutions of multi-dimensional systems of hyperbolic balance laws.
Canc{\`e}s, Mathis, and Seguin \cite{cances2016error} obtained the rate $\O(h^{1/4})$ for time-explicit finite volume schemes with entropy-dissipative fluxes under a strengthened CFL condition.
For the Godunov method applied to the multidimensional compressible Euler equations, Luk{\'a}\v{c}ov{\'a}-Medvidov{\'a}, She, and Yuan \cite{lukacova2022godunov} obtained an $L^2$ rate of $\O(h^{1/4})$, improved to $\O(h^{1/2})$ under a uniform total-variation bound.
More recently, Anandan, Arun, Krishnamurthy, and Luk{\'a}\v{c}ov{\'a}-Medvidov{\'a} \cite{anandan2026error} used relative energy to analyze an asymptotic-preserving, energy-stable finite volume method for the barotropic Euler equations.
On the a~posteriori side, Giesselmann, Makridakis, and Pryer \cite{giesselmann2015aposteriori} combined reconstructions of discontinuous Galerkin solutions with the relative-entropy framework, with extensions to fully discrete schemes \cite{dedner2016aposteriori}; see \cite{giesselmann2016relative} for an overview.
A genuinely discrete relative-energy method was developed by Gallou{\"e}t, Herbin, Maltese, and Novotn{\'y} \cite{gallouet2016error} for a mixed finite-element/finite-volume discretization of the compressible barotropic Navier-Stokes equations, with fractional convergence rates.
At the continuous level, the relative-energy method has recently been used to prove convergence of hyperbolic approximations to higher-order partial differential equations \cite{giesselmann2026convergence}, whose residual-based structure we transfer to the semidiscrete setting.
Compared to these works, the present result appears to be, to the best of our knowledge, the first a~priori error estimate of order $h^p$ for entropy-conservative high-order SBP semidiscretizations of nonlinear symmetrizable systems obtained by a discrete relative-entropy argument.

\subsection{Scope and restrictions of this article}

SBP operators \cite{svard2014review,fernandez2014review} provide a general framework for constructing structure-preserving discretizations.
While designed originally for finite differences (FDs) \cite{kreiss1974finite,strand1994summation,carpenter1994time}, they can be used to formulate finite volume methods \cite{nordstrom2001finite}, discontinuous Galerkin (DG) methods \cite{gassner2013skew,carpenter2014entropy}, continuous Galerkin (CG) methods \cite{hicken2016multidimensional,hicken2020entropy,abgrall2020analysisI}, flux reconstruction \cite{ranocha2016summation}, active flux methods \cite{barsukow2026stability}, meshless schemes \cite{hicken2025constructing,kwan2025robust}, and cut-cell methods \cite{taylor2026entropy,petri2026kinetic}.
Although the unifying framework of SBP operators has been used widely to construct large classes of structure-preserving methods, convergence theory has been limited to individual classes of methods.
For example, convergence rates of energy-stable FD SBP schemes for linear problems were established by Sv{\"a}rd and Nordstr{\"o}m \cite{svard2019convergence,svard2020convergence}.
For nonlinear problems, the classical error estimates for Runge--Kutta DG methods and smooth solutions by Zhang and Shu \cite{zhang2004error,zhang2006error,zhang2010stability} and their extension to schemes with quadrature \cite{huang2016error} rely on projections onto a polynomial basis and on upwind-type numerical fluxes;
they do not apply to general entropy-conservative SBP flux differencing schemes.
The recent work of Worku, Del Rey Fern{\'a}ndez, and Zingg \cite{worku2026convergence}, which motivated this work, is restricted to CG-type SBP operators.

In this work, we consider a deliberately abstract framework of diagonal-norm SBP operators, including FD, DG, and CG methods in multiple dimensions on curved meshes.
Thus, we obtain very general convergence results, but cannot use additional structures of specific classes of methods to improve the estimates.
Hence, the resulting convergence rates are sharp in general (e.g., for regular periodic FD methods), but may be suboptimal for specific classes of methods (e.g., for DG methods with even-degree polynomials or multi-block FD-SBP methods increasing the resolution by adding more blocks), as discussed in Section~\ref{sec:numerics}.

Next, the analysis considers smooth solutions of quite general hyperbolic conservation laws with strictly convex entropy and source terms depending on space and time.
This covers, in particular, the method of manufactured solutions \cite{salari2000code}.
To enable a unified formulation of SBP discretizations for such problems, we consider periodic boundary conditions.
While problems on bounded domains are arguably more interesting in practice, imposing boundary conditions in an entropy-stable way requires a case-by-case study of the specific system and boundary conditions \cite{svard2021entropy,svard2025entropy,winters2026numerical,winters2026aligning}, which is beyond the scope of this work.

Furthermore, the analysis in this work is restricted to entropy-conservative semidiscretizations with SBP operators based on collocation.
This simplification excludes generalized/hybridized SBP operators that would require more techniques, e.g., an entropy projection \cite{chan2018discretely,parsani2016entropy,chan2022entropy}.
It also excludes other ways to enforce entropy conservation, e.g., correction terms \cite{abgrall2018general,abgrall2022reinterpretation,edoh2024conservative,chan2025artificial} or formulations using the entropy variables and exact integration \cite{hughes1986new,jiang1994cell,hiltebrand2018entropy,badrkhani2025entropy,carnevali2026efficient}.
Moreover, we do not cover entropy-stable schemes, since the variety of approaches to add dissipation \cite{fjordholm2012arbitrarily,fisher2013high,hiltebrand2014entropy,hennemann2021provably,doehring2026volume} to an entropy-conservative baseline scheme is too large to be covered in a unified framework.
This aligns with the focus on smooth solutions.
In the presence of shocks and other discontinuities, entropy-stable methods should of course be used.
There are some related convergence results in the sense of generalized solution concepts such as measure-valued solutions \cite{fjordholm2017construction} or dissipative weak solutions \cite{lukacova2022convergence}.
However, these do not provide explicit high-order rates of convergence, which is a central aspect of the present work.
After the preprint of this work was published, the preprint \cite{chen2026error} appeared, which proves error estimates of order $h^k$ for smooth solutions of semidiscrete entropy-stable discontinuous Galerkin methods with polynomials of degree $k$ by a consistency--stability argument at the nodal level, on simplicial meshes and under a polynomial-reconstruction hypothesis.
In particular, their analysis covers entropy-stable interface fluxes such as the local Lax--Friedrichs flux, i.e., interface dissipation, which is outside the scope of our analysis, whereas our analysis covers general diagonal-norm SBP operators, including finite-difference and continuous Galerkin operators.
Their analysis uses an entropy-weighted energy of the nodal errors.
This energy is the sum of the relative entropies of the numerical solution with respect to the exact solution and vice versa, whereas our analysis is based on the evolution of the discrete relative entropy alone.
Moreover, their theorems assume an a~priori bound of order $h$ on the $L^\infty$ error of a polynomial reconstruction of the numerical solution, which they justify by a continuation argument if $k > 1 + d/2$ \cite[Section~3.2]{chen2026error} in $d$ space dimensions, whereas our bootstrap argument requires only the order condition $p > d/2$, where $p = k$ for the entropy-conservative discontinuous Galerkin methods covered by our analysis.

The paper is organized as follows.
Section~\ref{sec:setting} introduces the continuous problem, the SBP framework, the entropy-conservative semidiscretizations, and the assumptions.
Section~\ref{sec:main-result} states the main convergence theorem.
Section~\ref{sec:proof} develops the proof:
coercivity of the discrete relative entropy, the discrete entropy balance, the truncation error estimate, the exact cancellation of the linear error terms, the source and remainder bounds, the Gronwall and bootstrap arguments, and a discussion of the order condition.
Section~\ref{sec:applications} verifies the assumptions for scalar conservation laws, symmetric systems, the shallow water equations, and the compressible Euler equations.
Section~\ref{sec:numerics} presents numerical experiments.
Section~\ref{sec:summary} summarizes the results and discusses open problems.
Appendix~\ref{app:euler-constants} collects explicit constants for the Euler equations, and Appendix~\ref{app:encapsulated} verifies the operator assumptions for encapsulated curvilinear operators.

\section{Problem setting and assumptions}
\label{sec:setting}

In this section, we collect the continuous problem, the class of semidiscretizations, and the assumptions used in our analysis.
Throughout, we consider the periodic domain $\Omega = \bigl(\R / (L \Z)\bigr)^d$ in $d \ge 1$ space dimensions with period $L > 0$, Cartesian coordinate directions $m \in \{1, \dots, d\}$ with partial derivatives $\partial_m$, and a finite time horizon $T > 0$.

\subsection{Continuous problem and entropy pair}
\label{ssec:continuous}

We study hyperbolic conservation laws with a source term,
\begin{equation}
\label{eq:conservation-law}
  \partial_t \usol(t, x) + \sum_{m=1}^{d} \partial_m f_m\bigl(\usol(t, x)\bigr) = s(t, x),
  \qquad \usol(t, \cdot)\colon \Omega \to \mathcal{A},
\end{equation}
where $\mathcal{A} \subseteq \R^n$ is an open and convex set of admissible states, $f_m \in C^1(\mathcal{A}; \R^n)$ are the directional fluxes, and $s$ is a given source function specified in Assumption~\ref{asm:source} below.
For the compressible Euler equations, $\mathcal{A}$ is the set of states with positive density and pressure; for typical scalar problems such as Burgers' equation, $\mathcal{A} = \R$.

\begin{assumption}[Entropy pair]
\label{asm:entropy-pair}
  The conservation law \eqref{eq:conservation-law} possesses a strictly convex entropy $U \in C^2(\mathcal{A}; \R)$ and entropy fluxes $F_m \in C^1(\mathcal{A}; \R)$ satisfying the compatibility conditions $F_m'(a)^T = U'(a)^T f_m'(a)$ for all $a \in \mathcal{A}$ and all directions $m$.
\end{assumption}

Here and throughout, derivatives of scalar-valued functions such as $U'$, $F_m'$, and $\psi_m'$ are represented as column vectors (gradients), while $f_m'$ denotes the Jacobian matrix of $f_m$.

We write $w(a) = U'(a)$ for the entropy variables and introduce the directional entropy potentials
\begin{equation}
\label{eq:entropy-potential}
  \psi_m(a) = w(a)^T f_m(a) - F_m(a).
\end{equation}
For systems, the existence of an entropy pair in the sense of Assumption~\ref{asm:entropy-pair} is equivalent to the requirement that all $U''(a) f_m'(a)$ are symmetric, i.e., that the entropy Hessian symmetrizes the conservation law;
this circle of ideas goes back to Godunov \cite{godunov1961interesting}, Friedrichs and Lax \cite{friedrichs1971systems}, and Mock \cite{mock1980systems};
see also \cite{harten1983symmetric,tadmor2003entropy} and \cite[Section~3.2]{dafermos2016hyperbolic}.
Differentiating \eqref{eq:entropy-potential} and using the compatibility condition together with the symmetry of $U''$ yields the identities
\begin{equation}
\label{eq:potential-gradient}
  \psi_m'(a) = U''(a) f_m(a),
\end{equation}
which are convenient for computing entropy potentials of concrete systems in Section~\ref{sec:applications}.

\begin{assumption}[Smooth solution]
\label{asm:smooth-solution}
  The balance law \eqref{eq:conservation-law} possesses a solution $\usol$ of class $C^{p+1}([0, T] \times \Omega; \mathcal{A})$, where $p \ge 1$ is the order of accuracy introduced in Assumption~\ref{asm:sbp} below.
  Moreover, there is a compact and convex set $K \subset \mathcal{A}$ containing a neighborhood of the range $\{ \usol(t, x) \mid (t,x) \in [0,T] \times \Omega \}$ such that
  \begin{equation}
  \label{eq:hessian-bounds}
    c_U \operatorname{I} \le U''(a) \le C_U \operatorname{I},
    \qquad a \in K,
  \end{equation}
  with constants $0 < c_U \le C_U < \infty$, i.e., the entropy $U$ is not only strictly convex but also strongly convex on $K$.
  If the source term does not vanish identically, we additionally assume $U \in C^3(K)$ and set $C_w = \sup_{K} |U'''|$;
  this is used only in the source estimate of Lemma~\ref{lem:source-bound}.
  For $s \equiv 0$, we set $C_w = 0$.
\end{assumption}

\begin{remark}[Convexity of $K$]
\label{rem:convexity-K}
  The convexity of $K$ guarantees that segments $b + \sigma (a - b)$, $\sigma \in [0, 1]$, connecting states $a, b \in K$ stay in $K$.
  This is used for Taylor expansions with integral remainders in the proof of the main result.
  For scalar problems, $K$ can be any compact interval containing the range of $\usol$ in its interior.
  For the shallow water and compressible Euler equations, the admissible sets are convex, so that $K$ can be chosen as the closed convex hull of the range of $\usol$ plus a small margin.
\end{remark}

\subsection{Summation-by-parts operators}
\label{ssec:sbp}

Typically, SBP operators are defined locally, e.g., in multi-block FD methods or DGSEM.
To facilitate the analysis, we work with global operators that can be constructed by coupling operators continuously like in CG methods or discontinuously like in DGSEM using central fluxes (or the corresponding simultaneous approximation terms) \cite{hicken2016multidimensional,ranocha2021broad}.
Typical implementations will use the local operators for efficiency \cite{ranocha2023efficient}; the global operators are only a theoretical tool for the analysis.

We discretize $\Omega$ by nodes $x_1, \dots, x_N \in \Omega$ in an arbitrary layout;
duplicated coordinates are deliberately allowed, since element-based operators such as DGSEM carry copies of the nodes at element interfaces.
Grid functions are collected in vectors $\vec{g} = (g_1, \dots, g_N)^T$; for a function $g\colon \Omega \to \R$, we use the same symbol for its nodal values $g_i = g(x_i)$.
All operations and estimates extend componentwise to grid functions of vector-valued quantities.
We measure distances of nodes by the periodic distance $\distper(x_i, x_j)$ on the torus $\Omega$, so that duplicated coordinates have distance zero.
All results are statements about a family of discretizations indexed by $h$;
accordingly, every constant introduced in the assumptions below is understood to be independent of $h$ and $N$.

\begin{assumption}[Diagonal-norm periodic SBP operators of order $p$]
\label{asm:sbp}
  There are a nominal mesh size $h > 0$ and matrices $M, Q_1, \dots, Q_d \in \R^{N \times N}$ with the first-derivative operators $D_m = M^{-1} Q_m$ such that
  \begin{enumerate}[label=(\roman*)]
    \item
    \label{item:sbp-norm}
    $M = \diag(m_1, \dots, m_N)$ with $c_M h^d \le m_i \le C_M h^d$ and $c_N \le N h^d \le C_N$ for constants $0 < c_M \le C_M$ and $0 < c_N \le C_N$ independent of $N$,
    \item
    \label{item:sbp-skew}
    $Q_m + Q_m^T = 0$ for all $m$,
    \item
    \label{item:sbp-constants}
    $Q_m \vec{1} = \vec{0}$ for all $m$, where $\vec{1} = (1, \dots, 1)^T$,
    \item
    \label{item:sbp-accuracy}
    there are $p \ge 1$ and $C_D > 0$ such that for all $g \in C^{p+1}(\Omega)$ and all $m$
    \begin{equation}
    \label{eq:accuracy}
      \abs{(D_m \vec{g})_i - (\partial_m g)(x_i)}
      \le
      C_D h^p \|g\|_{C^{p+1}(\Omega)}
      \qquad \text{uniformly in } i \text{ and } N.
    \end{equation}
  \end{enumerate}
\end{assumption}

Here and throughout, $\|g\|_{C^{p+1}(\Omega)} = \max_{|\gamma| \le p+1} \sup_{x \in \Omega} |\partial^\gamma g(x)|$ denotes the usual $C^{p+1}$ norm, i.e., the largest supremum norm over $\Omega$ among all partial derivatives $\partial^\gamma g$ up to order $p+1$, where $\gamma \in \N_0^d$ is a multi-index;
for vector-valued $g$ the derivatives are taken componentwise and $|\cdot|$ is the Euclidean norm.

Property \ref{item:sbp-skew} is the periodic summation-by-parts (SBP) property;
it implies $\vec{1}^T Q_m = -(Q_m \vec{1})^T = \vec{0}^T$ by \ref{item:sbp-constants}, so that all row and column sums of each $Q_m$ vanish.
Property \ref{item:sbp-constants} states that constants are differentiated exactly and is satisfied by every consistent method.
The node-count coupling in \ref{item:sbp-norm} expresses that $h$ plays the role of an average mesh width;
it implies $\sum_i m_i \le C_M C_N$ for all $N$.

\begin{remark}[Zero row sums as a metric identity]
\label{rem:gcl}
  For operators acting on curved grids, obtained by transforming reference operators to the physical domain, property \ref{item:sbp-constants} for the transformed operators is precisely the discrete free-stream-preservation condition, i.e., the discrete metric identities \cite{kopriva2006metric}, which in the finite-difference literature are also referred to as the geometric conservation law.
  This connection is made quantitative in Appendix~\ref{app:encapsulated}.
\end{remark}

\begin{remark}[The order $p$]
\label{rem:order-p}
  We define the order of accuracy $p$ directly by the estimate \eqref{eq:accuracy} for the assembled global operators acting on smooth grid functions, since this is precisely what the proof of the main result uses.
  Polynomial exactness alone does not imply \eqref{eq:accuracy}, not even together with the quasi-uniformity \ref{item:sbp-norm}:
  adding to $D_m$ an arbitrarily large matrix that locally annihilates the nodal values of all polynomials of degree at most $p$ preserves the exactness while destroying any bound on the action on general smooth data.
  A stability condition on the differentiation matrices is needed in addition.
  The fixed local stencil of Assumption~\ref{asm:bandwidth} below is such a condition and is satisfied by the standard operators:
  it gives $|(D_m)_{ij}| = m_i^{-1} |(Q_m)_{ij}| \le C_q h^{-1} / c_M$ and hence the row-sum bound $\sum_j |(D_m)_{ij}| \le \kappa C_q h^{-1} / c_M$, so that expanding $g$ around $x_i$ into its Taylor polynomial of degree $p$ plus a remainder bounded by $(\sqrt{d} \, \omega h)^{p+1} \|g\|_{C^{p+1}(\Omega)} / (p+1)!$ on the stencil, and applying the exactness to the polynomial part, yields \eqref{eq:accuracy} with a constant $C_D$ depending only on $\kappa$, $C_q$, $\omega$, $c_M$, $p$, and $d$.
  Any other uniform bound on the differentiation coefficients or on the operator norms serves the same purpose.
  See also \cite{glaubitz2026why} for related discussion on sparsity and additional properties of SBP operators.
\end{remark}

\begin{remark}[Operator classes]
\label{rem:operator-classes}
  Assumption~\ref{asm:sbp} covers, on Cartesian grids, tensor products of periodic finite-difference SBP operators \cite{kreiss1974finite,strand1994summation,carpenter1994time,svard2014review,fernandez2014review}, of DGSEM operators on Legendre-Gauss-Lobatto nodes with fixed polynomial degree \cite{gassner2013skew,carpenter2014entropy}, and of the standard nodal or mass-lumped realizations of continuous Galerkin \cite{hicken2016multidimensional,abgrall2020analysisI,ranocha2021broad} and flux reconstruction methods \cite{huynh2007flux,vincent2011newclass,ranocha2016summation} that are diagonal-norm SBP operators.
  Beyond tensor-product grids, the assumptions are verified in Appendix~\ref{app:encapsulated} for encapsulated operators on smooth curvilinear periodic grids in the sense of {\AA}lund and Nordstr{\"o}m \cite{alund2019encapsulated}, in two and in three space dimensions;
  metric-averaged curved-mesh flux differencing schemes reduce to the framework exactly and are covered whenever their discrete metric terms satisfy the discrete metric identities and are accurate of order $p$, as shown there.
  We emphasize that such global operators are a theoretical tool for the analysis;
  efficient implementations use element-local kernels instead of assembled global matrices \cite{ranocha2023efficient}.
\end{remark}

The assumptions do not cover formulations with a non-diagonal mass matrix.
Moreover, single-domain pseudospectral (Fourier) collocation methods satisfy Assumption~\ref{asm:sbp} only with $p$ depending on $N$ and are excluded from the structural assumption below; see the discussion in Section~\ref{sec:summary}.

\subsection{Entropy-conservative semidiscretizations}
\label{ssec:ec-semidiscretization}

We use the flux differencing framework of \cite{lefloch2002fully,fisher2013high} based on entropy-conservative fluxes \cite{tadmor1987numerical,tadmor2003entropy}.

\begin{assumption}[Entropy-conservative fluxes]
\label{asm:ec-flux}
  For each direction $m$, the two-point flux $\fnum_m \in C^{p+1}(\mathcal{A} \times \mathcal{A}; \R^n)$ is
  \begin{enumerate}[label=(\roman*)]
    \item
    \label{item:flux-symmetric}
    symmetric, $\fnum_m(a, b) = \fnum_m(b, a)$,
    \item
    \label{item:flux-consistent}
    consistent, $\fnum_m(a, a) = f_m(a)$,
    \item
    \label{item:flux-ec}
    entropy-conservative in the sense of Tadmor \cite{tadmor1987numerical,tadmor2003entropy}, i.e.,
    \begin{equation}
    \label{eq:ec-condition}
      \bigl(w(b) - w(a)\bigr)^T \fnum_m(a, b) = \psi_m(b) - \psi_m(a),
      \qquad a, b \in \mathcal{A}.
    \end{equation}
  \end{enumerate}
\end{assumption}

Given such fluxes, we consider the flux differencing semidiscretization \cite{lefloch2002fully,fisher2013high,gassner2016split,chen2017entropy,ranocha2018comparison}
\begin{equation}
\label{eq:semidiscretization}
  m_i \, \frac{\dif}{\dif t} \unum_i(t) + \sum_{m=1}^{d} 2 \sum_{j=1}^{N} (Q_m)_{ij} \fnum_m\bigl(\unum_i(t), \unum_j(t)\bigr) = m_i \, s(t, x_i),
  \quad i \in \{1, \dots, N\},
\end{equation}
with initial data $\unum_i(0)$ approximating $\usol(0, x_i)$ as specified in Assumption~\ref{asm:initial-data} below and the source values collocated at the nodes.
Since the fluxes are defined on $\mathcal{A} \times \mathcal{A}$, the semidiscretization is well-defined as long as all states $\unum_i(t)$ remain in $\mathcal{A}$;
part of the main result is that they even remain in the compact set $K \subset \mathcal{A}$ for sufficiently fine meshes.

\subsection{Discrete relative entropy and structural assumptions}
\label{ssec:structural-assumptions}

Let $\usol$ be the smooth solution from Assumption~\ref{asm:smooth-solution}.
We compare the numerical solution $\vnum(t)$ of \eqref{eq:semidiscretization} with the grid function $\vsol = (\usol_i)_{i=1}^N$ of nodal samples
\begin{equation}
\label{eq:nodal-samples}
  \usol_i(t) = \usol(t, x_i),
  \qquad i \in \{1, \dots, N\},
\end{equation}
and write $e_i(t) = \unum_i(t) - \usol_i(t)$ for the nodal errors.
We abbreviate the entropy variables along both trajectories by
\begin{equation}
  \wnum_i = w(\unum_i),
  \qquad
  \wsol_i = w(\usol_i).
\end{equation}
Our central quantity is the discrete relative entropy
\begin{equation}
\label{eq:relative-entropy}
  E_h(t)
  =
  \sum_{i=1}^{N} m_i \Bigl( U(\unum_i) - U(\usol_i) - \wsol_i^T (\unum_i - \usol_i) \Bigr),
\end{equation}
the discrete analogue of the relative entropy of Dafermos \cite{Dafermos1979} and DiPerna \cite{Diperna1979} with respect to the quadrature induced by $M$.
By strict convexity of $U$, $E_h(t) \ge 0$, with equality if and only if $\vnum(t) = \vsol(t)$ \cite[Section~3.1.3]{boyd2004convex}.

While all states are in $K$, we may expand each numerical flux around the exact states and define the flux Taylor remainders $r^m_{ij}$ by
\begin{equation}
\label{eq:flux-taylor}
  \fnum_m(\unum_i, \unum_j)
  =
  \fnum_m(\usol_i, \usol_j)
  + \partial_a \fnum_m(\usol_i, \usol_j) \, e_i
  + \partial_b \fnum_m(\usol_i, \usol_j) \, e_j
  + r^m_{ij},
\end{equation}
where $\partial_{a/b} \fnum_m$ denotes the derivative of $\fnum_m$ with respect to its first/second argument,
and $e_i = \unum_i - \usol_i$ are the nodal errors introduced above.
Since $K$ is convex and the fluxes are $C^{p+1}$ near $K \times K$ with $p \ge 1$, Taylor's theorem yields
\begin{equation}
\label{eq:remainder-bound}
  \abs{r^m_{ij}}
  \le
  \frac{C_S}{2} \bigl( |e_i|^2 + |e_j|^2 \bigr),
  \qquad
  C_S = \max_m \sup_{K \times K} \abs{\mathrm{D}^2 \fnum_m} < \infty.
\end{equation}

\begin{assumption}[Structural remainder bound]
\label{asm:remainder}
  There is a constant $C_R \ge 0$, independent of $h$ and $N$, such that, whenever $\unum_i \in K$ for all $i$,
  \begin{equation}
  \label{eq:remainder-hypothesis}
    \left| \sum_{m=1}^{d} \sum_{i,j=1}^{N} (\wsol_i - \wsol_j)^T (Q_m)_{ij} \, r^m_{ij} \right|
    \le
    C_R E_h.
  \end{equation}
\end{assumption}

Assumption~\ref{asm:remainder} is the only structural hypothesis of the main theorem beyond the SBP property and entropy conservation.
It is implied by the following checkable condition on the sparsity and locality of the operators, as we prove in Lemma~\ref{lem:bandwidth}.

\begin{assumptionstar}[Fixed local stencil]
\label{asm:bandwidth}
  There are constants $\omega, C_q > 0$ and $\kappa \in \N$ independent of $N$ such that, for each direction $m$,
  \begin{enumerate}[label=(\roman*)]
    \item
    \label{item:bandwidth-stencil}
    $(Q_m)_{ij} \ne 0$ implies $\distper(x_i, x_j) \le \omega h$,
    \item
    \label{item:bandwidth-count}
    each row of $Q_m$ has at most $\kappa$ nonzero entries,
    \item
    \label{item:bandwidth-entry}
    $|(Q_m)_{ij}| \le C_q h^{d-1}$.
  \end{enumerate}
\end{assumptionstar}

\begin{remark}[On the fixed-local-stencil condition]
\label{rem:stencil-condition}
  Conditions \ref{item:bandwidth-count} and \ref{item:bandwidth-entry} imply the row-sum bound $\sum_j |(Q_m)_{ij}| \le \kappa C_q h^{d-1}$, and by skew-symmetry the same bound holds for the column sums;
  the proofs below use only these sum bounds together with the locality \ref{item:bandwidth-stencil}, so the sum bound may replace \ref{item:bandwidth-count} and \ref{item:bandwidth-entry} as the weakest sufficient hypothesis.
  The locality condition \ref{item:bandwidth-stencil} is stated in terms of the physical periodic distance and is independent of any ordering of the unknowns.
  A bounded number of nonzeros per row alone would not suffice:
  a sparse matrix may still couple physically distant nodes, which destroys the smallness of $|\wsol_i - \wsol_j|$ exploited in Lemma~\ref{lem:bandwidth}.
  The entry scaling \ref{item:bandwidth-entry} is natural, since $Q_m = M D_m$ combines masses of size $\O(h^d)$ with difference operators of size $\O(h^{-1})$.
  For $d = 1$, the condition reduces to the fixed-stencil property of periodic finite-difference operators;
  all operator families listed in Remark~\ref{rem:operator-classes} satisfy it under $h$-refinement with fixed stencil or fixed polynomial degree.
\end{remark}

For element-based operators, the uniform bounds of Assumption~\ref{asm:sbp}~\ref{item:sbp-norm} and Assumption~\ref{asm:bandwidth} correspond to the usual quasi-uniformity and shape-regularity conditions.
Consider operators assembled from a fixed reference element with fixed nodes, positive quadrature weights, and fixed polynomial degree, where each element is the image of the reference element under a mapping whose Jacobian matrix is bounded by $C h$, whose inverse Jacobian matrix is bounded by $C h^{-1}$, and whose Jacobian determinant lies between $c h^d$ and $C h^d$, with constants $0 < c \le C$ independent of the element and of $h$.
For affine meshes, these conditions hold for quasi-uniform and shape-regular mesh families.
Then, each $m_i$ is a sum of reference quadrature weights multiplied by Jacobian determinants and satisfies $c_M h^d \le m_i \le C_M h^d$, provided that the number of elements sharing a node is bounded, which is again a consequence of shape regularity;
together with the bounded total volume, this also yields $c_N \le N h^d \le C_N$.
The entries of $Q_m$ combine the reference SBP matrices, which do not depend on $h$, with volume metric terms and, for discontinuous coupling by central fluxes or the corresponding simultaneous approximation terms, face metric terms (reference face weights times scaled normals), all of size $\O(h^{d-1})$, which gives \ref{item:bandwidth-entry};
the element diameters of size $\O(h)$ give the locality \ref{item:bandwidth-stencil}, and the fixed number of nodes per element together with the bounded number of neighboring elements gives \ref{item:bandwidth-count}.
Quasi-uniformity is required since a single mesh size $h$ enters the lower and upper bounds of Assumption~\ref{asm:sbp}~\ref{item:sbp-norm}.
For the encapsulated curvilinear operators of Appendix~\ref{app:encapsulated}, the grid is the image of the tensor-product grid of a one-dimensional periodic SBP operator under a fixed smooth diffeomorphism with Jacobian determinant bounded from below by $J_{\min} > 0$;
there, the corresponding bounds are established for sufficiently small $h$ with constants depending on the reference operator, on the derivatives of this diffeomorphism, and on $J_{\min}$.
The remaining properties of Assumption~\ref{asm:sbp} on curved elements depend on the discrete metric terms and are not implied by the mesh conditions above.

\begin{assumption}[Source term]
\label{asm:source}
  The source $s\colon [0, T] \times \Omega \to \R^n$ is continuous and bounded, and it enters the scheme \eqref{eq:semidiscretization} by exact collocation, i.e., by the values $s(t, x_i)$.
\end{assumption}

We emphasize that no structural relation between the source and the entropy is assumed;
the mechanism behind this generality is explained in Section~\ref{ssec:source-bound}.
If the source terms are state-dependent, as for the shallow water equations with variable bathymetry or the compressible Euler equations with gravity, additional residual terms appear that require separate estimates;
this is beyond the scope of the present work.

\begin{assumption}[Consistent initial data]
\label{asm:initial-data}
  The initial data satisfy $\unum_i(0) \in K$ for all $i$ and $E_h(0)^{1/2} \le C_0 h^p$ with a constant $C_0 \ge 0$ independent of $h$ and $N$.
\end{assumption}

If the scheme is initialized by nodal sampling, $\unum_i(0) = \usol(0, x_i)$, then $E_h(0) = 0$ and Assumption~\ref{asm:initial-data} holds trivially.

\begin{assumption}[Order condition]
\label{asm:order}
  The order of accuracy is $p > d/2$.
\end{assumption}

Assumption~\ref{asm:order} is needed only for the $L^\infty$ bootstrap argument in the proof.
In one space dimension it is satisfied by every consistent method;
for $d \in \{2, 3\}$ it excludes only first-order methods.
Its role, conditional versions of the main result that avoid it, and a possible route towards removing it are discussed in Section~\ref{ssec:order-condition}.
The numerical results in Section~\ref{sec:numerics-2d} suggest that it is a technical assumption that is not necessary for convergence in practice.

\section{Main result}
\label{sec:main-result}

Our main result is the following convergence theorem.
The proof is developed in Section~\ref{sec:proof}.

\begin{theorem}[Convergence of entropy-conservative SBP semidiscretizations]
\label{thm:convergence}
  Let Assumptions~\ref{asm:entropy-pair}, \ref{asm:smooth-solution}, \ref{asm:sbp}, \ref{asm:ec-flux}, \ref{asm:remainder}, \ref{asm:source}, \ref{asm:initial-data}, and \ref{asm:order} be satisfied.
  Then, there is $h_0 > 0$ depending on $T$ and the data such that for all $0 < h < h_0$, the semidiscretization \eqref{eq:semidiscretization} has a unique solution on $[0, T]$, all states $\unum_i(t)$ remain in $K$, and
  \begin{equation}
  \label{eq:main-estimate}
    \norm[1]{\vnum(t) - \vsol(t)}_M
    \le
    C_T h^p,
    \qquad t \in [0, T],
  \end{equation}
  where $\|\vec{a}\|_M^2 = \sum_i m_i |a_i|^2$ and $C_T$ is independent of $h$;
  in the presence of a source, $C_T$ additionally depends on $\norm{s}_{L^\infty}$ and $C_w$.
  In particular, the numerical solution converges to the smooth solution at the order of accuracy $p$ of the SBP operators.
\end{theorem}

The rate $p$ in \eqref{eq:main-estimate} is the order of the accuracy estimate \eqref{eq:accuracy} of the assembled operators.
Numerical solutions may converge faster;
e.g., DG methods exhibit superconvergence phenomena depending on the polynomial degree and the numerical flux \cite{adjerid2002posteriori,cheng2010superconvergence,cao2014superconvergence}.
Theorem~\ref{thm:convergence} makes no claim about such improved rates.

The structural bound of Assumption~\ref{asm:remainder} is the theorem hypothesis;
the fixed-local-stencil condition of Assumption~\ref{asm:bandwidth} is a sufficient condition that is easy to check in practice (see Lemma~\ref{lem:bandwidth}).
The order condition $p > d/2$ of Assumption~\ref{asm:order} enters only through the inverse estimate in the bootstrap argument of Section~\ref{ssec:gronwall};
see Section~\ref{ssec:order-condition} for a detailed discussion.

Theorem~\ref{thm:convergence} places no structural restriction on how the source interacts with the entropy;
in particular, sources constructed by the method of manufactured solutions \cite{salari2000code} are covered.
The reason is that the same source values enter the numerical and the exact dynamics, so that their contributions to the relative entropy cancel up to terms that are quadratic in the error (Lemma~\ref{lem:source-bound});
the mild price is the additional regularity $U \in C^3(K)$ of Assumption~\ref{asm:smooth-solution}, which is vacuous for quadratic entropies.

Theorem~\ref{thm:convergence} establishes convergence of the nodal values in the discrete $M$-norm.
Convergence in $L^2(\Omega)$ follows for every linear reconstruction operator $R_h$, mapping grid functions to functions on $\Omega$, that is stable and accurate in the sense that
\begin{equation}
\label{eq:reconstruction}
  \|R_h \vec{g}\|_{L^2(\Omega)} \le C_{\mathrm{rec}} \|\vec{g}\|_M,
  \qquad
  \|R_h \vsol(t) - \usol(t, \cdot)\|_{L^2(\Omega)} \le C_{\mathrm{rec}} h^p,
\end{equation}
for all grid functions $\vec{g}$ and all $t \in [0, T]$, with a constant $C_{\mathrm{rec}}$ independent of $h$.
Indeed, linearity and the triangle inequality give
\begin{equation}
\begin{aligned}
  \|R_h \vnum(t) - \usol(t, \cdot)\|_{L^2(\Omega)}
  &\le
  \|R_h (\vnum(t) - \vsol(t))\|_{L^2(\Omega)} + \|R_h \vsol(t) - \usol(t, \cdot)\|_{L^2(\Omega)}
  \\
  &\le
  C_{\mathrm{rec}} (C_T + 1) h^p.
\end{aligned}
\end{equation}
For element-based operators of fixed polynomial degree $k = p$ on tensor-product elements with tensor-product nodes, such as DGSEM and CGSEM on Legendre--Gauss--Lobatto nodes, $R_h$ can be chosen as the elementwise polynomial interpolation of the nodal values, using the values of the respective element at duplicated interface nodes of DGSEM.
Its accuracy of order $h^{k+1}$ for $\usol(t, \cdot) \in C^{p+1}$ follows from standard interpolation estimates on affine elements, and on curved elements if the element mappings $F_K$ additionally satisfy $|\mathrm{D}^j F_K| \le C h^j$ for $2 \le j \le k+1$, as for mappings obtained from a fixed smooth diffeomorphism.
For its stability, note that the nodes are unisolvent for the tensor-product polynomials of degree $k$ on the reference element, so that the discrete norm induced by the positive quadrature weights is a norm on this finite-dimensional space and therefore equivalent to the $L^2$ norm.
Under the mesh conditions discussed after Remark~\ref{rem:stencil-condition}, the bounds $c h^d \le \det \mathrm{D} F_K \le C h^d$ transfer this equivalence to all elements with constants independent of $h$.
Summing over the elements and using that each diagonal entry $m_i$ of the assembled norm is the sum of the local weights associated with node $i$ (a single local weight for the duplicated interface nodes of DGSEM) yields the stability estimate in \eqref{eq:reconstruction}.

\section{Proof of the main theorem}
\label{sec:proof}

Throughout this section, the assumptions of Theorem~\ref{thm:convergence} are in force.
All estimates are stated for time intervals on which the numerical states $\unum_i(t)$ remain in $K$;
the bootstrap argument in Section~\ref{ssec:gronwall} shows that this covers $[0, T]$ for sufficiently fine meshes.
Generic constants $C$ may change from line to line but are always independent of $h$, $N$, and $t \in [0, T]$.

\subsection{Coercivity of the relative entropy}
\label{ssec:relative-entropy}

We first record that the discrete relative entropy $E_h$ \eqref{eq:relative-entropy} is equivalent to the squared $M$-norm of the error $\vec{e} = \vnum - \vsol$.

\begin{lemma}[Coercivity]
\label{lem:coercivity}
  If $\unum_i \in K$ for all $i$, then
  \begin{equation}
  \label{eq:coercivity}
    \frac{c_U}{2} \|\vec{e}\|_M^2
    \le
    E_h
    \le
    \frac{C_U}{2} \|\vec{e}\|_M^2.
  \end{equation}
\end{lemma}

\begin{proof}
  Taylor expansion with integral remainder yields
  \begin{equation}
    U(\unum_i) - U(\usol_i) - \wsol_i^T e_i
    =
    \int_0^1 (1 - \sigma) \, e_i^T U''(\usol_i + \sigma e_i) e_i \dif \sigma.
  \end{equation}
  Since $K$ is convex, the segment $\usol_i + \sigma e_i$, $\sigma \in [0, 1]$, lies in $K$, so that the Hessian bounds \eqref{eq:hessian-bounds} give
  $\frac{c_U}{2} |e_i|^2 \le U(\unum_i) - U(\usol_i) - \wsol_i^T e_i \le \frac{C_U}{2} |e_i|^2$.
  Multiplying by $m_i$ and summing over $i$ yields \eqref{eq:coercivity}.
\end{proof}

\subsection{Discrete entropy balance}
\label{ssec:entropy-conservation}

The semidiscretization \eqref{eq:semidiscretization} reproduces the entropy balance of the continuous problem exactly.
For vanishing sources, this is the classical discrete entropy conservation \cite{tadmor1987numerical,fisher2013high,chen2017entropy,ranocha2018comparison};
we include the short proof since the same mechanism is used repeatedly below.

\begin{lemma}[Discrete entropy balance]
\label{lem:entropy-conservation}
  As long as the solution of \eqref{eq:semidiscretization} stays in $\mathcal{A}$,
  \begin{equation}
    \frac{\dif}{\dif t} \sum_{i=1}^N m_i U(\unum_i) = \sum_{i=1}^N m_i \, (\wnum_i)^T s(t, x_i).
  \end{equation}
  In particular, the total discrete entropy is conserved exactly if the source vanishes.
\end{lemma}

\begin{proof}
  Contracting \eqref{eq:semidiscretization} with the entropy variables $\wnum_i$ yields
  \begin{multline}
    \frac{\dif}{\dif t} \sum_i m_i U(\unum_i)
    =
    \sum_i m_i (\wnum_i)^T \frac{\dif \unum_i}{\dif t}
    \\
    =
    - \sum_{m} 2 \sum_{i,j} (\wnum_i)^T (Q_m)_{ij} \fnum_m(\unum_i, \unum_j)
    + \sum_i m_i (\wnum_i)^T s(t, x_i).
  \end{multline}
  Fix a direction $m$.
  Exchanging the summation indices $i \leftrightarrow j$, using the skew-symmetry of $Q_m$ and the symmetry of $\fnum_m$, and averaging the two expressions gives
  \begin{equation}
  \label{eq:antisymmetrization}
    2 \sum_{i,j} (\wnum_i)^T (Q_m)_{ij} \fnum_m(\unum_i, \unum_j)
    =
    \sum_{i,j} (\wnum_i - \wnum_j)^T (Q_m)_{ij} \fnum_m(\unum_i, \unum_j).
  \end{equation}
  By the entropy-conservation condition \eqref{eq:ec-condition} with $a = \unum_j$ and $b = \unum_i$, each summand reduces to the potential difference $(Q_m)_{ij} \bigl( \psi_m(\unum_i) - \psi_m(\unum_j) \bigr)$.
  Thus,
  \begin{equation}
  \begin{aligned}
    &\quad
    \sum_{i,j} (Q_m)_{ij} \bigl( \psi_m(\unum_i) - \psi_m(\unum_j) \bigr)
    \\
    &=
    \sum_{i} \psi_m(\unum_i) \sum_{j} (Q_m)_{ij}
    -
    \sum_{j} \psi_m(\unum_j) \sum_{i} (Q_m)_{ij}
    =
    0,
  \end{aligned}
  \end{equation}
  since all row and column sums of $Q_m$ vanish.
\end{proof}

\subsection{Truncation error of the flux differencing form}
\label{ssec:truncation}

Next, we analyze the consistency of the scheme, i.e., we insert the nodal samples $\vsol(t)$ of the smooth solution into the scheme and quantify the defect.
Results of this type go back to \cite{lefloch2002fully,fisher2013high} and were generalized to arbitrary smooth, consistent, and symmetric two-point fluxes in \cite[Theorem~3.1]{ranocha2018comparison} and \cite[Theorem~3.3]{chen2017entropy};
we give a short self-contained proof based on \cite[Theorem~3.3]{chen2017entropy} yielding the uniform-in-$i$ bound with the constants required later.
We define the truncation errors $\tau_i$ by
\begin{equation}
\label{eq:truncation-definition}
  m_i \tau_i
  =
  m_i \frac{\dif \usol_i}{\dif t}
  + \sum_{m=1}^{d} 2 \sum_{j} (Q_m)_{ij} \fnum_m(\usol_i, \usol_j)
  - m_i \, s(t, x_i),
  \quad i \in \{1, \dots, N\}.
\end{equation}
Here, $\usol_i(t) = \usol(t, x_i)$ are the nodal samples \eqref{eq:nodal-samples} of the smooth solution, not the numerical solution $\unum_i(t)$.

\begin{lemma}[Truncation error]
\label{lem:truncation}
  Under Assumptions~\ref{asm:smooth-solution}, \ref{asm:sbp}, \ref{asm:ec-flux}, and \ref{asm:source},
  \begin{equation}
    \max_{i} |\tau_i| \le C h^p
    \qquad \text{and} \qquad
    \norm{\vec{\tau}}_M \le C (C_M C_N)^{1/2} h^p,
  \end{equation}
  where $C$ depends only on $d$, $C_D$, $\norm{\usol}_{C^{p+1}}$, and the $C^{p+1}$-norms of the fluxes near $K \times K$.
\end{lemma}

\begin{proof}
  Fix $i$, $t$, and a direction $m$, and consider the auxiliary function
  \begin{equation}
    g^m_i(x) = 2 \fnum_m\bigl(\usol_i, \usol(t, x)\bigr),
  \end{equation}
  which is $C^{p+1}$ in $x$ with $\|g^m_i\|_{C^{p+1}(\Omega)}$ bounded uniformly in $i$ and $t$ by the chain rule, since $\fnum_m$ is $C^{p+1}$ near $K \times K$ and $\usol \in C^{p+1}$ with values in $K$.
  The nodal values of $g^m_i$ are $g^m_i(x_j) = 2 \fnum_m(\usol_i, \usol_j)$, so that, by the diagonal structure of $M$,
  \begin{equation}
    \frac{2}{m_i} \sum_j (Q_m)_{ij} \fnum_m(\usol_i, \usol_j)
    =
    \bigl( M^{-1} Q_m \, \vec{g^m_i} \bigr)_i
    =
    (D_m \vec{g^m_i})_i.
  \end{equation}
  Applying the accuracy estimate \eqref{eq:accuracy} componentwise yields
  \begin{equation}
    (D_m \vec{g^m_i})_i = (\partial_m g^m_i)(x_i) + R^m_i,
    \qquad
    \abs{R^m_i} \le C_D h^p \|g^m_i\|_{C^{p+1}(\Omega)} \le C h^p.
  \end{equation}
  It remains to identify $(\partial_m g^m_i)(x_i)$.
  By the chain rule,
  \begin{equation}
    \partial_m g^m_i(x) = 2 \, \partial_b \fnum_m\bigl(\usol_i, \usol(t, x)\bigr) \, \partial_m \usol(t, x),
  \end{equation}
  where $\partial_{a/b} \fnum_m(a, b)$ denotes again the derivative of $\fnum_m$ with respect to its first/second argument.
  Differentiating the consistency relation $\fnum_m(a, a) = f_m(a)$ gives $\partial_a \fnum_m(a, a) + \partial_b \fnum_m(a, a) = f_m'(a)$, and differentiating the symmetry relation $\fnum_m(a, b) = \fnum_m(b, a)$ gives $\partial_a \fnum_m(a, a) = \partial_b \fnum_m(a, a)$;
  hence
  \begin{equation}
  \label{eq:flux-derivative-consistency}
    2 \, \partial_b \fnum_m(a, a) = f_m'(a).
  \end{equation}
  Consequently, $(\partial_m g^m_i)(x_i) = f_m'(\usol_i) \, \partial_m \usol(t, x_i) = \partial_m f_m\bigl(\usol(t, x)\bigr)\big|_{x = x_i}$.
  Since $\usol$ solves the balance law \eqref{eq:conservation-law} classically,
  \begin{equation}
    \frac{\dif \usol_i}{\dif t} = \partial_t \usol(t, x_i) = - \sum_m \partial_m f_m\bigl(\usol(t, x)\bigr)\big|_{x = x_i} + s(t, x_i).
  \end{equation}
  Inserting this into \eqref{eq:truncation-definition}, the collocated source values cancel exactly and $\tau_i = \sum_m R^m_i$.
  The $M$-norm bound follows from $\sum_i m_i \le C_M C_N$, see Assumption~\ref{asm:sbp}~\ref{item:sbp-norm}.
\end{proof}

\begin{remark}[Smoothness of the fluxes]
\label{rem:flux-smoothness}
  The regularity $\fnum_m \in C^{p+1}$ required by Assumption~\ref{asm:ec-flux} is dictated by Lemma~\ref{lem:truncation};
  all other steps of the proof only use $C^2$ regularity through the constant $C_S$ in \eqref{eq:remainder-bound}.
  Conversely, lower regularity of the solution or of the fluxes lowers the rate guaranteed by our analysis, even if the SBP operators have a higher order.
  For instance, let the operators satisfy Assumption~\ref{asm:bandwidth} and be exact for polynomials of degree $p$, as the operators on Cartesian grids of Remark~\ref{rem:operator-classes} (with $p$ defined as in Section~\ref{sec:numerics-setup} for multi-block finite-difference operators), but let $\usol$ and $\fnum_m$ only be of class $C^{q+1}$ for some integer $q$ with $d/2 < q < p$.
  Then, the argument of Remark~\ref{rem:order-p} yields \eqref{eq:accuracy} with $q$ instead of $p$ for all $g \in C^{q+1}(\Omega)$, and Theorem~\ref{thm:convergence} and all results of this section hold with $q$ in place of $p$ in Assumptions~\ref{asm:smooth-solution}, \ref{asm:ec-flux}, \ref{asm:initial-data}, and \ref{asm:order}.
  For the encapsulated curvilinear operators of Appendix~\ref{app:encapsulated}, which are not exact for polynomials on curved grids, the same conclusion follows from the proof of Proposition~\ref{prop:encapsulated-accuracy} if the one-dimensional reference operator is exact for polynomials of degree $p$.
  Thus, if $\usol$ and $\fnum_m$ are of class $C^{q+1}$ for an integer $q > d/2$, the guaranteed rate is $\min\{p, q\}$.
\end{remark}

\subsection{Cancellation of the linear error terms}
\label{ssec:cancellation}

We now compute the evolution of the discrete relative entropy.
This is the core of the proof:
the entropy-conservation condition \eqref{eq:ec-condition} and its first derivatives force all terms that are linear in the error to cancel exactly, separately in each direction.

Recall that the pointwise error is $e_i = \unum_i - \usol_i$.
We abbreviate the flux Jacobians at the exact states by
\begin{equation}
  A^{m,a}_{ij} = \partial_a \fnum_m(\usol_i, \usol_j),
  \qquad
  A^{m,b}_{ij} = \partial_b \fnum_m(\usol_i, \usol_j).
\end{equation}

\begin{lemma}[Evolution of the relative entropy]
\label{lem:master}
  As long as the solution of \eqref{eq:semidiscretization} stays in $K$,
  \begin{equation}
  \label{eq:master}
  \begin{aligned}
    \frac{\dif}{\dif t} E_h
    ={}&
    \sum_{m} \sum_{i,j} (\wsol_i - \wsol_j)^T (Q_m)_{ij} \bigl( A^{m,a}_{ij} e_i + A^{m,b}_{ij} e_j \bigr)
    \\
    &
    + \sum_{m} 2 \sum_{i,j} (Q_m)_{ij} \, e_i^T U''(\usol_i) \fnum_m(\usol_i, \usol_j)
    \\
    &
    + \sum_{m} \sum_{i,j} (\wsol_i - \wsol_j)^T (Q_m)_{ij} \, r^m_{ij}
    - \sum_{i} m_i \, e_i^T U''(\usol_i) \tau_i
    + \mathrm{SRC},
  \end{aligned}
  \end{equation}
  with the source contribution
  \begin{equation}
  \label{eq:src-definition}
    \mathrm{SRC}
    =
    \sum_{i} m_i \bigl[ (\wnum_i - \wsol_i) - U''(\usol_i) e_i \bigr]^T s(t, x_i).
  \end{equation}
\end{lemma}

\begin{proof}
  Differentiating \eqref{eq:relative-entropy} in time and using $\frac{\dif}{\dif t} \wsol_i = U''(\usol_i) \frac{\dif \usol_i}{\dif t}$ gives
  \begin{equation}
  \label{eq:Eh-derivative}
    \frac{\dif}{\dif t} E_h
    =
    \sum_i m_i (\wnum_i - \wsol_i)^T \frac{\dif \unum_i}{\dif t}
    -
    \sum_i m_i \, e_i^T U''(\usol_i) \frac{\dif \usol_i}{\dif t},
  \end{equation}
  where the terms $\mp \wsol_i^T \frac{\dif \usol_i}{\dif t}$ canceled and we used the symmetry of $U''$.
  We treat the two sums separately.

  For the first sum, we insert the scheme \eqref{eq:semidiscretization};
  the source contributes $\sum_i m_i (\wnum_i - \wsol_i)^T s(t, x_i)$, and each directional flux term is antisymmetrized as in \eqref{eq:antisymmetrization},
  \begin{equation}
  \begin{aligned}
    &\quad
    - 2 \sum_{i,j} (\wnum_i - \wsol_i)^T (Q_m)_{ij} \fnum_m(\unum_i, \unum_j)
    \\&
    =
    - \sum_{i,j} \bigl( (\wnum_i - \wsol_i) - (\wnum_j - \wsol_j) \bigr)^T (Q_m)_{ij} \fnum_m(\unum_i, \unum_j).
  \end{aligned}
  \end{equation}
  The contribution of $\wnum_i - \wnum_j$ vanishes exactly as in the proof of Lemma~\ref{lem:entropy-conservation}.
  The remaining contribution enters with a positive sign and is expanded by the flux Taylor formula \eqref{eq:flux-taylor},
  \begin{equation}
  \begin{aligned}
    &\quad
    \sum_{i,j} (\wsol_i - \wsol_j)^T (Q_m)_{ij} \fnum_m(\unum_i, \unum_j)
    \\&
    =
    \sum_{i,j} (\wsol_i - \wsol_j)^T (Q_m)_{ij} \Bigl( \fnum_m(\usol_i, \usol_j) + A^{m,a}_{ij} e_i + A^{m,b}_{ij} e_j + r^m_{ij} \Bigr).
  \end{aligned}
  \end{equation}
  In the first term, the entropy-conservation condition \eqref{eq:ec-condition} evaluated at the exact states gives $(\wsol_i - \wsol_j)^T \fnum_m(\usol_i, \usol_j) = \psi_m(\usol_i) - \psi_m(\usol_j)$, which sums to zero against $Q_m$ as in Lemma~\ref{lem:entropy-conservation}.

  For the second sum in \eqref{eq:Eh-derivative}, we insert the definition \eqref{eq:truncation-definition} of the truncation error $\tau_i$,
  \begin{multline}
    - \sum_i m_i \, e_i^T U''(\usol_i) \frac{\dif \usol_i}{\dif t}
    \\
    =
    \sum_{m} 2 \sum_{i,j} (Q_m)_{ij} \, e_i^T U''(\usol_i) \fnum_m(\usol_i, \usol_j)
    -
    \sum_i m_i \, e_i^T U''(\usol_i) \bigl( s(t, x_i) + \tau_i \bigr).
  \end{multline}
  Collecting all terms and combining the two source contributions into \eqref{eq:src-definition} yields \eqref{eq:master}.
\end{proof}

The first two sums \eqref{eq:master} collect all flux terms that are linear in the error.
The following lemma shows that they add up to zero, separately in each direction.

\begin{lemma}[Cancellation]
\label{lem:cancellation}
  For each direction $m$, all error vectors $\vec{e}$, and all nodal states $\usol_i \in K$,
  \begin{equation}
  \label{eq:cancellation}
    \sum_{i,j} (\wsol_i - \wsol_j)^T (Q_m)_{ij} \bigl( A^{m,a}_{ij} e_i + A^{m,b}_{ij} e_j \bigr)
    + 2 \sum_{i,j} (Q_m)_{ij} \, e_i^T U''(\usol_i) \fnum_m(\usol_i, \usol_j)
    =
    0.
  \end{equation}
\end{lemma}

\begin{proof}
  The function
  \begin{equation}
    G_m(a, b) = \bigl( w(b) - w(a) \bigr)^T \fnum_m(a, b) - \psi_m(b) + \psi_m(a)
  \end{equation}
  vanishes identically on $K \times K$ by the entropy-conservation condition \eqref{eq:ec-condition}.
  Hence, its partial derivatives vanish as well:
  \begin{equation}
  \label{eq:diff-ec}
  \begin{aligned}
    0 = \partial_a G_m(a, b)
    &=
    \bigl[ \partial_a \fnum_m(a, b) \bigr]^T \bigl( w(b) - w(a) \bigr)
    - U''(a) \fnum_m(a, b)
    + \psi_m'(a),
    \\
    0 = \partial_b G_m(a, b)
    &=
    \bigl[ \partial_b \fnum_m(a, b) \bigr]^T \bigl( w(b) - w(a) \bigr)
    + U''(b) \fnum_m(a, b)
    - \psi_m'(b).
  \end{aligned}
  \end{equation}
  We evaluate \eqref{eq:diff-ec} at $(a, b) = (\usol_i, \usol_j)$, so that $w(b) - w(a) = \wsol_j - \wsol_i$ and $\partial_{a/b} \fnum_m(a, b) = A^{m,a/b}_{ij}$.
  Thus, multiplication by $e_i$ and $e_j$ gives
  \begin{equation}
  \begin{aligned}
    (\wsol_i - \wsol_j)^T A^{m,a}_{ij} e_i
    &=
    - \fnum_m(\usol_i, \usol_j)^T U''(\usol_i) e_i + \psi_m'(\usol_i)^T e_i,
    \\
    (\wsol_i - \wsol_j)^T A^{m,b}_{ij} e_j
    &=
    \fnum_m(\usol_i, \usol_j)^T U''(\usol_j) e_j - \psi_m'(\usol_j)^T e_j.
  \end{aligned}
  \end{equation}
  The contributions of $\psi_m'$ telescope to zero against $Q_m$, since all row and column sums of $Q_m$ vanish.
  Adding the term $2 \sum_{i,j} (Q_m)_{ij} \, e_i^T U''(\usol_i) \fnum_m(\usol_i, \usol_j)$, the left-hand side of \eqref{eq:cancellation} becomes
  \begin{equation}
    \sum_{i,j} (Q_m)_{ij} X^m_{ij},
    \qquad
    X^m_{ij} = \fnum_m(\usol_i, \usol_j)^T \bigl( U''(\usol_i) e_i + U''(\usol_j) e_j \bigr).
  \end{equation}
  The summand $X^m_{ij}$ is symmetric under the exchange $i \leftrightarrow j$, because $\fnum_m$ is symmetric and the bracket is invariant.
  Contracting a symmetric array with the skew-symmetric $Q_m$ gives zero.
\end{proof}

Combining Lemma~\ref{lem:master} and Lemma~\ref{lem:cancellation}, we arrive at the identity
\begin{equation}
\label{eq:sharp}
  \frac{\dif}{\dif t} E_h
  =
  \sum_{m} \sum_{i,j} (\wsol_i - \wsol_j)^T (Q_m)_{ij} \, r^m_{ij}
  -
  \sum_{i} m_i \, e_i^T U''(\usol_i) \tau_i
  + \mathrm{SRC},
\end{equation}
in which only the quadratic flux remainders $r^m_{ij}$, the truncation errors $\tau_i$, and the source contribution appear.

\subsection{Bounding the source term}
\label{ssec:source-bound}

The source contribution \eqref{eq:src-definition} is the only genuinely new term compared to the source-free case, and the following estimate is the reason no entropy structure of the source is needed.

\begin{lemma}[Source bound]
\label{lem:source-bound}
  If $\unum_i \in K$ for all $i$, then
  \begin{equation}
    \abs{\mathrm{SRC}}
    \le
    \frac{C_w}{c_U} \, \norm{s}_{L^\infty([0,T] \times \Omega)} \, E_h.
  \end{equation}
\end{lemma}

\begin{proof}
  If $s \equiv 0$, then $\mathrm{SRC} = 0$ and the bound holds trivially with $C_w = 0$.
  Otherwise, Assumption~\ref{asm:smooth-solution} provides $U \in C^3(K)$, so that Taylor expansion of $w = U'$ with integral remainder along the segment $\usol_i + \sigma e_i \in K$ gives
  \begin{equation}
    \abs[1]{w(\unum_i) - w(\usol_i) - U''(\usol_i) e_i}
    \le
    \frac{C_w}{2} |e_i|^2.
  \end{equation}
  Multiplying by $\abs{s(t, x_i)}$, summing with the weights $m_i$, and applying the coercivity \eqref{eq:coercivity} yields the claim.
\end{proof}

Each of the two terms constituting $\mathrm{SRC}$ in the proof of Lemma~\ref{lem:master} is only linear in the error, with coefficients of order $\norm{s}_{L^\infty}$;
estimating them separately would yield $\frac{\dif}{\dif t} E_h \le C E_h + C' E_h^{1/2}$ and destroy the convergence rate.
It is their difference that is quadratic in the error, because the same source values enter the numerical and the exact dynamics.
This is where the exact collocation of a state-independent source in Assumption~\ref{asm:source} is used.
For quadratic entropies, $U''' = 0$ and $\mathrm{SRC} = 0$ exactly.
More generally, if the source is neutral or dissipative with respect to the relative entropy in the sense that $\mathrm{SRC} \le 0$ whenever $\unum_i \in K$ for all $i$, then this sign condition can replace Lemma~\ref{lem:source-bound}, and the term $C_{\mathrm{src}} E_h$ disappears from the Gronwall estimate \eqref{eq:gronwall-differential} below.
Since $\mathrm{SRC}$ is the sum of the weighted second-order Taylor remainders of the functions $a \mapsto w(a)^T s(t, x_i)$, a sufficient condition is that these functions are concave on $K$, e.g., $U'''(a) \, s(t, x) \le 0$ for scalar problems.
In this case, the regularity $U \in C^3(K)$ and the constant $C_w$ are not needed, and the constant $C_T$ of Theorem~\ref{thm:convergence} does not depend on $\norm{s}_{L^\infty}$ and $C_w$;
the rate $h^p$, which is determined by the truncation error of Lemma~\ref{lem:truncation}, is unchanged in general.

\subsection{Bounding the quadratic remainder}
\label{ssec:remainder}

The first sum in \eqref{eq:sharp} is controlled by the structural bound \eqref{eq:remainder-hypothesis} of Assumption~\ref{asm:remainder}.
We now prove that the fixed-local-stencil condition of Assumption~\ref{asm:bandwidth} implies this bound.

\begin{lemma}[Fixed local stencil implies the structural bound]
\label{lem:bandwidth}
  Let Assumptions~\ref{asm:smooth-solution}, \ref{asm:sbp}, \ref{asm:ec-flux}, and \ref{asm:bandwidth} be satisfied.
  Then, Assumption~\ref{asm:remainder} holds with
  \begin{equation}
    C_R
    =
    \frac{2}{c_U} \, \frac{d \, C_S C_L \omega \kappa C_q}{c_M},
    \qquad
    C_L = \sup_{t \in [0, T]} \norm[1]{\nabla_x ( w \circ \usol )(t, \cdot)}_{L^\infty(\Omega)}.
  \end{equation}
\end{lemma}

\begin{proof}
  The map $x \mapsto w\bigl(\usol(t, x)\bigr)$ is continuously differentiable and periodic, so that
  \begin{equation}
    |\wsol_i - \wsol_j|
    \le
    C_L \distper(x_i, x_j).
  \end{equation}
  Fix a direction $m$.
  On the stencil of Assumption~\ref{asm:bandwidth}~\ref{item:bandwidth-stencil}, i.e., for $(Q_m)_{ij} \ne 0$, this yields $|\wsol_i - \wsol_j| \le C_L \omega h$.
  Together with the remainder bound \eqref{eq:remainder-bound},
  \begin{equation}
    \left| \sum_{i,j} (\wsol_i - \wsol_j)^T (Q_m)_{ij} \, r^m_{ij} \right|
    \le
    C_L \omega h \, \frac{C_S}{2} \sum_{i,j} |(Q_m)_{ij}| \bigl( |e_i|^2 + |e_j|^2 \bigr).
  \end{equation}
  By Remark~\ref{rem:stencil-condition}, the row sums of $Q_m$ are bounded by $\kappa C_q h^{d-1}$, so
  \begin{equation}
    \sum_{i,j} |(Q_m)_{ij}| |e_i|^2 \le \kappa C_q h^{d-1} \sum_i |e_i|^2.
  \end{equation}
  By skew-symmetry, $|(Q_m)_{ij}| = |(Q_m)_{ji}|$, the same bound holds for the column sums and thus for the second contribution.
  Finally, the quasi-uniformity $m_i \ge c_M h^d$ of Assumption~\ref{asm:sbp}~\ref{item:sbp-norm} and the coercivity \eqref{eq:coercivity} give
  \begin{equation}
    \sum_i |e_i|^2
    \le
    \frac{1}{c_M h^d} \|\vec{e}\|_M^2
    \le
    \frac{2}{c_U \, c_M h^d} E_h.
  \end{equation}
  Combining the three estimates, the powers of $h$ cancel exactly, $h \cdot h^{d-1} \cdot h^{-d} = 1$;
  summing over the $d$ directions yields the claim.
\end{proof}

\subsection{Gronwall estimate and bootstrap}
\label{ssec:gronwall}

We can now prove Theorem~\ref{thm:convergence}.
The proof combines a Gronwall estimate with a bootstrap (continuation) argument that verifies the required bounds on the numerical solution.
Arguments of this type are classical in the error analysis of nonlinear problems;
for example, the error estimates for smooth solutions of Runge--Kutta discontinuous Galerkin methods \cite[Section~5]{zhang2004error}, \cite[Section~4]{zhang2006error} and of semidiscrete discontinuous Galerkin methods with quadrature \cite[proof of Theorem~3.1]{huang2016error} verify an a~priori assumption on the numerical solution in this way, by induction over the time steps in the fully discrete case.

\begin{proof}[Proof of Theorem~\ref{thm:convergence}]
  Let $\delta > 0$ be such that the closed $\delta$-neighborhood of the range of $\usol$ is contained in $K$;
  such a $\delta$ exists by Assumption~\ref{asm:smooth-solution}.
  Since the right-hand side of \eqref{eq:semidiscretization} is continuously differentiable in a neighborhood of $K^N$ and continuous in time, the Picard-Lindel{\"o}f theorem yields a unique local solution $\unum$, which can be continued as long as the states remain in $K^N$.
  Define
  \begin{equation}
    T_h
    =
    \sup \bigl\{ t \in [0, T] \bigm| \unum \text{ exists on } [0, t] \text{ and } \max_i \, \abs{e_i(\sigma)} \le \delta \text{ for all } \sigma \in [0, t] \bigr\},
  \end{equation}
  which is positive for small $h$:
  indeed, the coercivity \eqref{eq:coercivity} and the quasi-uniformity of Assumption~\ref{asm:sbp}~\ref{item:sbp-norm} give
  \begin{equation}
    \max_i \abs{e_i(0)} \le (c_M h^d)^{-1/2} \sqrt{2/c_U} \, E_h(0)^{1/2} \le C h^{p - d/2} < \delta
  \end{equation}
  for small $h$ by Assumptions~\ref{asm:initial-data} and \ref{asm:order}, and the constraint persists on a short time interval by continuity.
  On $[0, T_h)$, all states satisfy $\unum_i(t) \in K$, so the results of the preceding subsections apply.

  \emph{Gronwall estimate.}
  Combining the identity \eqref{eq:sharp}, Assumption~\ref{asm:remainder} (or Lemma~\ref{lem:bandwidth}), Lemma~\ref{lem:source-bound}, the Hessian bound \eqref{eq:hessian-bounds}, the Cauchy-Schwarz inequality, Lemma~\ref{lem:truncation}, and the coercivity \eqref{eq:coercivity} yields, on $[0, T_h)$,
  \begin{equation}
  \label{eq:gronwall-differential}
    \frac{\dif}{\dif t} E_h
    \le
    (C_R + C_{\mathrm{src}}) E_h + C_U \|\vec{e}\|_M \norm{\vec{\tau}}_M
    \le
    (C_R + C_{\mathrm{src}}) E_h + C_1 h^p E_h^{1/2},
  \end{equation}
  with $C_{\mathrm{src}} = (C_w / c_U) \norm{s}_{L^\infty}$ and $C_1 = C_U C (C_M C_N)^{1/2} \sqrt{2 / c_U}$.
  We set $C_2 = \max\{1, C_R + C_{\mathrm{src}}\}$;
  since $C_2 \ge C_R + C_{\mathrm{src}}$, the inequality \eqref{eq:gronwall-differential} remains valid with $C_2$ in place of $C_R + C_{\mathrm{src}}$, and $C_2 \ge 1$ avoids a division by zero in \eqref{eq:gronwall-solved} below in the borderline case of vanishing Gronwall constants.
  To avoid differentiating $E_h^{1/2}$ at zeros of $E_h$, let $\epsilon > 0$ and set $\phi_\epsilon = (E_h + \epsilon)^{1/2}$.
  Then, \eqref{eq:gronwall-differential} implies
  \begin{equation}
    \frac{\dif}{\dif t} \phi_\epsilon
    =
    \frac{1}{2 \phi_\epsilon} \frac{\dif}{\dif t} E_h
    \le
    \frac{C_2}{2} \phi_\epsilon + \frac{C_1}{2} h^p,
    \quad \text{i.e.,} \quad
    \frac{\dif}{\dif t} \bigl( \e^{-C_2 t / 2} \phi_\epsilon \bigr)
    \le
    \frac{C_1}{2} h^p \, \e^{-C_2 t / 2}.
  \end{equation}
  Integrating the right-hand inequality --- the integrating-factor form of the differential Gronwall inequality \cite[Appendix~B.2]{evans2010partial} --- over $[0, t]$, multiplying by $\e^{C_2 t / 2}$, and letting $\epsilon \to 0$ gives, for $t \in [0, T_h)$,
  \begin{equation}
  \label{eq:gronwall-solved}
    E_h(t)^{1/2}
    \le
    \e^{C_2 t / 2} E_h(0)^{1/2}
    + \frac{C_1}{C_2} \bigl( \e^{C_2 t / 2} - 1 \bigr) h^p
    \le
    C_G h^p,
  \end{equation}
  with $C_G = \e^{C_2 T / 2} \bigl( C_0 + C_1 / C_2 \bigr)$ since $E_h(0)^{1/2} \le C_0 h^p$ by Assumption~\ref{asm:initial-data}.
  By the coercivity \eqref{eq:coercivity}, this proves the error estimate \eqref{eq:main-estimate} with $C_T = \sqrt{2 / c_U} \, C_G$ on $[0, T_h)$.

  \emph{Bootstrap.}
  Together with the error estimate \eqref{eq:main-estimate}, quasi-uniformity of Assumption~\ref{asm:sbp}~\ref{item:sbp-norm} yields the inverse estimate
  \begin{equation}
  \label{eq:inverse-estimate}
    \max_i \, \abs{e_i(t)}
    \le
    (c_M h^d)^{-1/2} \norm{\vec{e}(t)}_M
    \le
    c_M^{-1/2} C_T \, h^{p - d/2},
    \qquad t \in [0, T_h).
  \end{equation}
  Since $p > d/2$ by Assumption~\ref{asm:order}, we may choose $h_0 > 0$ such that
  \begin{equation}
    c_M^{-1/2} C_T h_0^{p - d/2} \le \delta / 2.
  \end{equation}
  For $h < h_0$, the errors satisfy $\max_i \abs{e_i(t)} \le \delta / 2$ on $[0, T_h)$, so all states stay in the compact $\delta/2$-neighborhood of the range of $\usol$, which lies in the interior of $K^N$ with a positive margin.
  Hence, the solution cannot cease to exist at $T_h$, and by continuity the defining constraint $\max_i |e_i| \le \delta$ of $T_h$ is not active at $t = T_h$ either.
  If $T_h < T$, both properties would allow the solution to be continued beyond $T_h$ within the constraint, contradicting the maximality of $T_h$.
  Therefore, $T_h = T$, and the error estimate \eqref{eq:main-estimate} holds on all of $[0, T]$.
\end{proof}

\subsection{On the order condition}
\label{ssec:order-condition}

The condition $p > d/2$ of Assumption~\ref{asm:order} deserves a closer look, since one does not observe any need for increased orders of accuracy in several space dimensions in practice.
Its role in the proof is confined to a single step:
the inverse estimate \eqref{eq:inverse-estimate} converts the $M$-norm bound into control of the maximum nodal error $\max_i |e_i|$ that keeps the states in $K$;
here and below, $L^\infty$ control refers to this maximum over the nodes, not to the $L^\infty$ norm of a function defined on all of $\Omega$.
Nothing else in the proof uses it, and for $d \in \{2, 3\}$ it excludes only first-order methods.
A similar restriction to sufficiently high polynomial degrees in several space dimensions appears in the error analysis of discontinuous Galerkin methods with quadrature for scalar conservation laws, where it is likewise needed only for the a~priori assumption on the numerical solution and does not seem to be necessary in practice \cite[Remark~3.2]{huang2016error}.
The following two propositions substantiate that the condition is an artifact of the localization to $K$ rather than of the entropy estimate.

\begin{proposition}[Globally bounded data]
\label{prop:global-bounds}
  Let the assumptions of Theorem~\ref{thm:convergence} hold with $\mathcal{A} = \R^n$ and with the following global versions:
  the fluxes $\fnum_m$ are smooth on $\R^n \times \R^n$, the bounds \eqref{eq:hessian-bounds} and \eqref{eq:remainder-bound} and, if a source is present, the constant $C_w$ hold globally, and the structural bound \eqref{eq:remainder-hypothesis} of Assumption~\ref{asm:remainder} holds for all $\unum_i \in \R^n$ rather than only for $\unum_i \in K$.
  Then, the conclusion of Theorem~\ref{thm:convergence} holds for every $p \ge 1$ in every dimension, without Assumption~\ref{asm:order}.
\end{proposition}

The global form of \eqref{eq:remainder-hypothesis} is not automatic;
it is again implied by the fixed local stencil of Assumption~\ref{asm:bandwidth}, since the proof of Lemma~\ref{lem:bandwidth} uses the states only through the global versions of \eqref{eq:remainder-bound} and \eqref{eq:coercivity} and therefore returns the same constant $C_R$ without localization.
The Lipschitz constant $C_L$ of that lemma needs no global version:
it involves only the smooth solution, whose range is compact, and is therefore finite as soon as $w = U'$ is continuously differentiable.

\begin{proof}
  All estimates of this section hold on the maximal existence interval without any localization, since the constants are global and, by hypothesis, so is the structural bound \eqref{eq:remainder-hypothesis}.
  It remains to exclude finite-time blow-up for fixed $h$:
  the Gronwall estimate \eqref{eq:gronwall-solved} bounds $\|\vec{e}\|_M$, hence $\max_i |\unum_i| \le \max_i |\usol_i| + (c_M h^d)^{-1/2} \|\vec{e}\|_M$ is bounded on any bounded existence interval for fixed $h$, and solutions of ODEs with locally Lipschitz right-hand side can be continued while they remain bounded.
  Thus, the solution exists on $[0, T]$ and the estimate applies there.
\end{proof}

The model case for Proposition~\ref{prop:global-bounds} is a quadratic entropy with globally Lipschitz flux derivatives.
This covers the class treated in \cite{worku2026convergence}, symmetric hyperbolic systems with the quadratic entropy and fluxes with globally bounded second derivatives:
the latter give the global form of \eqref{eq:remainder-bound} for the flux \eqref{eq:symmetric-flux}, and $U'' = \I$ supplies the remaining global constants.
On that class, Theorem~\ref{thm:convergence} therefore holds for every $p \ge 1$ in every space dimension and does not use the homogeneity of the fluxes (Remark~\ref{rem:no-homogeneity}), whereas the convergence proof of \cite{worku2026convergence} requires $p > 1 + d/2$.

Of the two properties a quadratic entropy provides --- the global Hessian bounds \eqref{eq:hessian-bounds} with $c_U = C_U = 1$ and a vanishing third-derivative constant $C_w$ --- only the first is essential;
the second matters only in the presence of a source, since the source contribution \eqref{eq:src-definition} vanishes identically for $s \equiv 0$.
For source-free problems, Proposition~\ref{prop:global-bounds} thus covers every entropy that is uniformly convex with bounded Hessian on all of $\R^n$, which need not be quadratic:
for instance, $U'' = 1 + \frac{1}{2} \sin(u^2)$ stays in $[\frac{1}{2}, \frac{3}{2}]$ while $U'''$ is unbounded.
We emphasize that the proposition does not apply to the shallow water or compressible Euler equations, whose admissible sets are proper subsets of $\R^n$ and for which the localization away from vacuum is essential.

\begin{proposition}[Conditional a~posteriori version]
\label{prop:aposteriori}
  Let the assumptions of Theorem~\ref{thm:convergence} except Assumption~\ref{asm:order} hold.
  If, for some $h$, the solution of \eqref{eq:semidiscretization} exists on $[0, T]$ and all states $\unum_i(t)$ remain in $K$, then the error estimate \eqref{eq:main-estimate} holds for this solution.
\end{proposition}

\begin{proof}
  The Gronwall estimate in the proof of Theorem~\ref{thm:convergence} uses Assumption~\ref{asm:order} only through the bootstrap that guarantees $\unum_i(t) \in K$;
  under the stated hypothesis, this containment is given, and \eqref{eq:gronwall-solved} applies on $[0, T]$.
\end{proof}

Proposition~\ref{prop:aposteriori} certifies the $M$-norm error estimate for computed solutions that are verified to remain in a fixed compact convex set $K$, including first-order methods in several space dimensions.
For a mesh sequence, this yields convergence provided the checked set and its margin are uniform in $h$ and $K$ contains both the exact range and the numerical states.
For the shallow water or compressible Euler equations, verifying positivity alone is not sufficient;
the check must provide a uniform positive margin and bounded states, so that the constants entering the estimate remain controlled.
The proposition complements the a~priori theorem;
it neither proves that all sufficiently fine computations stay admissible, nor does it provide an $L^\infty$ convergence rate when $p \le d/2$.
Conditional a~priori estimates of this flavor, in which an admissibility or boundedness assumption on the computed solution replaces an unconditional bound, have recently also been developed for finite volume and Runge-Kutta discontinuous Galerkin methods with abstract limiting \cite{leotta2025conditional}.

The hypothesis of Proposition~\ref{prop:aposteriori}, that the numerical solution remains in a fixed compact set of admissible states, is precisely the standing assumption of the relative-entropy error analysis of Jovanovi{\'c} and Rohde \cite{jovanovic2006error}, who require the finite volume approximation to be uniformly bounded in $L^\infty$ --- obtained from a discrete maximum principle for monotone or weakly coupled schemes, and assumed otherwise.
Under this assumption their analysis is unconditional in the space dimension and covers first-order schemes with $p = 1$ for $d \ge 2$;
the order condition $p > d/2$ does not appear because the $L^\infty$ bound is assumed rather than derived from the $M$-norm estimate through the inverse estimate \eqref{eq:inverse-estimate}.
Both analyses share the Dafermos relative-entropy mechanism but differ in the discretization framework:
\cite{jovanovic2006error} treat first-order, \emph{dissipative} finite volume schemes satisfying a discrete entropy \emph{inequality}, for which the numerical viscosity supplies the entropy dissipation bound entering their estimate but limits the rate to $\O(\sqrt{h})$;
the schemes analyzed here are high-order and entropy \emph{conservative}, so that the only residual is the truncation error and the rate is the full order $\O(h^p)$, at the cost of providing no a~priori $L^\infty$ bound of their own.

For operators whose truncation errors admit an expansion with smooth coefficients --- such as finite-difference operators on uniform grids --- a classical device of Strang \cite{strang1964accurate} suggests itself:
compare the numerical solution not with $\vsol$ but with a corrected state $\usol + h^p z_1 + \dots + h^{p + \ell - 1} z_\ell$, where the correctors $z_k$ solve linear hyperbolic systems obtained by linearization around $\usol$, which are symmetrizable by the entropy Hessian $U''(\usol)$.
The relative-entropy machinery applied to the corrected comparison state then yields an $M$-norm error of order $h^{p + \ell}$, and the inverse estimate gives $L^\infty$ control whenever $p + \ell > d/2$, for every $p \ge 1$, at unchanged final rate $\O(h^p)$ and at the price of additional smoothness of $\usol$.
Carrying this out rigorously requires a precise residual expansion of the flux differencing around the corrected state, the full corrector hierarchy including all interaction terms, and admissibility of the corrected state;
we have not carried out this program and present it only as a possible route for future work.

\section{Applications}
\label{sec:applications}

We now verify the assumptions of Theorem~\ref{thm:convergence} for several classes of conservation laws.
The operator assumptions \ref{asm:sbp} and \ref{asm:bandwidth} concern only the discretization and are satisfied by the families listed in Remark~\ref{rem:operator-classes};
what remains to be checked for each system is the entropy pair (Assumption~\ref{asm:entropy-pair}), the choice of the compact convex set $K$ (Assumption~\ref{asm:smooth-solution}), and the existence of sufficiently smooth entropy-conservative fluxes (Assumption~\ref{asm:ec-flux}).
All corollaries below include arbitrary bounded continuous source terms $s(t, x)$ in the sense of Assumption~\ref{asm:source};
in particular, they cover convergence tests by the method of manufactured solutions \cite{salari2000code}.

\subsection{Scalar conservation laws}
\label{ssec:scalar}

For scalar conservation laws, $n = 1$ and $\mathcal{A} = \R$, every strictly convex function $U$ is an entropy, and the entropy-conservation condition \eqref{eq:ec-condition} determines each directional two-point flux uniquely for $a \ne b$ as the jump quotient \cite{tadmor1987numerical,tadmor2003entropy}
\begin{equation}
\label{eq:scalar-ec-flux}
  \fnum_m(a, b)
  =
  \frac{\psi_m(b) - \psi_m(a)}{w(b) - w(a)},
  \qquad
  \fnum_m(a, a) = f_m(a).
\end{equation}
Writing $u(\cdot) = (U')^{-1}$ for the inverse of the entropy variables and using $(\psi_m)' = U'' f_m$ from \eqref{eq:potential-gradient}, the flux \eqref{eq:scalar-ec-flux} has the integral representation \cite[Equation~(4.6a)]{tadmor1987numerical}
\begin{equation}
\label{eq:scalar-integral-flux}
  \fnum_m(a, b)
  =
  \int_0^1 f_m\Bigl( u\bigl( w(a) + \sigma \, (w(b) - w(a)) \bigr) \Bigr) \dif \sigma,
\end{equation}
i.e., it is the average of $f_m$ along the straight path in entropy variables.
In particular, $\fnum_m$ is $C^{p+1}$ on $\mathcal{A} \times \mathcal{A}$ whenever $f_m \in C^{p+1}$ and $U \in C^{p+2}$ with $U'' > 0$, since then $w = U'$ is a $C^{p+1}$ diffeomorphism onto its image and the entropy-variable segment in \eqref{eq:scalar-integral-flux} stays in that image.

\begin{corollary}[Scalar conservation laws]
\label{cor:scalar}
  Consider a scalar conservation law with fluxes $f_m \in C^{p+1}(\R)$, a source as in Assumption~\ref{asm:source}, and an entropy $U \in C^{p+2}(\R)$ with $U'' > 0$.
  Let $\usol \in C^{p+1}(\Omega \times [0, T])$ be a solution, let $K$ be a compact interval containing the range of $\usol$ in its interior, and let the operators satisfy Assumptions~\ref{asm:sbp}, \ref{asm:bandwidth}, and \ref{asm:order}.
  Then, the semidiscretization \eqref{eq:semidiscretization} with the fluxes \eqref{eq:scalar-ec-flux} and nodal initial data converges at rate $p$ in the sense of Theorem~\ref{thm:convergence}.
\end{corollary}

\begin{proof}
  Assumption~\ref{asm:entropy-pair} holds with $F_m(a) = \int^a U'(\sigma) f_m'(\sigma) \dif \sigma$.
  The interval $K$ is convex, and $U''$ is continuous and positive, so the Hessian bounds \eqref{eq:hessian-bounds} hold by compactness;
  likewise, $U \in C^{p+2} \subseteq C^3$ provides the constant $C_w$ for the source estimate.
  Assumption~\ref{asm:ec-flux} holds by the discussion above, and Assumption~\ref{asm:initial-data} holds with $E_h(0) = 0$.
  Apply Theorem~\ref{thm:convergence}.
\end{proof}

\subsection{Symmetric systems}
\label{ssec:symmetric}

Next, consider systems whose flux Jacobians $f_m'$ are all symmetric on a convex admissible set $\mathcal{A}$, so that $f_m = \chi_m'$ for potentials $\chi_m \in C^{p+2}(\mathcal{A})$.
Then, the quadratic entropy $U(a) = |a|^2 / 2$ forms an entropy pair with $F_m = a^T f_m(a) - \chi_m(a)$, the entropy variables are $w(a) = a$, and the entropy potentials \eqref{eq:entropy-potential} are $\psi_m = \chi_m$ up to irrelevant constants.
Tadmor's flux \eqref{eq:scalar-integral-flux} reduces to the line integrals
\begin{equation}
\label{eq:symmetric-flux}
  \fnum_m(a, b)
  =
  \int_0^1 f_m\bigl( a + \sigma \, (b - a) \bigr) \dif \sigma,
\end{equation}
which are symmetric (substitute $\sigma \mapsto 1 - \sigma$), consistent, of class $C^{p+1}$ whenever $f_m \in C^{p+1}$, and entropy-conservative by the fundamental theorem of calculus,
\begin{equation}
\begin{aligned}
  (b - a)^T \fnum_m(a, b)
  &=
  \int_0^1 \frac{\dif}{\dif \sigma} \, \chi_m\bigl( a + \sigma (b - a) \bigr) \dif \sigma
  \\
  &=
  \chi_m(b) - \chi_m(a)
  =
  \psi_m(b) - \psi_m(a).
\end{aligned}
\end{equation}

\begin{corollary}[Symmetric systems]
\label{cor:symmetric}
  Consider a system with symmetric Jacobians $f_m' \in C^{p}(\mathcal{A})$ on an open convex set $\mathcal{A} \subseteq \R^n$, a source as in Assumption~\ref{asm:source}, and the quadratic entropy.
  Let $\usol \in C^{p+1}(\Omega \times [0, T]; \mathcal{A})$ be a solution, let $K \subset \mathcal{A}$ be a compact convex set containing the range of $\usol$ in its interior, and let the operators satisfy Assumptions~\ref{asm:sbp}, \ref{asm:bandwidth}, and \ref{asm:order}.
  Then, the semidiscretization \eqref{eq:semidiscretization} with the fluxes \eqref{eq:symmetric-flux} and nodal initial data converges at rate $p$;
  the Hessian bounds hold with $c_U = C_U = 1$, and the source contribution vanishes identically since $U''' = 0$.
\end{corollary}

\begin{remark}[No homogeneity required]
\label{rem:no-homogeneity}
  Corollary~\ref{cor:symmetric} covers the class of systems analyzed in \cite{worku2026convergence}, whose energy-method analysis additionally uses homogeneity properties of the fluxes $f_m$ for the split-form structure.
  The present relative-entropy argument does not require any homogeneity.
  If the second derivatives of the fluxes are bounded globally, as assumed there, the order condition can be dropped as well;
  see the discussion following Proposition~\ref{prop:global-bounds}.
\end{remark}

For general systems with a strictly convex entropy, Tadmor's flux \eqref{eq:scalar-integral-flux} interpreted along straight paths in entropy variables provides entropy-conservative fluxes in great generality \cite{tadmor1987numerical,lefloch2002fully};
its smoothness on $\mathcal{A} \times \mathcal{A}$ follows from the integral representation as above, provided the entropy-variable path stays in the domain where the entropy Hessian is invertible.
In practice, explicit algebraic fluxes are preferred for efficiency \cite{ismail2009affordable};
the following two subsections verify Assumption~\ref{asm:ec-flux} for such fluxes for the shallow water and compressible Euler equations.

\subsection{Shallow water equations}
\label{ssec:shallow-water}

The shallow water equations without bottom topography, in $d \in \{1, 2\}$ horizontal dimensions, read
\begin{equation}
\label{eq:swe}
  \partial_t
  \begin{pmatrix} H \\ H v \end{pmatrix}
  + \sum_{m=1}^{d} \partial_m
  \begin{pmatrix} H v_m \\ H v_m v + \frac{g}{2} H^2 e_m \end{pmatrix}
  =
  0,
\end{equation}
where $H$ is the water height\footnote{We use $H$ instead of the usual $h$ to avoid confusion with the typical mesh size $h$ introduced in Assumption~\ref{asm:sbp}.}, $v \in \R^d$ the velocity with components $v_k$, $g > 0$ the gravitational acceleration, and $e_m$ the $m$-th unit vector.
The admissible set $\mathcal{A} = \{ (H, H v) \mid H > 0 \}$ is an open half-space and thus convex.
The total energy and the associated fluxes \cite{fjordholm2011well,gassner2016well,wintermeyer2017entropy,ranocha2017shallow},
\begin{equation}
  U = \frac{1}{2} H |v|^2 + \frac{g}{2} H^2,
  \qquad
  F_m = \Bigl( \frac{1}{2} H |v|^2 + g H^2 \Bigr) v_m,
\end{equation}
form an entropy pair with entropy variables and potentials
\begin{equation}
  w = \begin{pmatrix} g H - |v|^2 / 2 \\ v \end{pmatrix},
  \qquad
  \psi_m = \frac{g}{2} H^2 v_m.
\end{equation}
In the conserved variables $(H, q)$ with $q = H v$, the entropy $U = |q|^2 / (2 H) + g H^2 / 2$ has the Hessian
\begin{equation}
  U''
  =
  \frac{1}{H}
  \begin{pmatrix}
    g H + |v|^2 & - v^T \\
    - v & \operatorname{I}_d
  \end{pmatrix},
  \qquad v = \frac{q}{H},
\end{equation}
which is strictly positive definite on $\mathcal{A}$, since its quadratic form
\begin{equation}
  \begin{pmatrix} a \\ b \end{pmatrix}^T U'' \begin{pmatrix} a \\ b \end{pmatrix}
  =
  g \, a^2 + \frac{1}{H} \, |a v - b|^2
\end{equation}
is positive for every $(a, b) \in \R \times \R^d$ with $(a, b) \ne 0$, using $g > 0$ and $H > 0$.
The bounds \eqref{eq:hessian-bounds} and the constant $C_w$ thus hold on every compact $K \subset \mathcal{A}$ by compactness;
their possible deterioration towards vacuum is through the positive lower height bound $H_{\min} = \min_{K} H > 0$.
As in Remark~\ref{rem:convexity-K}, $K$ can be chosen as the closed convex hull of the range of the smooth solution plus a small margin.

Fjordholm, Mishra, and Tadmor \cite{fjordholm2009energy,fjordholm2011well} proposed the entropy-conservative fluxes
\begin{equation}
\label{eq:fmt-flux}
  \fnum_m(a, b)
  =
  \begin{pmatrix}
    \mean{H} \mean{v_m} \\
    \mean{H} \mean{v_m} \mean{v} + \frac{g}{2} \mean{H^2} \, e_m
  \end{pmatrix},
\end{equation}
where $\mean{\cdot}$ denotes the arithmetic mean of the values at the states $a$ and $b$, applied componentwise to $v$.
The fluxes \eqref{eq:fmt-flux} are symmetric and consistent by inspection, satisfy the entropy-conservation condition \eqref{eq:ec-condition} by a direct computation \cite{fjordholm2009energy,fjordholm2011well}, and are polynomial in $(H, v)$ of both states;
since $(H, q) \mapsto (H, v)$ is smooth on $\mathcal{A} = \{H > 0\}$, they are $C^\infty$ on $\mathcal{A} \times \mathcal{A}$.
The same holds for other entropy-conservative fluxes \cite{gassner2016well,wintermeyer2017entropy,ranocha2017shallow}.

\begin{corollary}[Shallow water equations]
\label{cor:shallow-water}
  Let $u = (H, H v) \in C^{p+1}(\Omega \times [0, T])$ be a solution of \eqref{eq:swe}, or of \eqref{eq:swe} with an added source term satisfying Assumption~\ref{asm:source}, with $H \ge 2 H_{\min} > 0$, and let the operators satisfy Assumptions~\ref{asm:sbp}, \ref{asm:bandwidth}, and \ref{asm:order}.
  Then, the semidiscretization \eqref{eq:semidiscretization} with the fluxes \eqref{eq:fmt-flux} of \cite{fjordholm2009energy,fjordholm2011well} (or similar fluxes of \cite{gassner2016well,wintermeyer2017entropy,ranocha2017shallow}) and nodal initial data converges at rate $p$;
  among the data of Theorem~\ref{thm:convergence}, the constants deteriorate towards vacuum only through the lower height bound $H_{\min}$.
  In particular, the discrete water heights of the semidiscretization remain positive on $[0, T]$ for all sufficiently fine meshes, without any positivity limiter;
  see the discussion after Corollary~\ref{cor:euler} for the limitations of this statement.
\end{corollary}

\subsection{Compressible Euler equations}
\label{ssec:euler}

The compressible Euler equations of an ideal gas in $d$ space dimensions read
\begin{equation}
\label{eq:euler}
  \partial_t
  \begin{pmatrix} \rho \\ \rho v \\ \rho e \end{pmatrix}
  + \sum_{m=1}^{d} \partial_m
  \begin{pmatrix} \rho v_m \\ \rho v_m v + p \, e_m \\ (\rho e + p) v_m \end{pmatrix}
  =
  0,
  \qquad
  p = (\gamma - 1) \Bigl( \rho e - \frac{1}{2} \rho |v|^2 \Bigr),
\end{equation}
with density $\rho$, velocity $v \in \R^d$, total energy density $\rho e$, pressure $p$, and ratio of specific heats $\gamma > 1$.
First, we state a well-known result.

\begin{lemma}[Admissible set]
\label{lem:euler-admissible}
  The set of admissible states $\mathcal{A} = \{ u = (\rho, \rho v, \rho e) \mid \rho > 0, \ p(u) > 0 \}$ is open and convex.
\end{lemma}

The convexity of the admissible set is classical and underlies the positivity-preserving schemes of Perthame and Shu \cite{perthame1996positivity} and Zhang and Shu \cite{zhang2010positivity}; we include the short argument for completeness.

\begin{proof}
  The set $\{\rho > 0\}$ is an open half-space.
  The map $u \mapsto \rho e - |\rho v|^2 / (2 \rho)$ is concave on $\{ \rho > 0 \}$, since $|\rho v|^2 / (2 \rho)$ is a perspective function of the convex map $q \mapsto |q|^2 / 2$ and hence jointly convex \cite[Section~3.2.6]{boyd2004convex}.
  Thus, $\{ p(u) > 0 \}$ is a superlevel set of a concave function and convex, and $\mathcal{A}$ is the intersection of two convex sets.
\end{proof}

We use the physical specific entropy\footnote{We use $\varsigma$ instead of $s$ to avoid confusion with the source term $s(t,x)$ of \eqref{eq:conservation-law}.} $\varsigma = \log(p \rho^{-\gamma})$ and the entropy pair \cite{ismail2009affordable,chandrashekar2013kinetic,ranocha2018comparison}
\begin{equation}
\label{eq:euler-entropy}
  U = - \frac{\rho \varsigma}{\gamma - 1},
  \qquad
  F_m = - \frac{\rho v_m \varsigma}{\gamma - 1},
\end{equation}
which is the canonical member of the Harten family \cite{harten1983symmetric}.
The entropy variables and the entropy potentials are
\begin{equation}
\label{eq:euler-entropy-variables}
  w
  =
  \begin{pmatrix}
    \frac{\gamma - \varsigma}{\gamma - 1} - \beta |v|^2 \\
    2 \beta v \\
    - 2 \beta
  \end{pmatrix},
  \qquad
  \beta = \frac{\rho}{2 p},
  \qquad
  \psi_m = \rho v_m,
\end{equation}
i.e., the entropy potentials are the momentum components \cite[Remark~3]{tadmor2003entropy};
both formulas follow from direct computations.
$U$ is strictly convex on $\mathcal{A}$ \cite{harten1983symmetric} and \cite[Section~3.3.5]{dafermos2016hyperbolic}, so the bounds \eqref{eq:hessian-bounds} and the constant $C_w$ are finite on every compact subset of $\mathcal{A}$;
explicit expressions for $c_U$ and $C_U$ in the one-dimensional case are derived in Appendix~\ref{app:euler-constants}, and the same compactness argument applies verbatim for $d \ge 2$.

Given a smooth solution with $\rho \ge 2 \rho_{\min} > 0$ and $p \ge 2 p_{\min} > 0$ on $\Omega \times [0, T]$, let $\mathcal{N}_\epsilon$ denote the union of the closed Euclidean balls of radius $\epsilon$ around the points of the range of $\usol$, and choose
\begin{equation}
\label{eq:euler-K}
  K
  =
  \overline{\operatorname{conv}}\bigl( \mathcal{N}_\epsilon \bigr)
  \cap
  \{ \rho \ge \rho_{\min} \}
  \cap
  \{ p(u) \ge p_{\min} \},
\end{equation}
where $\epsilon > 0$ is small enough that $\mathcal{N}_\epsilon \subseteq \{ \rho \ge \rho_{\min} \} \cap \{ p(u) \ge p_{\min} \}$;
this is possible since $\rho$ and $p$ are continuous and satisfy the twofold bounds on the compact range.
The set $K$ is compact, convex (all three sets are convex, the last one by the concavity used in Lemma~\ref{lem:euler-admissible}), contained in $\mathcal{A}$, and contains the $\epsilon$-neighborhood $\mathcal{N}_\epsilon$ of the range of $\usol$.
We emphasize that a box in the primitive variables $(\rho, v, p)$ is in general \emph{not} convex in the conserved variables, since $p(u)$ is concave;
the construction \eqref{eq:euler-K} avoids this pitfall.

As a concrete entropy-conservative flux family for the entropy pair \eqref{eq:euler-entropy}, we use the fluxes of \cite{ranocha2018comparison,ranocha2021preventing}.
They are built from arithmetic means $\mean{\cdot}$, the logarithmic mean
\begin{equation}
\label{eq:logmean}
  \logmean{a}
  =
  \frac{a_L - a_R}{\log a_L - \log a_R},
  \qquad
  \logmean{a}\big|_{a_L = a_R} = a_L,
\end{equation}
and the product mean $\prodmean{a}{b} = ( a_L b_R + a_R b_L ) / 2$, and read, for the direction $m$ and momentum components $k \in \{1, \dots, d\}$,
\begin{equation}
\label{eq:flux-ranocha}
\begin{aligned}
  \fnum_{\rho} &= \logmean{\rho} \mean{v_m},
  \\
  \fnum_{\rho v_k} &= \mean{v_k} \fnum_\rho + \delta_{km} \mean{p},
  \\
  \fnum_{\rho e} &= \frac{1}{2} \Bigl( \sum_{k=1}^{d} \prodmean{v_k}{v_k} \Bigr) \fnum_\rho + \frac{1}{(\gamma - 1) \logmean{\rho / p}} \fnum_\rho + \prodmean{p}{v_m},
\end{aligned}
\end{equation}
which are also kinetic-energy-preserving and pressure-equilibrium-preserving \cite{ranocha2021preventing}.
Other classical entropy-conservative Euler fluxes include those of Chandrashekar \cite{chandrashekar2013kinetic} and of Ismail and Roe \cite{ismail2009affordable}; see \cite{ranocha2018comparison} for a comparison.
They are built from similar means of positive quantities and are covered by the analysis in the same way, so we work with \eqref{eq:flux-ranocha} for concreteness;
the numerical experiments of Section~\ref{sec:numerics} use it as well.

\begin{lemma}[Flux properties]
\label{lem:euler-fluxes}
  On $\mathcal{A} \times \mathcal{A}$, the flux family \eqref{eq:flux-ranocha} satisfies Assumption~\ref{asm:ec-flux} for each direction $m$, with $C^{p+1}$ replaced by $C^\infty$.
\end{lemma}

\begin{proof}
  Symmetry holds since all involved means are symmetric.
  For consistency, all means reduce to the point value at coinciding states, and short computations give $\fnum_m(a, a) = f_m(a)$;
  e.g., $\fnum_{\rho e}(a, a) = \rho v_m |v|^2 / 2 + p v_m / (\gamma - 1) + p v_m = (\rho e + p) v_m$.
  The entropy-conservation condition \eqref{eq:ec-condition} with $\psi_m = \rho v_m$ is verified in \cite{ranocha2018comparison,ranocha2021preventing}.
  Concerning smoothness, on $\mathcal{A}$ the quantities $\rho$ and $\rho / p$ are positive, and the primitive variables are smooth functions of the conserved variables.
  The logarithmic mean is real-analytic on $(0, \infty)^2$ by the integral representation \eqref{eq:inverse-logmean} of Appendix~\ref{app:euler-constants}, arithmetic and product means are polynomial, and the denominator $\logmean{\rho / p}$ is positive on $\mathcal{A} \times \mathcal{A}$.
  Hence, the flux family is $C^\infty$ on $\mathcal{A} \times \mathcal{A}$.
\end{proof}

\begin{corollary}[Compressible Euler equations]
\label{cor:euler}
  Let $\usol \in C^{p+1}(\Omega \times [0, T]; \mathcal{A})$ be a solution of \eqref{eq:euler}, or of \eqref{eq:euler} with an added source term satisfying Assumption~\ref{asm:source}, with $\rho \ge 2 \rho_{\min} > 0$ and $p \ge 2 p_{\min} > 0$, let $K$ be given by \eqref{eq:euler-K}, and let the operators satisfy Assumptions~\ref{asm:sbp}, \ref{asm:bandwidth}, and \ref{asm:order}.
  Then, the semidiscretization \eqref{eq:semidiscretization} with the fluxes \eqref{eq:flux-ranocha} and nodal initial data converges at rate $p$ in the sense of Theorem~\ref{thm:convergence}.
  The threshold $h_0$ and the constant $C_T$ depend on $\rho_{\min}$ and $p_{\min}$ through the Hessian bounds $c_U, C_U$, the constant $C_w$, the flux constant $C_S$, and the Lipschitz constant $C_L$;
  some explicit (sufficient, non-sharp) expressions for $d = 1$ are given in Appendix~\ref{app:euler-constants}.
\end{corollary}

Since $K \subset \mathcal{A}$, the conclusion of Theorem~\ref{thm:convergence} includes that the discrete states keep positive density and pressure on $[0, T]$ for all $h < h_0$.
Thus, for smooth solutions away from vacuum, the exact solution of the semidiscretization \eqref{eq:semidiscretization} does not need a positivity limiter on sufficiently fine meshes.
This statement has two limitations.
First, it concerns the semidiscretization;
a fully discrete scheme additionally requires a time integration method and step sizes such that the fully discrete solution stays close to the semidiscrete one.
Second, the available bounds for the threshold $h_0$ are not sharp (Appendix~\ref{app:euler-constants}) and depend on norms of the exact solution, so that $h_0$ cannot be checked easily a~priori on a given mesh;
on practical meshes, which may be coarser than $h_0$, positivity limiters can still be required, as illustrated in Section~\ref{sec:numerics-diagnostics}.

More generally, $U = - \rho \, \eta(\varsigma) / (\gamma - 1)$ is a strictly convex entropy for every $\eta$ with $\eta' > 0$ and $\eta' - \gamma \eta'' > 0$, the latter being Harten's condition $\eta'' / \eta' < 1 / \gamma$ \cite[Equation~(2.8c)]{harten1983symmetric}.
The analysis only requires \emph{some} strictly convex entropy pair with smooth entropy-conservative fluxes;
we use the canonical member \eqref{eq:euler-entropy}, for which the fluxes \eqref{eq:flux-ranocha} are entropy-conservative.
Alternative fluxes are studied in \cite{ranocha2022note}.

\section{Numerical experiments}
\label{sec:numerics}

We now illustrate the theory with convergence experiments that span the operator classes, equations, entropy pairs, and space dimensions covered by the analysis.
Throughout, the observed rates confirm the central prediction of Theorem~\ref{thm:convergence}: the discrete $M$-norm error decays at least at the order $p$ of the underlying diagonal-norm SBP operator.

\subsection{Discretizations}
\label{sec:numerics-setup}

The methods are implemented in Julia \cite{bezanson2017julia} with the \texttt{DGMulti} solver of Trixi.jl \cite{schlottkelakemper2020trixi,schlottkelakemper2021purely,ranocha2022adaptive} together with the SBP operators of SummationByPartsOperators.jl \cite{ranocha2021sbp} and the reference elements of StartUpDG.jl \cite{chan2026startupdg}.
Time integration uses the fourth-order, nine-stage, low-storage, error-controlled Runge--Kutta method of \cite{ranocha2022optimized} implemented in the DifferentialEquations.jl ecosystem \cite{rackauckas2017differentialequations} with adaptive step-size control and equal absolute and relative tolerances chosen small enough that the temporal error lies well below the spatial error.
The diagnostics of Section~\ref{sec:numerics-diagnostics} additionally use ForwardDiff.jl \cite{revels2016forward} for automatic differentiation and Makie.jl \cite{danisch2021makie} for visualization.
All code to reproduce the numerical experiments is available online in the reproducibility repository \cite{ranocha2026convergenceRepro}.

All experiments use entropy-conservative flux differencing~\eqref{eq:semidiscretization} on periodic domains with an entropy-conservative two-point volume flux and the same entropy-conservative flux on element interfaces if multiple blocks/elements are present so that the semidiscretization is entropy conservative without any interface dissipation (Lemma~\ref{lem:entropy-conservation}).
They exercise four operator classes, all of which are diagonal-norm and fixed-local-stencil in the sense of Assumptions~\ref{asm:sbp} and \ref{asm:bandwidth} and hence are covered by Theorem~\ref{thm:convergence} via Lemma~\ref{lem:bandwidth}:
finite-difference SBP operators, used in their two refinement modes --- a single periodic block of orders $2$, $4$, and $6$ refined in the node count, and a multi-block coupling by simultaneous approximation terms (SBP-SAT) \cite{carpenter1994time} of the boundary-optimized operators of Mattsson, Almquist and van der Weide \cite{mattsson2018boundary} of interior orders $4$ and $6$ (with $11$ and $15$ nodes per block, respectively), refined in the number of blocks;
nodal discontinuous Galerkin spectral-element (DGSEM) operators on Legendre--Gauss--Lobatto nodes of polynomial degree $k$;
continuous-Galerkin spectral-element (CGSEM) operators of degree $k$,
and, in two space dimensions, multidimensional SBP operators on triangles \cite{hicken2016multidimensional,chen2017entropy}, which are genuinely multidimensional and non-tensor-product, specifically the diagonal-$E$ operators of degree $k$ built from the symmetric SBP quadrature rules with Gauss--Lobatto face nodes of \cite{wu2021high}.
We also consider a smoothly curved grid.

The discrete $L^2$ error is computed with the SBP mass-matrix quadrature and therefore coincides with the $M$-norm $\|\vnum - \vsol\|_M$ normalized by the factor $|\Omega|^{-1/2}$, i.e., with $|\Omega|^{-1/2} \|\vnum - \vsol\|_M$.
For systems, we report this error for the full state vector, in accordance with the vector-valued $M$-norm of Theorem~\ref{thm:convergence}.
The mesh size $h$ of Assumption~\ref{asm:sbp} is the node spacing of the assembled operator for the single periodic FD block, i.e., the length of the domain divided by the number $N_n$ of nodes per coordinate direction;
for all other classes --- the multi-block FD operators and the element-based DGSEM, CGSEM, and triangular operators --- the number of nodes per block or element is fixed under refinement, so $h$ is the block or element width, i.e., the length of the domain divided by the number $N_e$ of blocks or elements per coordinate direction.
Both are reported in the tables below; the symbol $N$ remains reserved for the total number of nodes of the assembled operator, as in Assumption~\ref{asm:sbp}.
The same convention is used in two space dimensions, where the meshes are uniform in each direction and the triangular meshes are obtained by splitting the quadrilaterals of such a mesh.

The order $p$ predicted by Theorem~\ref{thm:convergence} is the one of the pointwise accuracy estimate~\eqref{eq:accuracy} for the assembled operator, which differs from the nominal order of the building blocks in two of the classes above.
For the single periodic FD block, every row uses the interior stencil, so $p$ is the order of the operator, i.e., $2$, $4$, or $6$.
For the element-based operators of degree $k$ --- DGSEM, CGSEM, and the triangular SBP operators --- the assembled operator differentiates polynomials of degree $k$ exactly, so $p = k$.
For the multi-block FD operators, however, $p$ is limited by the closures at the block boundaries:
if the closures of a diagonal-norm SBP operator are exact for polynomials of degree $\tau$, its interior order is necessarily at least $2\tau$ \cite{kreiss1974finite}, and this holds irrespective of the node distribution near the boundary \cite[Theorem~9]{linders2018order};
the standard operator families are constructed with exactly this pairing \cite{strand1994summation,mattsson2018boundary}.
Hence, the estimate \eqref{eq:accuracy}, being uniform in the node index, holds with $p$ equal to half the interior order rather than the interior order itself.
The restriction is a consequence of the diagonal norm required by Assumption~\ref{asm:sbp}~\ref{item:sbp-norm};
with a full (block) norm, closures of one order below the interior order are attainable \cite{kreiss1977existence}, at the price of losing the diagonal structure on which the analysis of Section~\ref{sec:proof} rests.
For the two multi-block operators used below, of interior orders $4$ and $6$, Theorem~\ref{thm:convergence} therefore only guarantees the rates $p = 2$ and $p = 3$.

\subsection{A scalar conservation law}
\label{sec:numerics-burgers}

We first take the inviscid Burgers equation $\partial_t u + \partial_x u^2 / 2 = s$ with its quadratic entropy $U(u) = u^2/2$ and the manufactured smooth solution $u(t,x) = 2 + \sin(2\pi(x-t))$ on the periodic interval $\Omega = [0,1]$, driven by the corresponding state-independent source $s$;
this solution and source constitute the convergence test implemented in Trixi.jl.
All runs of this section use the final time $T = 0.5$.
Table~\ref{tab:burgers-fdsbp} shows that the single periodic finite-difference SBP block reproduces its design orders $2$, $4$, and $6$ cleanly under node refinement.

\begin{table}[!htb]
  \centering
  \caption{Burgers 1D, EC flux differencing, single periodic FD-SBP block of orders 2, 4, 6, refined in the node count. $L^2$ ($M$-norm) error and experimental order of convergence (EOC).}
  \label{tab:burgers-fdsbp}
  \begin{tabular}{cccccccc}
    \toprule
    $h$ & $N_n$ & \multicolumn{2}{c}{$\text{order 2}$} & \multicolumn{2}{c}{$\text{order 4}$} & \multicolumn{2}{c}{$\text{order 6}$} \\
     &  & $L^2$ & EOC & $L^2$ & EOC & $L^2$ & EOC \\
    \midrule
    3.125e-02 & 32 & 2.432e-02 & --- & 3.199e-04 & --- & 1.135e-05 & --- \\
    1.562e-02 & 64 & 6.147e-03 & 1.98 & 2.111e-05 & 3.92 & 1.870e-07 & 5.92 \\
    7.812e-03 & 128 & 1.519e-03 & 2.02 & 1.321e-06 & 4.00 & 2.942e-09 & 5.99 \\
    3.906e-03 & 256 & 3.775e-04 & 2.01 & 8.259e-08 & 4.00 & 4.651e-11 & 5.98 \\
    \bottomrule
  \end{tabular}
\end{table}

Table~\ref{tab:burgers-fdsbp-multiblock} instead refines the number of coupled blocks at a fixed number of nodes per block.
The observed orders, $3$ and $4$, are neither the interior orders $4$ and $6$ --- these are attained only by the single periodic block, which has no closures --- nor the rates $p = 2$ and $p = 3$ guaranteed by Theorem~\ref{thm:convergence} for these operators:
they exceed the guaranteed rates by exactly one order.
This gain is what the classical linear theory predicts.
For first-order hyperbolic problems, boundary closures may be one order less accurate than the interior stencil without reducing the global order \cite{gustafsson1975convergence,svard2006order};
for energy-stable, nullspace-consistent, and nullspace-invariant schemes with an interior truncation error of order $2p$ and a boundary truncation error of order $p$, the convergence rate is $\min(2p, p+1) = p+1$ \cite{svard2019convergence}, which carries over to domains split into several blocks coupled across interfaces \cite{svard2020convergence} --- exactly the setting of this table.
The relative-entropy argument does not capture this extra order, since it uses the accuracy estimate \eqref{eq:accuracy} uniformly in the node index, where the closure rows dominate;
we return to this point in Section~\ref{sec:summary}.

\begin{table}[!htb]
  \centering
  \caption{Burgers 1D, EC flux differencing, multi-block FD-SBP operators of interior orders 4 and 6, refined in the number of blocks at fixed nodes per block ($L^2$ error and EOC).}
  \label{tab:burgers-fdsbp-multiblock}
  \begin{tabular}{cccccc}
    \toprule
    $h$ & $N_e$ & \multicolumn{2}{c}{$\text{interior order 4}$} & \multicolumn{2}{c}{$\text{interior order 6}$} \\
     &  & $L^2$ & EOC & $L^2$ & EOC \\
    \midrule
    1.250e-01 & 8 & 1.655e-05 & --- & 1.924e-07 & --- \\
    6.250e-02 & 16 & 1.633e-06 & 3.34 & 6.539e-09 & 4.88 \\
    3.125e-02 & 32 & 2.031e-07 & 3.01 & 4.020e-10 & 4.02 \\
    1.562e-02 & 64 & 2.567e-08 & 2.98 & 2.211e-11 & 4.18 \\
    \bottomrule
  \end{tabular}
\end{table}

The element-based DGSEM and CGSEM discretizations of degrees $k \in \{1, 2, 3, 4\}$ (Tables~\ref{tab:burgers-dgsem} and \ref{tab:burgers-cgsem}) converge at rates consistent with the operator order;
the observed values fluctuate between roughly $k$ and $k+1$.
Rigorous results on the influence of the numerical flux on the attainable rate for DG methods are available for \emph{linear} hyperbolic problems, where the parity of the degree matters:
for central fluxes on uniform meshes, the optimal rate $k+1$ holds for even $k$, while only the sub-optimal rate $k$ can be guaranteed for odd $k$, and on non-uniform meshes the rate drops to $k$ for both parities;
these estimates are sharp \cite{liu2021suboptimal}, and upwind-biased fluxes restore the optimal rate for every degree \cite{meng2016optimal}.
Our DGSEM rates on uniform meshes follow the same pattern --- close to $k$ for the odd degrees $k = 1, 3$ and close to $k+1$ for the even degrees $k = 2, 4$.
For CG discretizations of first-order hyperbolic problems, the classical results have a similar flavor:
the standard method does not attain the optimal rate for every approximation space --- Dupont's example loses one order for Hermite cubics \cite{dupont1973galerkin} --- and modified Galerkin schemes are needed to recover the optimal rate for all admissible spaces \cite{dendy1974two}.
We are not aware of a comparable theory predicting the rate-optimality and parity behavior for the nonlinear entropy-conservative discretizations studied here.

\begin{table}[!htb]
  \centering
  \caption{Burgers 1D, EC flux differencing, DGSEM of degrees $k = 1, 2, 3, 4$ ($L^2$ error and EOC).}
  \label{tab:burgers-dgsem}
  \begin{tabular}{cccccccccc}
    \toprule
    $h$ & $N_e$ & \multicolumn{2}{c}{$k=1$} & \multicolumn{2}{c}{$k=2$} & \multicolumn{2}{c}{$k=3$} & \multicolumn{2}{c}{$k=4$} \\
     &  & $L^2$ & EOC & $L^2$ & EOC & $L^2$ & EOC & $L^2$ & EOC \\
    \midrule
    1.250e-01 & 8 & 3.523e-01 & --- & 1.662e-02 & --- & 6.513e-04 & --- & 3.836e-05 & --- \\
    6.250e-02 & 16 & 1.687e-01 & 1.06 & 7.456e-04 & 4.48 & 3.552e-05 & 4.20 & 1.017e-06 & 5.24 \\
    3.125e-02 & 32 & 8.347e-02 & 1.02 & 1.074e-04 & 2.79 & 9.758e-06 & 1.86 & 2.505e-08 & 5.34 \\
    1.562e-02 & 64 & 4.163e-02 & 1.00 & 1.175e-05 & 3.19 & 1.421e-06 & 2.78 & 7.722e-10 & 5.02 \\
    \bottomrule
  \end{tabular}
\end{table}

\begin{table}[!htb]
  \centering
  \caption{Burgers 1D, EC flux differencing, CGSEM operators of degrees $k = 1, 2, 3, 4$ ($L^2$ error and EOC).}
  \label{tab:burgers-cgsem}
  \begin{tabular}{cccccccccc}
    \toprule
    $h$ & $N_e$ & \multicolumn{2}{c}{$k=1$} & \multicolumn{2}{c}{$k=2$} & \multicolumn{2}{c}{$k=3$} & \multicolumn{2}{c}{$k=4$} \\
     &  & $L^2$ & EOC & $L^2$ & EOC & $L^2$ & EOC & $L^2$ & EOC \\
    \midrule
    1.250e-01 & 8 & 2.994e-01 & --- & 1.080e-02 & --- & 5.118e-04 & --- & 3.484e-05 & --- \\
    6.250e-02 & 16 & 8.866e-02 & 1.76 & 3.852e-03 & 1.49 & 2.638e-05 & 4.28 & 2.935e-06 & 3.57 \\
    3.125e-02 & 32 & 2.432e-02 & 1.87 & 1.034e-03 & 1.90 & 1.300e-06 & 4.34 & 5.415e-08 & 5.76 \\
    1.562e-02 & 64 & 6.147e-03 & 1.98 & 2.562e-04 & 2.01 & 8.484e-08 & 3.94 & 1.639e-09 & 5.05 \\
    \bottomrule
  \end{tabular}
\end{table}

The CGSEM operator of degree $k = 1$ is a degenerate member of the element-based class:
coupling the two-node Legendre--Gauss--Lobatto element operators continuously on a uniform mesh reproduces exactly the classical second-order central FD SBP operator, with the mass matrix $M = h \, \I$ on a periodic domain \cite[Example~2.1 and Corollary~2.2]{ranocha2021broad}.
Its column in Table~\ref{tab:burgers-cgsem} therefore agrees with the order-$2$ column of Table~\ref{tab:burgers-fdsbp} in every printed digit at the two mesh sizes the two tables have in common, and it converges at the rate $2$ of that FD operator instead of following the pattern of the continuous couplings.

\subsection{The compressible Euler equations in one space dimension}
\label{sec:numerics-euler1d}

For the compressible Euler equations with $\gamma = 1.4$, we use the entropy-conservative flux \eqref{eq:flux-ranocha} of \cite{ranocha2018comparison,ranocha2021preventing}.
Table~\ref{tab:euler1d-dw-fdsbp} reports the exact, source-free density-wave solution $\rho = 1 + 0.2\sin(2\pi(x - v_1 t))$ with velocity $v_1 = 0.1$ and constant pressure $p = 20$ on the periodic interval $\Omega = [-1,1]$ up to the final time $T = 1$, for the single-block FD SBP operators of orders $2$, $4$, and $6$ refined from $N_n = 16$ to $N_n = 128$ nodes.
This is similar to the density-wave setup of \cite{gassner2022stability} implemented in Trixi.jl.

\begin{table}[!htb]
  \centering
  \caption{Euler 1D density wave, EC flux differencing, periodic FD-SBP orders 2, 4, 6 (full-state $L^2$ error and EOC).}
  \label{tab:euler1d-dw-fdsbp}
  \begin{tabular}{cccccccc}
    \toprule
    $h$ & $N_n$ & \multicolumn{2}{c}{$\text{order 2}$} & \multicolumn{2}{c}{$\text{order 4}$} & \multicolumn{2}{c}{$\text{order 6}$} \\
     &  & $L^2$ & EOC & $L^2$ & EOC & $L^2$ & EOC \\
    \midrule
    1.250e-01 & 16 & 8.896e-03 & --- & 1.115e-03 & --- & 2.410e-04 & --- \\
    6.250e-02 & 32 & 2.277e-03 & 1.97 & 7.594e-05 & 3.88 & 5.689e-06 & 5.40 \\
    3.125e-02 & 64 & 5.727e-04 & 1.99 & 4.860e-06 & 3.97 & 1.013e-07 & 5.81 \\
    1.562e-02 & 128 & 1.434e-04 & 2.00 & 3.056e-07 & 3.99 & 1.640e-09 & 5.95 \\
    \bottomrule
  \end{tabular}
\end{table}

The multi-block FD SBP operators of interior orders $4$ and $6$ (Table~\ref{tab:euler1d-dw-fdsbp-multiblock}) again converge one order above the closure-limited rates $p = 2$ and $p = 3$ guaranteed by Theorem~\ref{thm:convergence}, as in Section~\ref{sec:numerics-burgers}.

\begin{table}[!htb]
  \centering
  \caption{Euler 1D density wave, EC flux differencing, multi-block FD-SBP operators of interior orders 4 and 6, refined in the number of blocks (full-state $L^2$ error and EOC).}
  \label{tab:euler1d-dw-fdsbp-multiblock}
  \begin{tabular}{cccccc}
    \toprule
    $h$ & $N_e$ & \multicolumn{2}{c}{$\text{interior order 4}$} & \multicolumn{2}{c}{$\text{interior order 6}$} \\
     &  & $L^2$ & EOC & $L^2$ & EOC \\
    \midrule
    5.000e-01 & 4 & 1.660e-04 & --- & 3.672e-06 & --- \\
    2.500e-01 & 8 & 2.160e-05 & 2.94 & 2.779e-07 & 3.72 \\
    1.250e-01 & 16 & 2.095e-06 & 3.37 & 1.674e-08 & 4.05 \\
    6.250e-02 & 32 & 2.806e-07 & 2.90 & 8.564e-10 & 4.29 \\
    \bottomrule
  \end{tabular}
\end{table}

The element-based CGSEM operators of degrees $k \in \{1, 2, 3, 4\}$ on $8$ to $64$ elements attain at least the predicted rate for the density wave (Table~\ref{tab:euler1d-dw-cgsem}).

\begin{table}[!htb]
  \centering
  \caption{Euler 1D density wave, EC flux differencing, CGSEM operators of degrees $k = 1, 2, 3, 4$ (full-state $L^2$ error and EOC).}
  \label{tab:euler1d-dw-cgsem}
  \begin{tabular}{cccccccccc}
    \toprule
    $h$ & $N_e$ & \multicolumn{2}{c}{$k=1$} & \multicolumn{2}{c}{$k=2$} & \multicolumn{2}{c}{$k=3$} & \multicolumn{2}{c}{$k=4$} \\
     &  & $L^2$ & EOC & $L^2$ & EOC & $L^2$ & EOC & $L^2$ & EOC \\
    \midrule
    2.500e-01 & 8 & 3.224e-02 & --- & 9.958e-03 & --- & 1.240e-03 & --- & 1.115e-04 & --- \\
    1.250e-01 & 16 & 8.896e-03 & 1.86 & 2.706e-03 & 1.88 & 8.393e-05 & 3.89 & 6.343e-06 & 4.14 \\
    6.250e-02 & 32 & 2.277e-03 & 1.97 & 6.912e-04 & 1.97 & 3.727e-06 & 4.49 & 3.993e-07 & 3.99 \\
    3.125e-02 & 64 & 5.727e-04 & 1.99 & 1.737e-04 & 1.99 & 1.908e-07 & 4.29 & 2.499e-08 & 4.00 \\
    \bottomrule
  \end{tabular}
\end{table}

Table~\ref{tab:euler1d-manufactured} adds a fully nonlinear cross-check on the manufactured solution $\rho = \rho v_1 = 2 + 0.1 \sin(\pi (x - t))$, $\rho e = \rho^2$ on $\Omega = [0,2]$ up to $T = 0.5$, with all state components varying and the corresponding state-independent source;
this is the one-dimensional member of the family \eqref{eq:euler2d-manufactured} implemented in Trixi.jl.
The periodic finite-difference blocks of orders $2$, $4$, and $6$ attain their orders on it as well.
We do not show results for the other classes of methods here, since they behave similarly to the 2D case of Section~\ref{sec:numerics-2d} below.

\begin{table}[!htb]
  \centering
  \caption{Euler 1D manufactured solution (nonlinear), EC flux differencing, periodic FD-SBP orders 2, 4, 6 (full-state $L^2$ error and EOC).}
  \label{tab:euler1d-manufactured}
  \begin{tabular}{cccccccc}
    \toprule
    $h$ & $N_n$ & \multicolumn{2}{c}{$\text{order 2}$} & \multicolumn{2}{c}{$\text{order 4}$} & \multicolumn{2}{c}{$\text{order 6}$} \\
     &  & $L^2$ & EOC & $L^2$ & EOC & $L^2$ & EOC \\
    \midrule
    1.250e-01 & 16 & 1.588e-02 & --- & 4.928e-04 & --- & 2.118e-05 & --- \\
    6.250e-02 & 32 & 3.964e-03 & 2.00 & 3.127e-05 & 3.98 & 3.491e-07 & 5.92 \\
    3.125e-02 & 64 & 9.905e-04 & 2.00 & 1.962e-06 & 3.99 & 5.532e-09 & 5.98 \\
    1.562e-02 & 128 & 2.476e-04 & 2.00 & 1.228e-07 & 4.00 & 8.694e-11 & 5.99 \\
    \bottomrule
  \end{tabular}
\end{table}

\subsection{The shallow water equations}
\label{sec:numerics-sw}

For the shallow water equations we use the entropy-conservative flux of Wintermeyer et al.\ \cite{wintermeyer2017entropy} implemented in TrixiShallowWater.jl \cite{winters2025trixi}, with a flat bottom, so that the topography term vanishes and the source is state independent, matching the scope of the analysis.
The gravitational acceleration is $g = 9.81$, and the manufactured solution on the periodic square $\Omega = [0,1]^2$ is
\begin{equation}
\label{eq:sw-manufactured}
\begin{aligned}
  H(t, x, y) &= 2 + 0.2 \sin\bigl(2\pi(x - v_1 t)\bigr) + 0.15 \sin\bigl(2\pi(y - v_2 t)\bigr),
  \\
  v &= (v_1, v_2) = (0.25, 0.2),
\end{aligned}
\end{equation}
with the state-independent source $s = (0, g H \partial_x H, g H \partial_y H)$ evaluated on the exact solution.
The final time is $T = 0.25$ and the operators have degrees $k \in \{2, 3, 4\}$.
Table~\ref{tab:sw2d-quad} reports this solution on Cartesian DGSEM elements;
the same experiment on triangular SBP elements converges at comparable rates.

\begin{table}[!htb]
  \centering
  \caption{Shallow water 2D, flat bathymetry, EC flux differencing, DGSEM (Quad) operators of degrees $k = 2, 3, 4$ (full-state $L^2$ error and EOC).}
  \label{tab:sw2d-quad}
  \begin{tabular}{cccccccc}
    \toprule
    $h$ & $N_e$ & \multicolumn{2}{c}{$k=2$} & \multicolumn{2}{c}{$k=3$} & \multicolumn{2}{c}{$k=4$} \\
     &  & $L^2$ & EOC & $L^2$ & EOC & $L^2$ & EOC \\
    \midrule
    2.500e-01 & 4 & 5.515e-02 & --- & 4.517e-03 & --- & 3.066e-04 & --- \\
    1.250e-01 & 8 & 4.363e-03 & 3.66 & 5.937e-04 & 2.93 & 5.900e-06 & 5.70 \\
    6.250e-02 & 16 & 4.196e-04 & 3.38 & 8.475e-05 & 2.81 & 1.648e-07 & 5.16 \\
    3.125e-02 & 32 & 5.027e-05 & 3.06 & 1.060e-05 & 3.00 & 4.809e-09 & 5.10 \\
    \bottomrule
  \end{tabular}
\end{table}

\subsection{Two space dimensions, general operators, and curved grids}
\label{sec:numerics-2d}

The two-dimensional experiments use the Euler manufactured solution
\begin{equation}
\label{eq:euler2d-manufactured}
  \rho = \rho v_1 = \rho v_2 = 2 + 0.1 \sin\bigl(\pi (x + y - t)\bigr),
  \qquad
  \rho e = \rho^2,
\end{equation}
with $\gamma = 1.4$ and the corresponding state-independent source, on the periodic square $\Omega = [-1,1]^2$ up to the final time $T = 0.4$.
It is a variant of the manufactured solution of \cite[Section~3.1.1]{schlottkelakemper2021purely}, which uses the same density ansatz and the same constant velocity $v = (1,1)$ but the thermodynamic closure $\gamma = 2$, $p = \rho^2 / \pi$;
the form \eqref{eq:euler2d-manufactured} is the one implemented in Trixi.jl.

On Cartesian tensor-product elements we use DGSEM and CGSEM operators and tensor products of periodic finite-difference SBP operators.
The FD operators attain their orders cleanly (Table~\ref{tab:euler2d-fdsbp}), as in one space dimension, whereas the element-based operators (Tables~\ref{tab:euler2d-dgsem} and \ref{tab:euler2d-cgsem}) again converge at rates between roughly $k$ and $k+1$;
the DGSEM rates show the same dependence on the parity of $k$ as in Section~\ref{sec:numerics-burgers}.
The identification of the degree-one CGSEM operator with the second-order finite-difference operator persists under tensor products, and indeed the $k = 1$ column of Table~\ref{tab:euler2d-cgsem} reproduces the order-$2$ column of Table~\ref{tab:euler2d-fdsbp} at the common mesh sizes.

\begin{table}[!htb]
  \centering
  \caption{Euler 2D manufactured solution, EC flux differencing, tensor products of periodic FD-SBP operators of orders 2, 4, 6; full-state $L^2$ error and EOC.}
  \label{tab:euler2d-fdsbp}
  \begin{tabular}{cccccccc}
    \toprule
    $h$ & $N_n$ & \multicolumn{2}{c}{$\text{order 2}$} & \multicolumn{2}{c}{$\text{order 4}$} & \multicolumn{2}{c}{$\text{order 6}$} \\
     &  & $L^2$ & EOC & $L^2$ & EOC & $L^2$ & EOC \\
    \midrule
    2.500e-01 & 8 & 7.763e-02 & --- & 8.498e-03 & --- & 1.245e-03 & --- \\
    1.250e-01 & 16 & 1.837e-02 & 2.08 & 5.547e-04 & 3.94 & 2.321e-05 & 5.75 \\
    6.250e-02 & 32 & 4.518e-03 & 2.02 & 3.517e-05 & 3.98 & 3.830e-07 & 5.92 \\
    3.125e-02 & 64 & 1.125e-03 & 2.01 & 2.206e-06 & 3.99 & 6.095e-09 & 5.97 \\
    \bottomrule
  \end{tabular}
\end{table}

\begin{table}[!htb]
  \centering
  \caption{Euler 2D manufactured solution, EC flux differencing, DGSEM of degrees $k = 1, 2, 3, 4$ on Cartesian \texttt{Quad} elements (full-state $L^2$ error and EOC).}
  \label{tab:euler2d-dgsem}
  \begin{tabular}{cccccccccc}
    \toprule
    $h$ & $N_e$ & \multicolumn{2}{c}{$k=1$} & \multicolumn{2}{c}{$k=2$} & \multicolumn{2}{c}{$k=3$} & \multicolumn{2}{c}{$k=4$} \\
     &  & $L^2$ & EOC & $L^2$ & EOC & $L^2$ & EOC & $L^2$ & EOC \\
    \midrule
    5.000e-01 & 4 & 3.757e-01 & --- & 3.312e-02 & --- & 2.987e-03 & --- & 3.843e-04 & --- \\
    2.500e-01 & 8 & 1.901e-01 & 0.98 & 3.237e-03 & 3.35 & 3.956e-04 & 2.92 & 1.033e-05 & 5.22 \\
    1.250e-01 & 16 & 9.536e-02 & 1.00 & 4.334e-04 & 2.90 & 5.498e-05 & 2.85 & 2.123e-07 & 5.60 \\
    6.250e-02 & 32 & 4.772e-02 & 1.00 & 7.851e-05 & 2.46 & 7.034e-06 & 2.97 & 7.986e-09 & 4.73 \\
    \bottomrule
  \end{tabular}
\end{table}

\begin{table}[!htb]
  \centering
  \caption{Euler 2D manufactured solution, EC flux differencing, tensor-product CGSEM operators of degrees $k = 1, 2, 3, 4$; full-state $L^2$ error and EOC.}
  \label{tab:euler2d-cgsem}
  \begin{tabular}{cccccccccc}
    \toprule
    $h$ & $N_e$ & \multicolumn{2}{c}{$k=1$} & \multicolumn{2}{c}{$k=2$} & \multicolumn{2}{c}{$k=3$} & \multicolumn{2}{c}{$k=4$} \\
     &  & $L^2$ & EOC & $L^2$ & EOC & $L^2$ & EOC & $L^2$ & EOC \\
    \midrule
    5.000e-01 & 4 & 3.404e-01 & --- & 3.671e-02 & --- & 2.659e-03 & --- & 4.280e-04 & --- \\
    2.500e-01 & 8 & 7.763e-02 & 2.13 & 7.267e-03 & 2.34 & 2.133e-04 & 3.64 & 2.972e-05 & 3.85 \\
    1.250e-01 & 16 & 1.837e-02 & 2.08 & 1.678e-03 & 2.11 & 9.129e-06 & 4.55 & 1.900e-06 & 3.97 \\
    6.250e-02 & 32 & 4.518e-03 & 2.02 & 4.110e-04 & 2.03 & 6.873e-07 & 3.73 & 1.198e-07 & 3.99 \\
    \bottomrule
  \end{tabular}
\end{table}

The triangular (simplex) SBP operators of Table~\ref{tab:euler2d-tri} (degrees $k \in \{1, 2, 3, 4\}$, $4 \times 4$ to $32 \times 32$ squares, each split into two triangles) are genuinely multidimensional and non-tensor-product, and exercise the general operator framework of Section~\ref{sec:setting} rather than only tensor products; they converge at least at the operator order $p = k$.
The degree $k = 1$ deserves a separate comment:
it corresponds to $p = 1$, which violates the condition $p > d/2$ of Assumption~\ref{asm:order} in two space dimensions, so Theorem~\ref{thm:convergence} does not apply to it.
The degree-one columns of Tables~\ref{tab:euler2d-dgsem} and \ref{tab:euler2d-tri} nevertheless converge at the operator order, which supports the view taken in Section~\ref{ssec:order-condition} that the condition is a technical artifact of the $L^\infty$ bootstrap rather than a genuine obstruction.

\begin{table}[!htb]
  \centering
  \caption{Euler 2D manufactured solution, EC flux differencing on triangular (simplex) SBP elements of degrees $k = 1, 2, 3, 4$ (full-state $L^2$ error and EOC).}
  \label{tab:euler2d-tri}
  \begin{tabular}{cccccccccc}
    \toprule
    $h$ & $N_e$ & \multicolumn{2}{c}{$k=1$} & \multicolumn{2}{c}{$k=2$} & \multicolumn{2}{c}{$k=3$} & \multicolumn{2}{c}{$k=4$} \\
     &  & $L^2$ & EOC & $L^2$ & EOC & $L^2$ & EOC & $L^2$ & EOC \\
    \midrule
    5.000e-01 & 4 & 4.495e-01 & --- & 8.132e-02 & --- & 1.435e-02 & --- & 3.420e-03 & --- \\
    2.500e-01 & 8 & 1.983e-01 & 1.18 & 1.113e-02 & 2.87 & 1.485e-03 & 3.27 & 1.298e-04 & 4.72 \\
    1.250e-01 & 16 & 8.645e-02 & 1.20 & 1.620e-03 & 2.78 & 1.760e-04 & 3.08 & 5.741e-06 & 4.50 \\
    6.250e-02 & 32 & 3.372e-02 & 1.36 & 1.822e-04 & 3.15 & 2.252e-05 & 2.97 & 1.535e-07 & 5.22 \\
    \bottomrule
  \end{tabular}
\end{table}

Finally, we place the same solution on the smoothly curved periodic mesh of Figure~\ref{fig:curved-mesh}, obtained by warping the reference square with
\begin{equation}
\label{eq:curved-mapping}
  x = \xi + 0.1 \sin(\pi \xi) \sin(\pi \eta),
  \quad
  y = \eta + 0.1 \sin(\pi \xi) \sin(\pi \eta),
  \qquad
  (\xi, \eta) \in [-1,1]^2,
\end{equation}
which leaves the boundaries of the reference square invariant so that opposite edges still match, and discretize it with DGSEM.
The mapping \eqref{eq:curved-mapping} is used in Trixi.jl;
smooth sinusoidal warpings of this kind are a standard test setting for curvilinear entropy-stable discretizations, cf.\ the related warping in \cite[Section~4]{wu2021high}.
The manufactured solution converges as usual (Table~\ref{tab:euler2d-curved}).

\input{code/results/curved_mesh.tex}

\begin{table}[!htb]
  \centering
  \caption{Euler 2D manufactured solution on a smoothly curved periodic grid, EC flux differencing, DGSEM of degrees $k = 2, 3, 4$ (full-state $L^2$ error and EOC).}
  \label{tab:euler2d-curved}
  \begin{tabular}{cccccccc}
    \toprule
    $h$ & $N_e$ & \multicolumn{2}{c}{$k=2$} & \multicolumn{2}{c}{$k=3$} & \multicolumn{2}{c}{$k=4$} \\
     &  & $L^2$ & EOC & $L^2$ & EOC & $L^2$ & EOC \\
    \midrule
    5.000e-01 & 4 & 1.244e-01 & --- & 1.024e-02 & --- & 6.462e-03 & --- \\
    2.500e-01 & 8 & 1.347e-02 & 3.21 & 1.192e-03 & 3.10 & 2.589e-04 & 4.64 \\
    1.250e-01 & 16 & 1.311e-03 & 3.36 & 1.226e-04 & 3.28 & 3.663e-06 & 6.14 \\
    6.250e-02 & 32 & 1.800e-04 & 2.86 & 1.567e-05 & 2.97 & 1.290e-07 & 4.83 \\
    \bottomrule
  \end{tabular}
\end{table}

\subsection{Diagnostics of the analysis in one space dimension}
\label{sec:numerics-diagnostics}

Finally, we illustrate three ingredients of the analysis for the compressible Euler equations in one space dimension:
the admissibility of the numerical solution without limiting (discussion after Corollary~\ref{cor:euler}), the equivalence of the discrete relative entropy and the squared $M$-norm error (Lemma~\ref{lem:coercivity}), and the structural remainder bound of Assumption~\ref{asm:remainder}.

To illustrate positivity without limiters, we consider two smooth solutions with small density on the periodic interval $\Omega = [-1, 1]$ up to the final time $T = 5$.
The first one is the density wave $\rho = 1 + 0.98 \sin\bigl(2 \pi (x - v t)\bigr)$ with velocity $v = 0.1$ and pressure $p = 20$ of \cite{gassner2022stability,ranocha2021preventing}, for which entropy-conservative schemes are known to crash on coarse meshes.
Since the flux \eqref{eq:flux-ranocha} preserves pressure equilibria, the pressure remains constant for this solution (up to errors caused by floating-point arithmetic).
Thus, we also consider the manufactured solution with the same density and velocity and the small pressure $p = 0.01 \bigl(1 + 0.5 \sin\bigl(2 \pi (x - v t)\bigr)\bigr)$, so that the minimal density and pressure are $0.02$ and $0.005$, respectively.
We use the entropy-conservative flux \eqref{eq:flux-ranocha} without any limiter, DGSEM operators of degrees $k \in \{2, 3\}$, and the periodic FD SBP operator of order $4$.
We use the time integration method of Section~\ref{sec:numerics-setup} with tolerances $10^{-12}$ and record the minimal nodal density and pressure over all accepted time steps;
a crash is caused by some quantities becoming \texttt{NaN}.
Reducing the tolerances to $10^{-13}$ does not change which simulations crash and changes the crash times and the minima by less than $10^{-4}$ relative, so that the temporal errors are negligible.
Table~\ref{tab:euler1d-positivity} shows that the DGSEM simulations on the coarsest meshes crash, while all finer meshes and all FD meshes remain positive, with minima approaching the exact minimal values.
This illustrates the existence of a mesh threshold $h_0$ as in Theorem~\ref{thm:convergence} and the remark after Corollary~\ref{cor:euler}.

\begin{table}[!htb]
  \centering
  \caption{Euler 1D without limiting, EC flux differencing: minimum nodal density and pressure over the accepted time steps on $[0, 5]$ for (I) the density wave $\rho = 1 + 0.98 \sin(2 \pi (x - v t))$, $v = 0.1$, $p = 20$ and (II) the manufactured solution $\rho = 1 + 0.98 \sin \theta$, $v = 0.1$, $p = 0.01 (1 + 0.5 \sin \theta)$, $\theta = 2 \pi (x - v t)$. The time given is that of the last accepted step. $N_e$ is the number of elements (DGSEM) and $N_n$ the number of nodes (FD-SBP).}
  \label{tab:euler1d-positivity}
  \begin{tabular}{lrccc}
    \toprule
     & & (I) & \multicolumn{2}{c}{(II)} \\
    operator & $N_e$ / $N_n$ & $\min \rho$ & $\min \rho$ & $\min p$ \\
    \midrule
    DGSEM $k = 2$ & 4 & \multicolumn{1}{c}{crash at $t = 1.43$} & \multicolumn{2}{c}{crash at $t = 2.16$} \\
     & 8 & 9.227e-03 & 1.048e-02 & 4.577e-03 \\
     & 16 & 1.752e-02 & 1.755e-02 & 4.877e-03 \\
     & 32 & 1.862e-02 & 1.849e-02 & 4.979e-03 \\
     & 64 & 1.968e-02 & 1.967e-02 & 4.997e-03 \\
    \midrule
    DGSEM $k = 3$ & 2 & \multicolumn{1}{c}{crash at $t = 0.97$} & \multicolumn{2}{c}{crash at $t = 1.72$} \\
     & 4 & 7.517e-02 & 7.245e-02 & 4.310e-03 \\
     & 8 & 1.854e-02 & 1.848e-02 & 4.651e-03 \\
     & 16 & 1.979e-02 & 1.979e-02 & 4.964e-03 \\
     & 32 & 2.000e-02 & 2.000e-02 & 4.992e-03 \\
     & 64 & 1.998e-02 & 1.998e-02 & 5.000e-03 \\
    \midrule
    FD-SBP order 4 & 16 & 2.000e-02 & 2.000e-02 & 4.738e-03 \\
     & 32 & 2.000e-02 & 2.000e-02 & 4.959e-03 \\
     & 64 & 2.000e-02 & 2.000e-02 & 4.995e-03 \\
     & 128 & 2.000e-02 & 2.000e-02 & 5.000e-03 \\
     & 256 & 2.000e-02 & 2.000e-02 & 5.000e-03 \\
    \midrule
    exact & & 2.000e-02 & 2.000e-02 & 5.000e-03 \\
    \bottomrule
  \end{tabular}
\end{table}

For the relative entropy and the remainder bound, we use the manufactured solution of Table~\ref{tab:euler1d-manufactured} and an independent implementation of the semidiscretization \eqref{eq:semidiscretization} with the assembled matrices $M$ and $Q$ of the periodic FD SBP operators of orders $2$, $4$, and $6$, whose errors agree with those of Table~\ref{tab:euler1d-manufactured} to within $0.2\%$.
We evaluate the discrete relative entropy \eqref{eq:relative-entropy}, the squared $M$-norm error, and the nonlinear remainder of \eqref{eq:remainder-hypothesis} at $100$ output times in $(0, 0.5]$, using the integral form of the Taylor remainders to avoid cancellation errors.
Figure~\ref{fig:euler1d-relative-entropy} shows that $E_h$ and $\|\vnum - \vsol\|_M^2$ at the final time decay at the same rate $2p$, in accordance with the upper and lower bounds of Lemma~\ref{lem:coercivity}.
Their ratio at the final time lies between $0.052$ and $0.062$ and is nearly independent of the mesh size.
Table~\ref{tab:euler1d-remainder-ratio} reports the maxima over the output times of the ratio of the absolute nonlinear remainder and the relative entropy, which remain bounded (by about $0.03$) under mesh refinement, as required by Assumption~\ref{asm:remainder}, together with the maxima over the output times of the ratio $E_h / \|\vnum - \vsol\|_M^2$.

\begin{figure}[!htb]
  \centering
  \includegraphics[width=\textwidth]{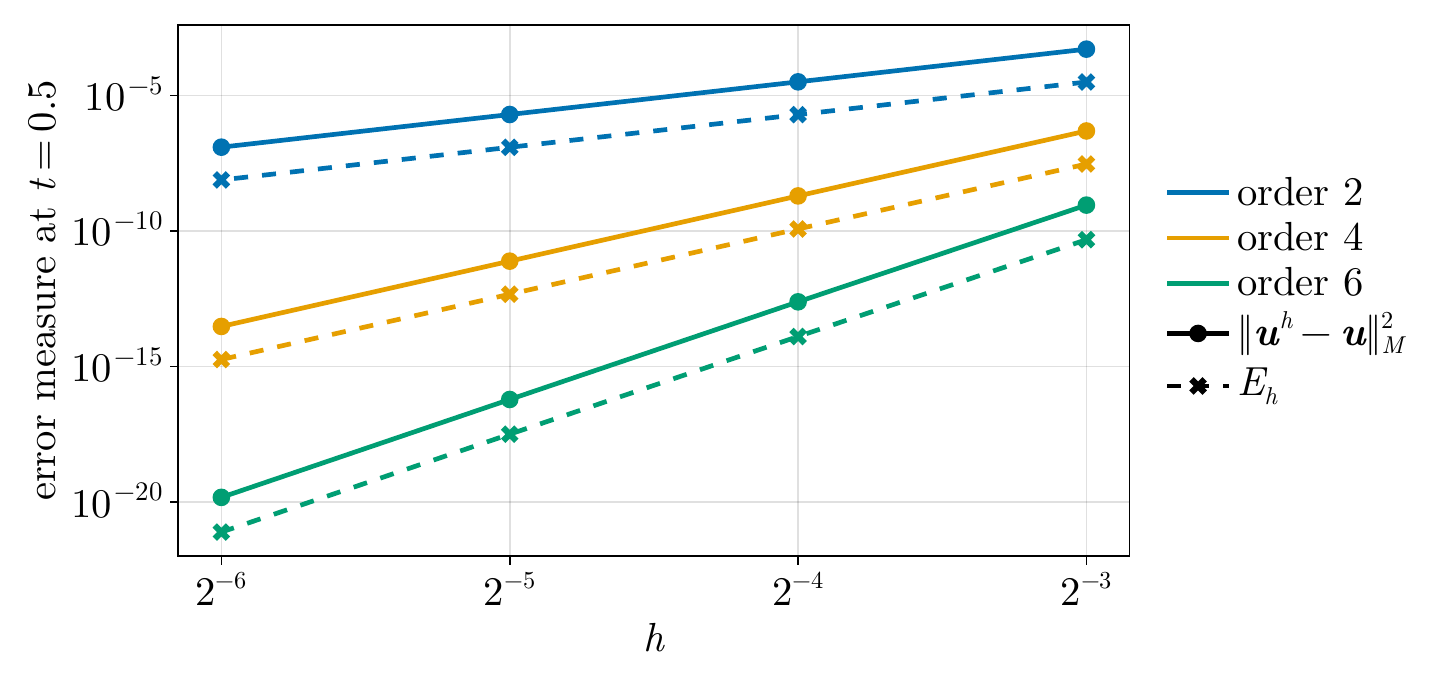}
  \caption{Euler 1D manufactured solution, EC flux differencing with periodic FD SBP operators of orders $2$, $4$, and $6$: discrete relative entropy $E_h$ and squared $M$-norm error $\|\vnum - \vsol\|_M^2$ at the final time $t = 0.5$ under mesh refinement.}
  \label{fig:euler1d-relative-entropy}
\end{figure}

\begin{table}[!htb]
  \centering
  \caption{Euler 1D manufactured solution, EC flux differencing with periodic FD-SBP operators of orders 2, 4, 6: maxima over the output times $t \in (0, 0.5]$ of the ratio $|\sum_{i,j} (\wsol_i - \wsol_j)^T Q_{ij} r_{ij}| / E_h$ of the nonlinear remainder and the relative entropy, and of $E_h / \|\vnum - \vsol\|_M^2$.}
  \label{tab:euler1d-remainder-ratio}
  \begin{tabular}{cccccccc}
    \toprule
    $h$ & $N_n$ & \multicolumn{2}{c}{order 2} & \multicolumn{2}{c}{order 4} & \multicolumn{2}{c}{order 6} \\
     &  & ratio & $E_h / \|\vnum - \vsol\|_M^2$ & ratio & $E_h / \|\vnum - \vsol\|_M^2$ & ratio & $E_h / \|\vnum - \vsol\|_M^2$ \\
    \midrule
    1.250e-01 & 16 & 0.023 & 0.062 & 0.028 & 0.063 & 0.029 & 0.103 \\
    6.250e-02 & 32 & 0.026 & 0.062 & 0.030 & 0.063 & 0.029 & 0.108 \\
    3.125e-02 & 64 & 0.027 & 0.062 & 0.030 & 0.063 & 0.029 & 0.110 \\
    1.562e-02 & 128 & 0.027 & 0.062 & 0.030 & 0.064 & 0.029 & 0.110 \\
    \bottomrule
  \end{tabular}
\end{table}

\section{Summary and discussion}
\label{sec:summary}

Using a semidiscrete relative entropy method, we have proved that entropy-conservative flux differencing semidiscretizations of periodic hyperbolic conservation laws in $d$ space dimensions converge to smooth solutions (at least) at the order of accuracy of the underlying diagonal-norm SBP operators, for quite general symmetrizable systems with strictly convex entropy and arbitrary bounded state-independent source terms $s(t,x)$.
The abstract SBP operator framework covers tensor-product grids as well as smooth curved periodic grids, verified in two and in three space dimensions in Appendix~\ref{app:encapsulated}, where the zero-row-sum assumption is exactly the discrete metric identity.
The applications include Burgers' equation, the shallow water equations and the compressible Euler equations.
We close by discussing some aspects of the result.

\emph{The role of the fixed-local-stencil condition.}
The only genuinely structural hypothesis beyond entropy conservation is the remainder bound of Assumption~\ref{asm:remainder}, which we verified under the fixed-local-stencil condition of Assumption~\ref{asm:bandwidth}.
The latter is a sufficient condition for the former; we have not shown that it is necessary.
The mechanism of Lemma~\ref{lem:bandwidth} is that $|\wsol_i - \wsol_j| = \O(h)$ on the stencil compensates exactly the mismatch between the row sums of order $h^{d-1}$ and the masses of order $h^d$.
For dense operators, such as single-domain Fourier or polynomial collocation methods, $|\wsol_i - \wsol_j| = \O(1)$ on the stencil, and the same estimate only yields a Gronwall constant that grows at least like $h^{-1}$, which is useless.
This shows that the proof of Lemma~\ref{lem:bandwidth} does not extend to dense operators.
For Fourier collocation, Assumption~\ref{asm:remainder} itself fails.
To see this, consider Burgers' equation with the quadratic entropy $U(u) = u^2/2$, the entropy-conservative flux $\fnum(a, b) = (a^2 + a b + b^2) / 6$, and the stationary solution $\usol(x) = \sin(x)$ on $\Omega = [0, 2\pi]$ with the source $s(x) = \sin(x) \cos(x)$.
The Fourier collocation method on $N = 2k + 1$ nodes $x_j = j h$, $h = 2 \pi / N$, uses $M = h \I$ and $Q = h D$ with the skew-symmetric Fourier differentiation matrix $D$.
For the nodal errors $e_j = \epsilon \cos(k x_j)$ with $0 < \epsilon \le 1$, all states lie in $K = [-2, 2]$, and the flux Taylor remainders are $r_{ij} = (e_i^2 + e_i e_j + e_j^2) / 6$.
Using $D^T = -D$ and $D \vec{1} = \vec{0}$, the left-hand side of \eqref{eq:remainder-hypothesis} becomes $\frac{h}{3} \bigl| \sum_j \bigl( \usol_j e_j (D \vec{e})_j - e_j^2 (D \vsol)_j \bigr) \bigr|$.
Inserting $(D \vsol)_j = \cos(x_j)$ and $(D \vec{e})_j = -k \epsilon \sin(k x_j)$ and using $\sin(2 k x_j) = -\sin(x_j)$ and $\cos(2 k x_j) = \cos(x_j)$ shows that this expression equals $(k - 1) N h \epsilon^2 / 12$, whereas $E_h = N h \epsilon^2 / 4$.
Thus, the ratio of both sides is $(k - 1) / 3 = \pi / (3 h) - 1 / 2$ independently of $\epsilon$, so that there is no constant $C_R$ independent of $h$, even for errors as small as $\epsilon = h^p$, which are compatible with Assumption~\ref{asm:initial-data}.
Hence, the growth of the estimate above is not an artifact of the proof of Lemma~\ref{lem:bandwidth}: for Fourier collocation, every admissible $C_R$ grows at least like $h^{-1}$.
This example does not decide whether Fourier collocation converges for smooth solutions;
it only shows that Theorem~\ref{thm:convergence} does not apply.
For polynomial collocation methods, it is open whether Assumption~\ref{asm:remainder} holds.
See, e.g., \cite{tadmor1989convergence} for related convergence results including a spectral vanishing viscosity.

\emph{No nullspace condition is needed.}
Purely entropy-conservative schemes do not damp spurious modes in the nullspace of $D$, such as the highest-frequency sawtooth mode of central operators on grids with an even number of nodes, whose symbol $\mathrm{i} \sin(\omega h) / h$ vanishes at the grid frequency $\omega h = \pi$ \cite[Section~1.1]{gustafsson2013time}.
One might expect that a nullspace-consistency hypothesis in the spirit of \cite{svard2019convergence} is necessary to prevent error growth in such modes.
It is not:
the coercivity of Lemma~\ref{lem:coercivity} comes from the convexity of the entropy and controls \emph{all} modes of the error, a spurious initial kernel mode is limited to size $\O(h^p)$ in the $M$-norm by the consistent initial data of Assumption~\ref{asm:initial-data}, and a kernel mode forced by the truncation error is controlled by the Gronwall estimate like any other error contribution.
Assumption~\ref{asm:initial-data} does not exclude an initial kernel component; it only makes it small enough not to affect the rate, and it is absent altogether for nodal sampling or under an additional orthogonality condition on the initial error.

\emph{Entropy-dissipative variants.}
We analyzed purely conservative schemes because they isolate the entropy structure in its sharpest form.
Adding entropy dissipation, e.g., interface dissipation in DG methods, adds a term to the evolution of $E_h$ that consists of the dissipation applied to the numerical solution, which has a favorable sign for the total entropy, plus consistency terms involving the smooth solution.
Making this precise requires bounding the dissipation operator applied to smooth data, which depends on the specific form of the dissipation;
we leave a rigorous treatment for future work.

\emph{Rate optimality and superconvergence.}
Theorem~\ref{thm:convergence} bounds the error by the truncation order $p$ of the assembled operator, which is the order of the pointwise estimate \eqref{eq:accuracy}, uniformly in the node index;
numerical experiments frequently show higher rates.
Two mechanisms are visible in Section~\ref{sec:numerics}.
For element-based methods, the observed rates exceed the guaranteed $p = k$ and approach the optimal rate $k+1$, following the parity pattern known for central-flux DG methods applied to linear problems \cite{liu2021suboptimal}.
For multi-block finite-difference operators, the low-order rows are confined to the closures at the block interfaces, and the classical linear theory converts an interior order $2p$ together with closures of order $p$ into the global rate $\min(2p, p+1) = p+1$ \cite{gustafsson1975convergence,svard2006order,svard2019convergence,svard2020convergence};
our experiments show the same gain for the nonlinear entropy-conservative schemes.
An estimate that is sensitive to the structure of the truncation error near the closures, rather than to its maximum over all nodes, should therefore be able to recover the additional order;
whether a relative-entropy argument can do so, or capture improved rates such as $p + 1/2$ in the spirit of the classical results for dissipative discontinuous Galerkin schemes \cite{johnson1986analysis,peterson1991note,zhang2004error}, is open.

\emph{The order condition.}
The condition $p > d/2$ is analyzed in Section~\ref{ssec:order-condition}:
it enters the proof at a single point, excludes only first-order methods for $d \in \{2, 3\}$, disappears entirely under globally bounded data (Proposition~\ref{prop:global-bounds}), and can be bypassed by an a~posteriori admissibility check (Proposition~\ref{prop:aposteriori}).
Whether it can be removed unconditionally, e.g.\ along the corrector route sketched at the end of Section~\ref{ssec:order-condition}, is open.

\emph{Local linear stability.}
The analysis does not use any local linear stability property of the semidiscretization;
besides consistency and the smoothness of the solution, it relies on entropy conservation, the SBP structure, the strong convexity of the entropy on a compact set $K$, and the structural remainder bound of Assumption~\ref{asm:remainder}.
Entropy-conservative flux differencing schemes are in general not locally energy stable in the terminology of \cite{gassner2022stability}:
their linearizations around non-constant states can have eigenvalues with positive real parts \cite[Sections~2.4 and 3]{gassner2022stability}, and simulations of smooth problems can crash on meshes of practical resolution \cite[Section~3]{gassner2022stability}, \cite{ranocha2021preventing}.
These observations do not contradict Theorem~\ref{thm:convergence}.
First, for Burgers' equation, the maximal positive real parts reported in \cite[Section~2.4.1]{gassner2022stability} stay approximately constant under mesh refinement, which led the authors to suggest that the schemes might be stable in the sense of Lax and might converge for smooth solutions \cite{gassner2022stability};
correspondingly, the Gronwall estimate \eqref{eq:gronwall-solved} allows an exponential growth of the error bound with a rate $C_2 / 2$ that is independent of $h$.
Second, the threshold $h_0$ of Theorem~\ref{thm:convergence} depends on the final time $T$, among others through the factor $\e^{C_2 T / 2}$, and the available bounds for it are not sharp, so that crashes on such meshes are compatible with convergence as $h \to 0$.

\emph{Further extensions.}
Several restrictions of our setting mark natural directions for future work.
On curved meshes, Appendix~\ref{app:encapsulated} casts practical metric-averaged flux differencing schemes into our framework.
For the conservative curl form \eqref{eq:discrete-metrics}, all conditions of Definition~\ref{def:admissible-metrics} are verified in Lemma~\ref{lem:encapsulated-construction}, so Theorem~\ref{thm:convergence} applies to that scheme without further hypotheses.
An implementation using different discrete metric terms still needs its own verification:
the identities~\ref{item:metric-identity} coincide with the discrete metric identities it must satisfy for entropy conservation anyway, but the accuracy of its particular metric and Jacobian terms --- conditions~\ref{item:metric-accuracy} and \ref{item:metric-jacobian}, the latter unless its Jacobian construction makes it follow from \ref{item:metric-accuracy} --- has to be checked case by case, which we do not do here.
For source terms, the present analysis covers exactly collocated, state-independent sources;
state-dependent sources, source discretizations by interpolation or quadrature, and entropy-compatible two-point source forms --- such as well-balanced discretizations of the shallow water equations with bottom topography --- require additional estimates and are left for future work.
Bounded domains with entropy-stable boundary treatments \cite{svard2021entropy,svard2025entropy} introduce boundary terms in the relative-entropy balance that require separate estimates.
With boundaries, the SBP property reads $Q_m + Q_m^T = B_m \ne 0$, so that the symmetrization and cancellation arguments of Section~\ref{sec:proof} leave boundary terms;
entropy-stable boundary treatments bound the entropy of the numerical solution, whereas the relative-entropy argument requires a bound on the relative entropy flux between the numerical and the exact solution across the boundary;
and the less accurate boundary closures of finite-difference SBP operators reduce the rate guaranteed by our analysis to the order of the closures.
Fully discrete schemes could be analyzed by combining the present spatial estimates with entropy-preserving time integration, e.g., relaxation Runge-Kutta methods \cite{ketcheson2019relaxation,ranocha2020relaxation,ranocha2020general}.

\appendix

\section{Explicit constants for the Euler equations}
\label{app:euler-constants}

This appendix provides explicit, sufficient (non-sharp) expressions for the constants entering Corollary~\ref{cor:euler}:
the entropy Hessian bounds $c_U, C_U$ of \eqref{eq:hessian-bounds} and the regularity of the logarithmic mean underlying the flux constant $C_S$ of \eqref{eq:remainder-bound}.
Throughout, $K$ is as in \eqref{eq:euler-K} with induced bounds
\begin{equation}
  \rho \in [\rho_{\min}, \rho_{\max}],
  \qquad
  |v| \le V,
  \qquad
  p \in [p_{\min}, p_{\max}],
  \qquad
  \rho_{\min}, p_{\min} > 0.
\end{equation}

\subsection{Entropy Hessian bounds}
\label{app:hessian}

Closed-form expressions for the entropy Hessian $U''$ of the compressible Euler equations and for its inverse, the symmetrizer $\partial u / \partial w$, are classical \cite[Appendix~A]{hughes1986new} and \cite{barth1999numerical};
what we need here, and derive below, are explicit bounds on its spectrum on the compact set $K$.
We compute the Hessian of the entropy \eqref{eq:euler-entropy} in primitive coordinates in 1D.
Let $q = (\rho, v, p)$ and let $u(q)$ denote the conserved variables, with Jacobian
\begin{equation}
  J
  =
  \frac{\partial u}{\partial q}
  =
  \begin{pmatrix}
    1 & 0 & 0 \\
    v & \rho & 0 \\
    v^2 / 2 & \rho v & 1 / (\gamma - 1)
  \end{pmatrix},
  \qquad
  \det J = \frac{\rho}{\gamma - 1}.
\end{equation}
Pulling the entropy quadratic form back to primitive perturbations, $\dif u^T U''(u) \dif u = \dif q^T S \dif q$ with $S = J^T U'' J$, a direct computation gives the closed form
\begin{equation}
\label{eq:hessian-primitive}
  S
  =
  \begin{pmatrix}
    \frac{\gamma}{(\gamma - 1) \rho} & 0 & - \frac{1}{(\gamma - 1) p} \\
    0 & \frac{\rho^2}{p} & 0 \\
    - \frac{1}{(\gamma - 1) p} & 0 & \frac{\rho}{(\gamma - 1) p^2}
  \end{pmatrix}.
\end{equation}
The velocity direction decouples with eigenvalue $\rho^2 / p$, and the $(\dif\rho, \dif p)$ block is $M_2 / (\gamma - 1)$ with
\begin{equation}
  M_2
  =
  \begin{pmatrix}
    \gamma / \rho & - 1/p \\
    - 1/p & \rho / p^2
  \end{pmatrix},
  \qquad
  \det M_2 = \frac{\gamma - 1}{p^2} > 0,
\end{equation}
so $S$ is symmetric positive definite on $\mathcal{A}$.

The eigenvalues of $U''(u)$ are the generalized eigenvalues of the pencil $(S, G)$ with $G = J^T J$, since $\abs{\dif u}^2 = \dif q^T G \dif q$.
By Rayleigh quotients,
\begin{equation}
\label{eq:rayleigh}
  \lambda_{\min}(U'')
  \ge
  \frac{\lambda_{\min}(S)}{\lambda_{\max}(G)},
  \qquad
  \lambda_{\max}(U'')
  \le
  \frac{\lambda_{\max}(S)}{\lambda_{\min}(G)}.
\end{equation}
We bound $G$ via Frobenius norms, $\lambda_{\max}(G) \le \|J\|_F^2$ and $\lambda_{\min}(G) \ge \|J^{-1}\|_F^{-2}$, where
\begin{equation}
  J^{-1}
  =
  \begin{pmatrix}
    1 & 0 & 0 \\
    - v / \rho & 1 / \rho & 0 \\
    (\gamma - 1) v^2 / 2 & - (\gamma - 1) v & \gamma - 1
  \end{pmatrix}.
\end{equation}
Maximizing the state-dependent quantities over $K$ yields the explicit constants
\begin{equation}
\label{eq:explicit-G-bounds}
\begin{aligned}
  \Gamma_{\mathrm{up}}
  &=
  1 + V^2 + \rho_{\max}^2 + \frac{V^4}{4} + \rho_{\max}^2 V^2 + \frac{1}{(\gamma - 1)^2}
  \ \ge\ \lambda_{\max}(G),
  \\
  \Lambda_{\mathrm{up}}
  &=
  1 + \frac{1 + V^2}{\rho_{\min}^2} + (\gamma - 1)^2 \Bigl( \frac{V^4}{4} + V^2 + 1 \Bigr)
  \ \ge\ \|J^{-1}\|_F^{2}.
\end{aligned}
\end{equation}
For the eigenvalues of $S$, we use the elementary bounds
\begin{equation}
  \lambda_+(M_2) \le \operatorname{tr}(M_2),
  \qquad
  \lambda_-(M_2)
  =
  \frac{\det(M_2)}{\lambda_+(M_2)}
  \ge
  \frac{\det(M_2)}{\operatorname{tr}(M_2)},
\end{equation}
which give
\begin{equation}
\label{eq:explicit-S-bounds}
\begin{aligned}
  \sigma_{\max}
  &=
  \max \biggl( \frac{\rho_{\max}^2}{p_{\min}}, \ \frac{1}{\gamma - 1} \Bigl( \frac{\gamma}{\rho_{\min}} + \frac{\rho_{\max}}{p_{\min}^2} \Bigr) \biggr)
  \ \ge\ \lambda_{\max}(S),
  \\
  \sigma_{\min}
  &=
  \min \biggl( \frac{\rho_{\min}^2}{p_{\max}}, \ \frac{1}{\gamma p_{\max}^2 / \rho_{\min} + \rho_{\max}} \biggr)
  \ \le\ \lambda_{\min}(S),
\end{aligned}
\end{equation}
where the second expression in $\sigma_{\min}$ uses
$\frac{1}{\gamma - 1} \lambda_-(M_2) \ge \frac{1}{\gamma - 1} \frac{(\gamma - 1)/p^2}{\gamma / \rho + \rho / p^2} = \frac{1}{\gamma p^2 / \rho + \rho}$ and is deliberately independent of $p_{\min}$.
Combining \eqref{eq:rayleigh}--\eqref{eq:explicit-S-bounds}, every $u \in K$ satisfies $\lambda_{\min}(U''(u)) \ge \sigma_{\min}/\Gamma_{\mathrm{up}}$ and $\lambda_{\max}(U''(u)) \le \sigma_{\max} \Lambda_{\mathrm{up}}$, so the explicit constants
\begin{equation}
\label{eq:explicit-hessian-bounds}
  c_U = \frac{\sigma_{\min}}{\Gamma_{\mathrm{up}}} > 0,
  \qquad
  C_U = \sigma_{\max} \Lambda_{\mathrm{up}} < \infty
\end{equation}
are admissible in Assumption~\ref{asm:smooth-solution}: they satisfy $c_U \operatorname{I} \le U''(u) \le C_U \operatorname{I}$ for all $u \in K$.

\begin{remark}[Behavior near vacuum]
\label{rem:vacuum-scaling}
  The bounds \eqref{eq:explicit-hessian-bounds} are sufficient but not sharp.
  Since $\sigma_{\min}$ is independent of $p_{\min}$, the displayed $c_U$ does not deteriorate as $p_{\min} \to 0$;
  on a set of states with bounded densities, velocities, and pressures, the true smallest eigenvalue of $U''$ degenerates only towards the density vacuum $\rho \to 0$.
  Boundedness matters here:
  uniform coercivity is also lost along unbounded states, e.g., $\lambda_{\min}(U'') \to 0$ for $\rho = 1$, $v = 0$, and $p \to \infty$.
  In terms of $\rho_{\min}$, the displayed $c_U$ behaves like $\rho_{\min}^2$, which is not sharp:
  at $v = 0$, the smallest eigenvalue of $U''$ scales like $(\gamma - 1)^2 \rho / (\gamma p^2)$ as $\rho \to 0$, i.e., linearly in $\rho$.
  The displayed $C_U$ grows without bound as $\rho_{\min} \to 0$ or $p_{\min} \to 0$.
  For Theorem~\ref{thm:convergence}, only the finiteness and positivity of the constants on a fixed $K$ away from vacuum are used.
\end{remark}

\subsection{Regularity of the logarithmic mean and the flux constant}
\label{app:logmean}

The logarithmic mean \eqref{eq:logmean} of positive arguments satisfies the integral representation
\begin{equation}
\label{eq:inverse-logmean}
  \frac{1}{\logmean{a}}
  =
  \int_0^1 \frac{\dif t}{t \, a_L + (1 - t) \, a_R},
\end{equation}
as follows from evaluating the elementary integral;
equivalently, the affine substitution $x = t \, a_L + (1 - t) \, a_R$ recovers the logarithmic mean as the $t = -1$ member of the integral family of means of Chen \cite{chen2005means}.
Differentiating under the integral sign yields, for $a_L, a_R > 0$ and all $j, k \ge 0$,
\begin{equation}
\label{eq:logmean-derivative}
  \partial_{a_L}^j \partial_{a_R}^k \frac{1}{\logmean{a}}
  =
  (-1)^{j + k} \, (j + k)! \int_0^1 \frac{t^j (1 - t)^k}{( t a_L + (1 - t) a_R )^{j + k + 1}} \dif t,
\end{equation}
where the sign comes from the $j + k$ differentiations of the denominator and the remaining integral is positive.
Taking absolute values, we obtain for $a_L, a_R \in [a_{\min}, a_{\max}]$ with $a_{\min} > 0$
\begin{equation}
\label{eq:logmean-derivative-bounds}
  \abs[2]{\partial_{a_L}^j \partial_{a_R}^k \frac{1}{\logmean{a}}}
  =
  (j + k)! \int_0^1 \frac{t^j (1 - t)^k}{( t a_L + (1 - t) a_R )^{j + k + 1}} \dif t
  \le
  \frac{j! \, k!}{(j + k + 1) \, a_{\min}^{j + k + 1}},
\end{equation}
where the last step uses the value $j! \, k! / (j + k + 1)! = \int_0^1 t^j (1 - t)^k \dif t$.
The bound \eqref{eq:logmean-derivative-bounds} on the magnitude of the derivative is sharp; equality holds at $a_L = a_R = a_{\min}$.
In particular, $1 / \logmean{a} \ge 1 / a_{\max} > 0$, so $\logmean{a}$ is real-analytic on $(0, \infty)^2$ with derivative bounds up to any fixed order depending only on $a_{\min}$ and $a_{\max}$.

On $K \times K$, the fluxes \eqref{eq:flux-ranocha} are compositions of
arithmetic and product means (polynomial),
logarithmic means of $\rho \in [\rho_{\min}, \rho_{\max}]$ and of $\rho / p \in [\rho_{\min} / p_{\max}, \rho_{\max} / p_{\min}]$,
reciprocals of quantities bounded away from zero,
and the smooth map from conserved to primitive variables.
Combining \eqref{eq:logmean-derivative-bounds} with the chain and product rules therefore yields an explicit upper bound for the flux constant $C_S$ of \eqref{eq:remainder-bound} that is polynomial in
\begin{equation}
  \frac{1}{\rho_{\min}},
  \quad
  \frac{1}{p_{\min}},
  \quad
  \rho_{\max},
  \quad
  p_{\max},
  \quad
  V
\end{equation}
for fixed $\gamma > 1$.
This upper bound grows without bound as $\rho_{\min} \to 0$ or $p_{\min} \to 0$.
For Corollary~\ref{cor:euler}, only the finiteness of $C_S$ on a fixed $K$ away from vacuum is used.

\section{Encapsulated curvilinear operators}
\label{app:encapsulated}

This appendix verifies the operator assumptions of Section~\ref{sec:setting} for global SBP operators on smooth curvilinear periodic grids, constructed by the encapsulation approach of {\AA}lund and Nordstr{\"o}m \cite{alund2019encapsulated}:
the grid transformation is reinterpreted as a transformation of the operators, yielding matrices that approximate the Cartesian derivatives directly on the curved grid.
We treat two and three space dimensions at once.
Section~\ref{app:encapsulated-abstract} builds the operators from abstract discrete metric terms and establishes the structural and accuracy assumptions once, given three properties of those metrics (Definition~\ref{def:admissible-metrics});
Section~\ref{app:encapsulated-construction} then constructs metric terms with these properties from the conservative curl form, in both two and three space dimensions.

The operators themselves and their SBP property are due to \cite[Definition~3 and Proposition~1]{alund2019encapsulated}.
On a periodic domain, the boundary terms of that SBP property are absent, so it reduces to the skew-symmetry required in Assumption~\ref{asm:sbp}~\ref{item:sbp-skew}.
The remaining requirements of Section~\ref{sec:setting} are not addressed in \cite{alund2019encapsulated} and are established here:
the row sums vanish exactly, which is the discrete metric identity of Remark~\ref{rem:gcl} and the reason why the discrete entropy balance is free of spurious production;
the assembled operators satisfy the pointwise accuracy estimate \eqref{eq:accuracy} of order $p$, uniformly in the node index;
and the norm and locality conditions of Assumptions~\ref{asm:sbp}~\ref{item:sbp-norm} and \ref{asm:bandwidth} hold.

\subsection{Encapsulated operators from discrete metric terms}
\label{app:encapsulated-abstract}

Fix the space dimension $d$.
Let $(\hat{M}_{\mathrm{1D}}, \hat{Q}_{\mathrm{1D}})$ be a one-dimensional periodic diagonal-norm SBP operator of order $p$ with mesh size $\hat{h}$, satisfying Assumptions~\ref{asm:sbp} and \ref{asm:bandwidth} with $d = 1$, and form the tensor-product reference operators on the reference torus with coordinates $\xi^1, \dots, \xi^d$,
\begin{equation}
\label{eq:reference-operators}
  \hat{Q}_l
  =
  \hat{M}_{\mathrm{1D}}^{\otimes (l-1)} \otimes \hat{Q}_{\mathrm{1D}} \otimes \hat{M}_{\mathrm{1D}}^{\otimes (d-l)},
  \qquad
  \hat{M} = \hat{M}_{\mathrm{1D}}^{\otimes d},
  \qquad
  \hat{D}_l = \hat{M}^{-1} \hat{Q}_l .
\end{equation}
The $\hat{Q}_l$ are skew-symmetric and annihilate constants ($\hat{Q}_l \vec{1} = \vec{0}$), and the $\hat{D}_l$ inherit the accuracy estimate \eqref{eq:accuracy} of order $p$ in the reference variables, for smooth periodic functions on the reference torus.
Since $\hat{D}_l$ acts as the one-dimensional derivative matrix in the $l$-th tensor factor and as the identity in all others, the derivative matrices commute pairwise, $\hat{D}_l \hat{D}_{l'} = \hat{D}_{l'} \hat{D}_l$.
Let the physical grid be the image of the reference grid under an orientation-preserving diffeomorphism $\Phi$ of the torus homotopic to the identity --- equivalently, its lift differs from the identity by a periodic function --- with Jacobian determinant $J \ge J_{\min} > 0$, and decompose the coordinates into their linear lift and a periodic part,
\begin{equation}
\label{eq:coordinate-lift}
  x_m = \xi^m + \tilde{x}_m,
  \qquad
  \tilde{x}_m \text{ periodic}.
\end{equation}
The analytic contravariant metric terms $J a^l_m = J \, \partial \xi^l / \partial x_m$ provide the chain rule $\partial_{x_m} = J^{-1} \sum_l (J a^l_m) \, \partial_{\xi^l}$ and satisfy the metric identities $\sum_l \partial_{\xi^l}(J a^l_m) = 0$ \cite[Section~6.2]{kopriva2009implementing}.

Given grid functions $\amdisc{l}{m}$ ($l, m \in \{1, \dots, d\}$) and $\Jdisc$ --- the discrete counterparts of the analytic $J a^l_m$ and $J$, constructed in Section~\ref{app:encapsulated-construction} --- the \emph{encapsulated operators} of \cite[Definition~3]{alund2019encapsulated} are
\begin{equation}
\label{eq:encapsulated-Q}
  Q_m
  =
  \frac{1}{2} \sum_{l=1}^{d} \Bigl( \diag(\amdisc{l}{m}) \, \hat{Q}_l + \hat{Q}_l \, \diag(\amdisc{l}{m}) \Bigr),
  \qquad
  M = \diag(\Jdisc) \hat{M},
  \qquad
  D_m = M^{-1} Q_m .
\end{equation}
Entrywise, $(Q_m)_{ij} = \sum_l (\hat{Q}_l)_{ij} \, \mean{\amdisc{l}{m}}_{ij}$ with the arithmetic mean $\mean{\vec{a}}_{ij} = \tfrac12 (a_i + a_j)$, so the encapsulated operator is exactly the metric-averaged flux differencing operator used in practice (Remark~\ref{rem:metric-averaged}).

\begin{definition}[Admissible discrete metric terms]
\label{def:admissible-metrics}
  The grid functions $\amdisc{l}{m}$ and $\Jdisc$ are \emph{admissible of order $p$} if, with analytic metric terms $J a^l_m \in C^{p+1}$,
  \begin{enumerate}[label=(M\arabic*)]
    \item
    \label{item:metric-identity}
    $\sum_{l} \hat{Q}_l \, \amdisc{l}{m} = \vec{0}$ for each $m$ (discrete metric identities);
    \item
    \label{item:metric-jacobian}
    $\Jdiscn_i = J(\xi_i) + \O(\hat{h}^p)$ uniformly in $i$, and $\Jdiscn_i \ge J_{\min}/2$;
    \item
    \label{item:metric-accuracy}
    $\amdisc{l}{m} = \vec{J a}^l_m + \vec{E}^{l}_m$, where the metric errors $\vec{E}^{l}_m$ obey the \emph{uniform} bounds
    \begin{equation}
    \label{eq:metric-accuracy-uniform}
      \|\vec{E}^{l}_m\|_{\ell^\infty} \le C_E \, \hat{h}^p,
      \qquad
      \norm[1]{\hat{D}_l ( \vec{E}^{l}_m \odot \vec{G} )}_{\ell^\infty} \le C_E \, \hat{h}^p \, \|g\|_{C^1}
    \end{equation}
    for every $G = g \circ \Phi$, $g \in C^1$, on the physical torus,
    with a constant $C_E$ independent of $\hat{h}$, of the node index, and of the test function $g$.
  \end{enumerate}
\end{definition}

\begin{lemma}[Structural properties]
\label{lem:encapsulated-structure}
  For arbitrary grid functions $\amdisc{l}{m}$, the operators \eqref{eq:encapsulated-Q} satisfy $Q_m + Q_m^T = 0$.
  If in addition \ref{item:metric-identity} holds, then $Q_m \vec{1} = \vec{0}$ exactly.
\end{lemma}

\begin{proof}
  For any diagonal matrix $A$ and skew-symmetric $\hat{Q}$, the combination $A \hat{Q} + \hat{Q} A$ is skew-symmetric;
  summing over $l$ gives $Q_m + Q_m^T = 0$, the periodic case of \cite[Proposition~1]{alund2019encapsulated}.
  Since $\hat{Q}_l \vec{1} = \vec{0}$, the row sums are $Q_m \vec{1} = \tfrac12 \sum_l \hat{Q}_l \amdisc{l}{m}$, which vanish under \ref{item:metric-identity}.
\end{proof}

The accuracy proof rests on the following bound, which reuses the fixed-local-stencil hypotheses.

\begin{lemma}[Discrete derivatives of stencilwise-Lipschitz grid functions]
\label{lem:lipschitz-bound}
  Let $D = M^{-1} Q$ satisfy $Q \vec{1} = \vec{0}$, the locality condition $Q_{ij} \ne 0 \Rightarrow \distper(x_i, x_j) \le \omega h$, the row-sum bound $\sum_j |Q_{ij}| \le C_{\mathrm{row}} h^{d-1}$, and $m_i \ge c_M h^d$.
  Then, for every grid function $\vec{g}$ obeying the \emph{stencilwise} Lipschitz bound
  $|g_j - g_i| \le \Lambda \, \distper(x_i, x_j)$ whenever $Q_{ij} \ne 0$,
  \begin{equation}
    \abs{(D \vec{g})_i}
    \le
    \frac{1}{c_M h^d} \, C_{\mathrm{row}} h^{d-1} \, \omega h \, \Lambda
    =
    \frac{C_{\mathrm{row}} \, \omega}{c_M} \Lambda.
  \end{equation}
  The hypothesis is required only across pairs coupled by $Q$;
  in particular, for a reference operator $\hat{D}_l$, whose stencil couples only nodes on a common $\xi^l$-line, it suffices that $\vec{g}$ be the restriction of a function Lipschitz in $\xi^l$ with constant $\Lambda$, and no globally Lipschitz extension is needed.
\end{lemma}

\begin{proof}
  By $Q \vec{1} = \vec{0}$, $(D \vec{g})_i = m_i^{-1} \sum_j Q_{ij} (g_j - g_i)$;
  the stencilwise bound and locality give $|g_j - g_i| \le \Lambda \, \omega h$ for every $j$ with $Q_{ij} \ne 0$, and the row-sum and mass bounds give the stated constant, with all powers of $h$ cancelling exactly.
\end{proof}

\begin{proposition}[Accuracy of the encapsulated operators]
\label{prop:encapsulated-accuracy}
  Let the discrete metric terms be admissible of order $p$ (Definition~\ref{def:admissible-metrics}).
  Then the operators $D_m = M^{-1} Q_m$ satisfy the accuracy estimate \eqref{eq:accuracy} of order $p$, with a constant depending on the reference operators, $J_{\min}$, the mapping $\Phi$, the norms $\|J a^l_m\|_{C^{p+1}}$, and the uniform constants implicit in the estimates \ref{item:metric-jacobian} and \ref{item:metric-accuracy} of Definition~\ref{def:admissible-metrics}.
\end{proposition}

\begin{proof}
  Let $g$ be smooth on the physical torus and $G = g \circ \Phi$ its pullback, with $\|G\|_{C^{p+1}} \le C_\Phi \|g\|_{C^{p+1}}$ by the chain rule, where the pullback constant $C_\Phi$ depends on the derivatives of $\Phi$ up to order $p+1$.
  \emph{Step 1: mass.}
  By \ref{item:metric-jacobian}, $\Jdiscn_i = J(\xi_i) + \O(\hat{h}^p)$ and $\Jdiscn_i \ge J_{\min}/2$.
  \emph{Step 2: diagonal commutation.}
  Since $\hat{M}$, $\diag(\Jdisc)$, and the metric diagonals commute,
  \begin{equation}
    (D_m \vec{G})_i
    =
    \frac{1}{2 \, \Jdiscn_i} \sum_{l} \Bigl[ \amdisci{l}{m}{i} \, (\hat{D}_l \vec{G})_i + \bigl( \hat{D}_l ( \amdisc{l}{m} \odot \vec{G} ) \bigr)_i \Bigr] .
  \end{equation}
  \emph{Step 3: non-product terms.}
  By the reference accuracy, $(\hat{D}_l \vec{G})_i = \partial_{\xi^l} G(\xi_i) + \O(\hat{h}^p)$, and $\amdisci{l}{m}{i} = J a^l_m(\xi_i) + \O(\hat{h}^p)$ by \ref{item:metric-accuracy}.

  \emph{Step 4: product terms.}
  Split $\amdisc{l}{m} \odot \vec{G} = \vec{J a}^l_m \odot \vec{G} + \vec{E}^{l}_m \odot \vec{G}$.
  The first part is the nodal restriction of the $C^{p+1}$ function $(J a^l_m)\, g \circ \Phi$, so the reference accuracy applies;
  the second satisfies $\hat{D}_l ( \vec{E}^{l}_m \odot \vec{G} ) = \O(\hat{h}^p)$ by \ref{item:metric-accuracy}.
  Hence, $\bigl( \hat{D}_l ( \amdisc{l}{m} \odot \vec{G} ) \bigr)_i = \partial_{\xi^l} \bigl( (J a^l_m) G \bigr)(\xi_i) + \O(\hat{h}^p)$.

  \emph{Step 5: consistency.}
  Collecting Steps 1--4,
  \begin{equation}
    (D_m \vec{G})_i
    =
    \frac{1}{2 J(\xi_i)} \sum_{l} \Bigl[ (J a^l_m) \, \partial_{\xi^l} G + \partial_{\xi^l} \bigl( (J a^l_m) \, G \bigr) \Bigr](\xi_i) + \O(\hat{h}^p) .
  \end{equation}
  Expanding the second product gives $\frac{1}{2 J} \bigl[ 2 \sum_l (J a^l_m) \partial_{\xi^l} G + G \sum_l \partial_{\xi^l}(J a^l_m) \bigr]$;
  the last sum vanishes by the continuous metric identities, and the first equals $\partial_{x_m} g$ by the chain rule, both evaluated at the node.
  Tracking the remainders, the $\O(\hat{h}^p)$ constant collects the reference truncation constant, the factor $1/J_{\min}$ from Step~1, the pullback constant $C_\Phi$, and the uniform constants of \ref{item:metric-jacobian} and \ref{item:metric-accuracy};
  the example $\amdisc{l}{m} = (1 + L \hat{h}^p) \delta_{lm} \vec{1}$, $\Jdisc = \vec{1}$ for the identity map --- admissible with an \ref{item:metric-accuracy}-constant $C_E$ proportional to $L$ and $D_m = (1 + L \hat{h}^p) \hat{D}_m$ --- shows this dependence is genuine.
  Since $\Phi$ is bi-Lipschitz, $\hat{h}$ and the physical mesh size $h$ are equivalent, which completes the proof.
\end{proof}

The remaining assumptions hold for any admissible metric terms.
The masses are $m_i = \Jdiscn_i \, \hat{m}_i$ with $\hat{m}_i \sim h^d$ and $\Jdiscn_i \in [J_{\min}/2, \, 2 J_{\max}]$ for small $h$, so Assumption~\ref{asm:sbp}~\ref{item:sbp-norm} holds.
Each row of $Q_m$ combines the $d$ reference line stencils with bounded metric factors, so the number of nonzeros per row is bounded and the entries are of size $\O(h^{d-1})$;
the physical locality follows from the reference locality through the bi-Lipschitz mapping, so Assumption~\ref{asm:bandwidth} holds.
Together with Lemma~\ref{lem:encapsulated-structure}, all hypotheses of Theorem~\ref{thm:convergence} concerning the operators are verified once admissible discrete metric terms are supplied.

\subsection{Construction of admissible discrete metric terms}
\label{app:encapsulated-construction}

It remains to exhibit discrete metric terms satisfying Definition~\ref{def:admissible-metrics}.
We use the conservative curl form of Thomas and Lombard \cite{thomas1979geometric}, computed with the reference operators.
As observed by Kopriva \cite[Remark~4]{kopriva2006metric}, this is what makes the discrete metric identities hold exactly.
Analytically evaluated or naively discretized metric terms do not in general satisfy the discrete metric identities, may leave nonzero row sums, and give the discrete entropy balance a spurious production.

\emph{Differentiating the coordinates.}
The reference operators act on grid functions on the reference torus, and the accuracy estimate \eqref{eq:accuracy} that they inherit is available for \emph{periodic} arguments only.
The coordinates \eqref{eq:coordinate-lift} are not periodic, and the obstruction is confined to the diagonal direction $l = m$:
applying $\hat{D}_m$ to the sampled values $\vec{x}_m$ is not a consistent approximation of $\partial_{\xi^m} x_m$.
(For $l \ne m$ no obstruction arises, since the linear lift $\xi^m$ is constant along $\xi^l$-lines, whence $\hat{D}_l \vec{\xi^m} = \vec{0}$ and $\hat{D}_l \vec{x}_m = \hat{D}_l \vec{\tilde{x}}_m$ approximates $\partial_{\xi^l} \tilde{x}_m = \partial_{\xi^l} x_m$ to order $p$.)
Indeed, skew-symmetry and $\hat{Q}_l \vec{1} = \vec{0}$ give
\begin{equation}
\label{eq:zero-weighted-mean}
  \vec{1}^T \hat{M} \hat{D}_l \vec{v}
  =
  \vec{1}^T \hat{Q}_l \vec{v}
  =
  - (\hat{Q}_l \vec{1})^T \vec{v}
  =
  0
  \qquad \text{for every grid function } \vec{v},
\end{equation}
so every output of $\hat{D}_l$ has vanishing $\hat{M}$-weighted mean.
The sampled exact derivative $\partial_{\xi^m} x_m = 1 + \partial_{\xi^m} \tilde{x}_m$, on the other hand, has $\hat{M}$-weighted mean $\vec{1}^T \hat{M} \vec{1} + \O(\hat{h}^p)$:
the quadrature $\vec{1}^T \hat{M} \, \vec{\partial_{\xi^m} \tilde{x}_m}$ approximates $\int \partial_{\xi^m} \tilde{x}_m = 0$ only to order $p$ and need not vanish.
Since $\vec{1}^T \hat{M} \vec{1}$ is bounded away from zero, this weighted mean is positive for small $\hat{h}$, which is incompatible with the exactly vanishing mean of $\hat{D}_m \vec{x}_m$.
We therefore differentiate the linear lift in \eqref{eq:coordinate-lift} exactly and discretize only its periodic part, writing
\begin{equation}
\label{eq:lifted-derivative}
  \vec{g}^{l}_m
  \coloneqq
  \delta_{lm} \vec{1} + \hat{D}_l \vec{\tilde{x}}_m
  \qquad (\text{the discrete counterpart of } \partial_{\xi^l} x_m) .
\end{equation}
This is a periodic grid function, and since $\tilde{x}_m$ is periodic and as smooth as $\Phi$, the reference accuracy is applicable to it and yields, whenever $\Phi \in C^{p+1}$,
\begin{equation}
\label{eq:lifted-derivative-accuracy}
  \vec{g}^{l}_m
  =
  \vec{\partial_{\xi^l} x_m} + \O(\hat{h}^p) .
\end{equation}
In practice, \eqref{eq:lifted-derivative} amounts to storing the mapping through its periodic part $\tilde{x}_m$, which is what a periodic curved grid provides anyway;
on a bounded domain, where $x_m$ itself is a legitimate argument of the operator, the lift is absent and $\vec{g}^{l}_m = \hat{D}_l \vec{x}_m$.

\emph{The curl form in lifted variables.}
For $d = 3$, let $(l, r, s)$ and $(m, n, q)$ denote cyclic permutations of $(1, 2, 3)$.
The analytic contravariant metric terms are cofactors of the Jacobian and admit the curl form \cite[Theorem~1]{kopriva2006metric}
\begin{equation}
\label{eq:curl-form}
  J a^l_m
  =
  \partial_{\xi^r} \bigl( x_n \, \partial_{\xi^s} x_q \bigr)
  - \partial_{\xi^s} \bigl( x_n \, \partial_{\xi^r} x_q \bigr),
\end{equation}
the second derivatives of $x_q$ cancelling by the product rule.
The factor $x_n$ in \eqref{eq:curl-form} is again non-periodic, but its lift can be removed analytically:
inserting $x_n = \xi^n + \tilde{x}_n$, the two terms containing $\xi^n \, \partial_{\xi^r} \partial_{\xi^s} x_q$ cancel against each other, and what is left is
\begin{equation}
\label{eq:curl-form-lifted}
  J a^l_m
  =
  \delta_{rn} \, \partial_{\xi^s} x_q
  - \delta_{sn} \, \partial_{\xi^r} x_q
  + \partial_{\xi^r} \bigl( \tilde{x}_n \, \partial_{\xi^s} x_q \bigr)
  - \partial_{\xi^s} \bigl( \tilde{x}_n \, \partial_{\xi^r} x_q \bigr),
\end{equation}
where every differentiated factor is now periodic, because $\tilde{x}_n$ is periodic by \eqref{eq:coordinate-lift} and $\partial_{\xi^a} x_q = \delta_{aq} + \partial_{\xi^a} \tilde{x}_q$ is a derivative of $\Phi$.
Discretizing \eqref{eq:curl-form-lifted} with the reference operators and \eqref{eq:lifted-derivative} defines the discrete metric terms and the discrete Jacobian determinant
\begin{equation}
\label{eq:discrete-metrics}
  \amdisc{l}{m}
  =
  \delta_{rn} \, \vec{g}^{s}_q
  - \delta_{sn} \, \vec{g}^{r}_q
  + \hat{D}_r \bigl( \vec{\tilde{x}}_n \odot \vec{g}^{s}_q \bigr)
  - \hat{D}_s \bigl( \vec{\tilde{x}}_n \odot \vec{g}^{r}_q \bigr),
  \qquad
  \Jdisc
  =
  \frac{1}{d} \sum_{m=1}^{d} \sum_{l=1}^{d} \amdisc{l}{m} \odot \vec{g}^{l}_m .
\end{equation}
For $d = 2$ the cofactors are the single derivatives $J a^1_1 = \partial_{\xi^2} x_2$, $J a^2_1 = -\partial_{\xi^1} x_2$, $J a^1_2 = -\partial_{\xi^2} x_1$, $J a^2_2 = \partial_{\xi^1} x_1$, with discrete counterparts $\amdisc{1}{1} = \vec{g}^{2}_2$, $\amdisc{2}{1} = -\vec{g}^{1}_2$, $\amdisc{1}{2} = -\vec{g}^{2}_1$, $\amdisc{2}{2} = \vec{g}^{1}_1$, and $\Jdisc = \vec{g}^{1}_1 \odot \vec{g}^{2}_2 - \vec{g}^{2}_1 \odot \vec{g}^{1}_2$, which is the case $d = 2$ of the second formula in \eqref{eq:discrete-metrics};
up to the lift \eqref{eq:lifted-derivative}, this is the original construction of {\AA}lund and Nordstr{\"o}m \cite{alund2019encapsulated}.
As a consistency check, for the identity mapping $\tilde{x}_m = 0$ one gets $\vec{g}^{l}_m = \delta_{lm} \vec{1}$, hence $\amdisc{l}{m} = \delta_{lm} \vec{1}$ and $\Jdisc = \vec{1}$ exactly, in every space dimension.

\begin{lemma}[Admissibility of the curl-form metric terms]
\label{lem:encapsulated-construction}
  Let $\Phi$ be a $C^{p+d}$ diffeomorphism of the torus, with $J \ge J_{\min} > 0$.
  Then the discrete metric terms \eqref{eq:discrete-metrics} are admissible of order $p$ in the sense of Definition~\ref{def:admissible-metrics}.
\end{lemma}

\begin{proof}
  \emph{Identities \ref{item:metric-identity}.}
  With $\hat{Q}_l = \hat{M} \hat{D}_l$, we have $\sum_l \hat{Q}_l \amdisc{l}{m} = \hat{M} \sum_l \hat{D}_l \amdisc{l}{m}$, so it suffices to show $\sum_l \hat{D}_l \amdisc{l}{m} = \vec{0}$.
  Since $\hat{D}_l \vec{1} = \vec{0}$, the lifted derivatives \eqref{eq:lifted-derivative} satisfy
  \begin{equation}
  \label{eq:lifted-derivative-commutator}
    \hat{D}_l \vec{g}^{a}_q = \hat{D}_l \hat{D}_a \vec{\tilde{x}}_q,
  \end{equation}
  i.e., under a further reference derivative the lift drops out.
  Let $d = 3$ and split $\amdisc{l}{m}$ from \eqref{eq:discrete-metrics} into its two $\delta$-terms and its two product terms.
  For the product terms, abbreviate $\vec{b}_a = \vec{\tilde{x}}_n \odot \vec{g}^{a}_q$;
  their contribution is
  \begin{equation}
    \sum_l \hat{D}_l \bigl( \hat{D}_r \vec{b}_s - \hat{D}_s \vec{b}_r \bigr)
    =
    (\hat{D}_1 \hat{D}_2 - \hat{D}_2 \hat{D}_1) \vec{b}_3
    + (\hat{D}_2 \hat{D}_3 - \hat{D}_3 \hat{D}_2) \vec{b}_1
    + (\hat{D}_3 \hat{D}_1 - \hat{D}_1 \hat{D}_3) \vec{b}_2
    =
    \vec{0},
  \end{equation}
  since the reference derivative matrices commute;
  this is the discrete metric identity of the conservative curl form \cite[Theorems~1 and 5]{kopriva2006metric}.
  For the $\delta$-terms, only two summands survive:
  $\delta_{rn} = 1$ forces $l = n - 1$ and then $s = n + 1$, while $\delta_{sn} = 1$ forces $l = n + 1$ and then $r = n - 1$ (indices modulo $3$).
  With \eqref{eq:lifted-derivative-commutator} their contribution is therefore
  \begin{equation}
    \hat{D}_{n-1} \vec{g}^{n+1}_q - \hat{D}_{n+1} \vec{g}^{n-1}_q
    =
    (\hat{D}_{n-1} \hat{D}_{n+1} - \hat{D}_{n+1} \hat{D}_{n-1}) \vec{\tilde{x}}_q
    =
    \vec{0} .
  \end{equation}
  For $d = 2$, \eqref{eq:lifted-derivative-commutator} reduces both identities to the single commutator $\sum_l \hat{D}_l \amdisc{l}{m} = \pm (\hat{D}_1 \hat{D}_2 - \hat{D}_2 \hat{D}_1) \vec{\tilde{x}}_{3-m} = \vec{0}$, cf.\ \cite[Theorem~3]{kopriva2006metric}.

  \emph{Accuracy \ref{item:metric-accuracy}.}
  After removal of the coordinate lifts, the exact continuous factors in \eqref{eq:discrete-metrics} are periodic, so the reference accuracy applies to them directly.
  The products that additionally contain a discrete derivative $\vec{g}^{s}_q$ are not restrictions of single smooth periodic functions --- at duplicated nodes two copies of a coordinate may carry slightly different discrete-derivative values --- and are instead decomposed below into a smooth periodic part, treated by the reference accuracy, and a truncation-error part, controlled by Lemma~\ref{lem:lipschitz-bound}.
  We apply Lemma~\ref{lem:lipschitz-bound} repeatedly with outer operator $\hat{D}_a$;
  its stencilwise hypothesis is met without a global extension, because every remainder to which it is applied below is built from reference operators $\hat{D}_b$ with $b \ne a$, which act as the identity in the $a$-th tensor factor.
  Duplicated nodes sharing a $\xi^a$-coordinate and coupled by the $\hat{D}_a$-stencil --- as may occur for assembled operators with duplicated interface nodes --- therefore receive identical values, so each linewise extension below is single-valued on its $\xi^a$-line, as Lemma~\ref{lem:lipschitz-bound} requires.
  The $\delta$-terms of \eqref{eq:discrete-metrics} approximate those of \eqref{eq:curl-form-lifted} to order $p$ by \eqref{eq:lifted-derivative-accuracy}.
  For a product term, split $\vec{g}^{s}_q = \vec{\partial_{\xi^s} x_q} + \vec{\varepsilon}_s$ with $\vec{\varepsilon}_s = \O(\hat{h}^p)$, again by \eqref{eq:lifted-derivative-accuracy}.
  Then $\hat{D}_r ( \vec{\tilde{x}}_n \odot \vec{\partial_{\xi^s} x_q} )$ applies the reference operator to the restriction of the \emph{periodic} $C^{p+1}$ function $\tilde{x}_n \, \partial_{\xi^s} x_q$ and is accurate of order $p$, while $\hat{D}_r ( \vec{\tilde{x}}_n \odot \vec{\varepsilon}_s ) = \O(\hat{h}^p)$ by Lemma~\ref{lem:lipschitz-bound}:
  Extend $\vec{\varepsilon}_s$ in $\xi^r$ by applying the same $\xi^s$-line stencil to the smooth periodic part of the mapping at continuous $\xi^r$.
  Its $\xi^r$-derivative is the $\xi^s$-truncation error of $\partial_{\xi^r} \tilde{x}_q \in C^{p+1}$, so $\vec{\tilde{x}}_n \odot \vec{\varepsilon}_s$ is the restriction of a function that is Lipschitz in $\xi^r$ with constant $\O(\hat{h}^p)$.
  Hence $\amdisc{l}{m} = \vec{J a}^l_m + \vec{E}^{l}_m$ with $\|\vec{E}^{l}_m\|_{\ell^\infty} = \O(\hat{h}^p)$, the constant depending only on the reference operator and on bounded derivatives of $\Phi$, not on the node index.
  Applying the same extension to $\vec{E}^{l}_m$ --- whose $\xi^l$-derivative is, by the product rule, a sum of truncation errors of the type above applied to the first derivatives of $\tilde{x}$, and which is built from the transverse operators $\hat{D}_r, \hat{D}_s$ with $r, s \ne l$ --- shows that $\vec{E}^{l}_m$ is the restriction of a function Lipschitz in $\xi^l$ with constant $\O(\hat{h}^p)$.
  For $G = g \circ \Phi$ with $g \in C^1$, the product $\vec{E}^{l}_m \odot \vec{G}$ is then the restriction of a function Lipschitz in $\xi^l$ with constant $\O(\hat{h}^p) \|G\|_{C^1} \le C_E \hat{h}^p \|g\|_{C^1}$, using $\|G\|_{C^1} \le C_\Phi \|g\|_{C^1}$;
  by the single-valuedness noted above, Lemma~\ref{lem:lipschitz-bound} applies and yields the uniform estimate $\|\hat{D}_l ( \vec{E}^{l}_m \odot \vec{G} )\|_{\ell^\infty} \le C_E \hat{h}^p \|g\|_{C^1}$ of \eqref{eq:metric-accuracy-uniform}, with $C_E$ independent of $g$, $\hat{h}$, and the node index.
  This is where $\Phi \in C^{p+2}$ for $d = 2$, resp.\ $C^{p+3}$ for $d = 3$, enters:
  the metric terms carry $d - 1$ derivatives of $\Phi$, and the extension costs one more.

  \emph{Jacobian \ref{item:metric-jacobian}.}
  By \ref{item:metric-accuracy} and \eqref{eq:lifted-derivative-accuracy} each factor in $\Jdisc$ is $\O(\hat{h}^p)$-accurate, and the analytic counterpart $\frac{1}{d} \sum_m \sum_l (J a^l_m) \, \partial_{\xi^l} x_m = \frac{1}{d} \sum_m J = J$ by the chain rule, so $\Jdiscn_i = J(\xi_i) + \O(\hat{h}^p)$ and $\Jdiscn_i \ge J_{\min}/2$ for small $\hat{h}$.
\end{proof}

\begin{remark}[Other curl forms]
\label{rem:invariant-curl}
  The conservative curl form is not unique:
  interchanging $x_n$ and $x_q$ in \eqref{eq:curl-form} and changing the sign gives the second variant of \cite[Section~4]{kopriva2006metric}, and their average is the coordinate-invariant curl form of \cite[Equation~(37)]{kopriva2006metric}, often preferred in practice.
  Both alternatives may replace \eqref{eq:discrete-metrics}:
  removing the lift of the undifferentiated coordinate as in \eqref{eq:curl-form-lifted} works verbatim --- the terms carrying $\xi$ times a second derivative again cancel in pairs --- and afterwards each variant is, as before, a sum of $\delta$-terms of the form \eqref{eq:lifted-derivative} and of differences of reference derivatives of products of periodic grid functions.
  Hence the same commutators cancel and the accuracy argument is unchanged.
\end{remark}

\begin{remark}[Metric-averaged curved flux differencing]
\label{rem:metric-averaged}
  The encapsulated operators \eqref{eq:encapsulated-Q} are exactly the assembled operators of the metric-averaged contravariant flux differencing schemes used in practice on curved meshes \cite{gassner2016split,crean2018entropy,ranocha2023efficient}.
  Such a scheme reads
  \begin{equation}
    m_i \, \frac{\dif}{\dif t} \unum_i + \sum_{l} \sum_j 2 (\hat{Q}_l)_{ij} \sum_{m} \mean{\amdisc{l}{m}}_{ij} \, \fnum_m(\unum_i, \unum_j) = 0,
  \end{equation}
  with the metric averages $\mean{\amdisc{l}{m}}_{ij}$;
  its assembled directional operators are $(Q_m)_{ij} = \sum_l (\hat{Q}_l)_{ij} \mean{\amdisc{l}{m}}_{ij}$, i.e., \eqref{eq:encapsulated-Q}, which casts the scheme exactly into the form \eqref{eq:semidiscretization}.
  By Lemma~\ref{lem:encapsulated-structure} and Proposition~\ref{prop:encapsulated-accuracy}, the scheme fits Theorem~\ref{thm:convergence} whenever its discrete metric terms are admissible (Definition~\ref{def:admissible-metrics}), and Lemma~\ref{lem:encapsulated-construction} exhibits one admissible choice.
  For a given implementation, the identities \ref{item:metric-identity} coincide with the discrete metric identities already needed for entropy conservation (Remark~\ref{rem:gcl}), so what remains to be checked is the accuracy of its particular metric and Jacobian terms --- conditions \ref{item:metric-accuracy} and \ref{item:metric-jacobian}.
  Condition \ref{item:metric-jacobian} must be verified separately unless, as for the construction \eqref{eq:discrete-metrics} (where Lemma~\ref{lem:encapsulated-construction} deduces it from \ref{item:metric-accuracy} through the trace formula defining $\Jdisc$), the implementation's definition of $\Jdisc$ makes it follow from \ref{item:metric-accuracy}.
\end{remark}

\section*{Tool disclosure}

Claude Opus~4.8 and Codex GPT-5.5 were used to assist in the preparation of this work.

\section*{Acknowledgments}

HR was supported by the Deutsche Forschungsgemeinschaft
(DFG, German Research Foundation, project number 528753982
as well as within the DFG priority program SPP~2410 with project number 526031774).

\printbibliography

\end{document}